\documentclass[10pt,reqno,final]{article}

\usepackage{graphicx}
\usepackage{amsmath,amsfonts,amssymb,amsthm,version}
\usepackage{mathrsfs,fancybox,pifont}
\usepackage{graphicx}
\usepackage{url,hyperref}
\usepackage[notcite,notref]{showkeys}
\usepackage{color}
\usepackage{subfigure,multirow}
\usepackage{epstopdf}
\usepackage{cases}
\usepackage{mathtools}
\usepackage{algorithm,algorithmic}
\usepackage{authblk}
\usepackage{lipsum}
\usepackage{float}
\allowdisplaybreaks

\def\XXint#1#2#3{{\setbox0=\hbox{$#1{#2#3}{\int}$ }
\vcenter{\hbox{$#2#3$ }}\kern-.6\wd0}}

\newtheorem{theorem}{Theorem}[section]
\newtheorem{lemma}[theorem]{Lemma}

\newtheorem{proposition}[theorem]{Proposition}
\newtheorem{corollary}[theorem]{Corollary}
\theoremstyle{definition}
\newtheorem{definition}[theorem]{Definition}

\theoremstyle{remark}
\newtheorem{remark}[theorem]{Remark}
\numberwithin{equation}{section}

\renewcommand{\L}{{\mathcal L}}
\newcommand{\X}{\mathcal X}
\newcommand{\Q}{\mathcal Q}

\begin{document}

\title {
Several functional capacities and Carleson type embeddings of   fractional Sobolev sapces on stratified Lie groups}

  \author{ Zhiyong Wang \ \ \  Pengtao Li \ \ \   Yu Liu\footnote{Corresponding author}} \date{}
    \maketitle



\maketitle

{  {\bf Abstract:} In this paper, we focus on the functional and
geometrical aspects of the fractional Sobolev capacity, the Besov
capacity and the Riesz capacity on stratified lie groups,
respectively. Firstly, we provide a new Carleson characterization of
the extension of fractional Sobolev spaces to
$L^{q}(\X\times\mathbb{R}_{+},\mu)$ with $q\in\mathbb{R}_{+}$ using
the fractional heat semigroup and the Caffarelli-Silvestre type
extension on stratified Lie groups $\X$. Secondly,  a
characterization of $\nu$ on $\X$ which ensures the continuity of
the fractional Sobolev space  belonging to $L^{q}(\X,\nu)$ is also
obtained via taking $t\rightarrow 0$. Finally, with the help of
inequalities related to the Besov capacity and its properties, we
also obtain a characterization of $\nu$ on $\X$ which ensures  the
continuity of the Besov type space belonging to $L^{q}(\X,\nu)$.

\vspace{0.2cm} {\bf Keywords:} {fractional Sobolev space, Besov
capacity, fractional Sobolev capacity, stratified Lie group.}

}

\tableofcontents

\section{Introduction}

Possessing a robust theoretical underpinning in physics, the concept
of capacities has been widely utilized across investigations in
analysis and partial differential equations. Specifically, the
capacities related to function spaces were introduced to
characterize the embedding or trace inequalities. For example, the
Besov-type capacity $cap(\cdot; \dot{\Lambda}^{p,q}_{\alpha})$ was
used to
 establish the embedding of homogeneous Besov spaces into the Lorentz spaces with respect to nonnegative Borel
 measures (see \cite{AX,  Ma, W1999,X2}). The embedding of Sobolev spaces via heat equations and  the $p$-variational capacity was attributed to
 Xiao \cite{X1, X2}. From then on, Zhai in \cite{zh-2} provided the Carleson characterization of embedding for fractional Sobolev spaces
 by factional capacities. The Carleson type characterization of embedding for  Lebesgue spaces was established by Chang and Xiao in \cite{cha}
 using $L^{p}$-capacities. Li and Zhai \cite{li-4} obtained the Carleson type characterization of the extension of fractional Sobolev spaces and
 Lebesgue spaces to $L^{q}(\mathbb{R}^{n+1}_{+},\mu)$ via space-time fractional equations.

 In recent decades, Sobolev classes on metric measure spaces have attracted a lot of attention. To
  overcome the disadvantages of the classical derivatives on metric measure spaces, the theory of
  weak upper gradients was developed to replace the theory of weak or distributional derivatives
  in the construction of Sobolev classes of functions, see \cite{He, Sh1,Sh2}. As generalizations of the
   Carleson type embeddings of function spaces via dissipative operators on the Euclidean space, the
   counterparts for the cases of metric measure spaces have been investigated by many researchers. On the metric
   measure space, Liu et al. \cite{liu}  established the Carleson type characterization of the extension of
    Newton-Sobolev spaces  $N^{1,p}(\mathbb M)$ to $L^{q}(\mathbb M\times(0,\infty),\mu)$, where $\mathbb{M}$
    satisfies some additional assumptions.  In \cite{hua}, Huang et al. obtained the Carleson characterization of the
    extension of Lebesgue spaces to $L^{q}(\mathbb M\times(0,\infty),\mu)$ via the fractional dissipative
    equations. For further information on this topic, we refer the reader to \cite{A2, cha,Co, DKX} and the references therein.

In the setting of $\mathbb R^{n}$, it is well-known that the fractional Sobolev spaces $W^{p,s}(\mathbb R^{n})$ have the following equivalent norms:
$$\|f\|_{W^{p,s}(\mathbb R^{n})}:=\|f\|_{L^{p}(\mathbb R^{n})}+\|(-\Delta)^{s/2}f\|_{L^{p}(\mathbb R^{n})}\sim\|f\|_{L^{p}(\mathbb R^{n})}+\sum^{n}_{i=1}\|\partial_{x_{i}}^{s}f\|_{L^{p}(\mathbb R^{n})},$$
where the symbols $(-\Delta)^{s/2}$ and $\partial_{x_{i}}^{s}$ denote the fractional Laplacian and the fractional derivatives
 defined via the Fourier transform, respectively. However, it should be mentioned that  Newton-Sobolev spaces  $N^{1,p}(\mathbb M)$ defined
  via the weak upper gradients are essentially a class of $1$-order Sobolev spaces. To establish the Carleson type embeddings
   of fractional Sobolev spaces on metric measure spaces, it is reasonable to assume that the target to be investigated should
   own a differential structure and fractional Laplacians defined via functional calculus.

Based on the above consideration, we focus on the stratified Lie group denoted by $\X$, which is  a typical metric space. Stratified Lie groups,
often referred to as Carnot groups, make prominent appearances across various areas of mathematics such as harmonic analysis,
 several complex variables, geometry, and topology, as well as within the scope of quantum physics. The simplest examples of stratified
 Lie groups are  Euclidean spaces $\mathbb{R}^{n}$,  Heisenberg groups $\mathbb{H}_{n}$ and  Heisenberg-type groups  introduced by Kaplan \cite{K80}.
  Additionally, stratified Lie groups have gained a lot of attention in recent years due to their rich structures and the potential for extending
  objects and results of classical analysis, which have been studied in contexts such as potential theory \cite{V90,F18}, subelliptic
  differential equations \cite{F75,F82,M23} and analysis \cite{Z23,C15,D09,S79}.

In \cite{miao},  in order to study  the Cauchy problem for a class
of nonlinear fractional power dissipative equations,  Miao, Yuan and
Zhang investigated the regularity estimate of the fractional heat
kernel related to the Laplacian $(-\Delta)$, which can be seen as
the fundamental solution to the following  homogeneous linear
fractional power dissipative equation:
\begin{equation*}\label{heat-1}
  \begin{cases}
 \partial_{t}u(x,t)+(-\Delta)^{\alpha}u(x,t)=0\ \ \  \forall\  (x,t)\in \mathbb{R}^{n}\times\mathbb R_{+};\\
  u(x,0)=f(x)\ \ \ \ \ \ \ \ \ \  \ \ \ \ \ \ \ \ \ \ \ \ \forall\ x\in\mathbb
  R^{n},
  \end{cases}
\end{equation*} where $\alpha\in(0,1)$ and $\mathbb{R}_{+}:=(0,\infty)$.
The solution to (\ref{heat-1}) can be written as
$$u(x,t):=e^{-t(-\Delta)^{\alpha}}f(x):=\int_{\mathbb R^{n}}e^{-t(-\Delta)^{\alpha}}(x,y)f(y)dy.$$
By an invariant derivative technique and the Fourier analysis method, Miao, Yuan and Zhang obtained the following
  estimate of the fractional heat kernel $e^{-t(-\Delta)^{\alpha}}(\cdot,\cdot)$:
\begin{equation*}\label{heat-2}
   e^{-t(-\Delta)^{\alpha}}(x,y)\sim\frac{t}{(t^{1/2\alpha}+|x-y|)^{n+2\alpha}}
\end{equation*}
holds for all $(t,x,y)\in\mathbb{R}_{+}\times\mathbb
R^{n}\times\mathbb R^{n},$ where the lower bound estimate can be
seen from \cite[Theorem 1]{CK}. Let $\L$ denote the
the corresponding sub-Laplacian on $\X$, i.e.,
\begin{equation*}\label{K}
\L:=-\sum_{j=1}^{n_{1}}\mathbf{X}_{j}^{2},
\end{equation*}
where $\mathbf{X}_{1},\ldots,\mathbf{X}_{n_{1}}$ are smooth first order differential operators invariant with respect to the
 left translations on $\X$ and homogeneous
of degree one with respect to the dilation group
$\{\delta_{r}\}_{r>0}$ of $\X$.  It is well known that $\L$ is a
hypoelliptic operator and $-\L$ generates a heat semigroup of
self-adjoint linear operators $e^{-t\L}$ on $L^{2}(\X)$. On
stratified Lie groups, the study of the fractional heat kernel
related to $\L$ has been ongoing for a period; see \cite{F18,M23}
and the references therein. The Poisson kernel on stratified Lie
groups was introduced and studied in \cite{F82}, while the
investigation of the Caffarelli-Silvestre type extension has only
emerged more recently. In \cite{F15}, Ferrari and Franchi studied  a
Harnack-type estimate for the fractional powers of sub-Laplacian on
Carnot groups. The proof of the main result of \cite{F15} follows
the ``abstract" formulation of a technique recently introduced by
Caffarelli and Silvestre in \cite{C07}. Later,  such a method to
deduce Harnack's inequalities was further developed by  Stinga and
Torrea \cite{ST-2010} and Stinga and Zhang \cite{SZ}. In the setting
of Carnot groups, owing to the group convolution, the authors in
\cite{F15} can represent the fractional powers of sub-Laplacian  via
the  spectral resolution.

For $(s,\alpha,\sigma)\in(0,1)\times(0,1)\times(0,1)$, denote by
\begin{equation}\label{eq1.1}\mathbf{s}=s\alpha\ \ \ \ \   \mathrm{and}\ \ \ \ \
\ \widetilde{\mathbf{s}}=s\sigma\end{equation} throughout this
paper.  Let $H_{\alpha,t}:=\{e^{-t\L^{\alpha}}\}_{t>0}$  denote the
fractional heat semigroup and   let $P_{\sigma,t}$ denote the
Caffarelli-Silvestre type extension. In Section \ref{sec-2}, we recall the definitions and some notations of stratified Lie groups $\X$, and
 collect some properties of the heat semigroup $\{e^{-t\L}\}_{t>0}$ and its kernel estimates. Section
  \ref{sec-2-3} is devoted to the study of the fractional heat semigroup and the Caffarelli-Silvestre type extension. Precisely,
   avoiding the use of the Fourier transform, we investigate the kernel estimates
   of $\{H_{\alpha,t}\}_{t>0}$ and $\{P_{\sigma,t}\}_{t>0}$ with the help of the subordinate
    formulas (\ref{eq-heat-1}) and (\ref{Poisson}), respectively; see Propositions \ref{pro-frac} \&\ \ref{pro-1}.
    In addition, we also discuss the spatial continuity of $H_{\alpha,t}u(g)$ with $u\in L^{p}(\X)$ and fixed $t>0$, which will play a key role
    in the proof of our main theorems. In Section \ref{sec-2.3}, we mainly study the fractional power of the sub-Laplace
    operators and the Riesz potential operators via the fractional heat semigroup  and the Caffarelli-Silvestre type extension, and
    give the definition of fractional Sobolev spaces using the fractional power of the sub-Laplace operators as a gradient.

In Section \ref{sec-3},   denote by  $\L^{\Theta}$ the fractional
power of $\L$,  we introduce the homogeneous fractional Sobolev
space $\mathcal{\dot{W}}^{2\Theta,p}(\X)$, $p\in[1,\infty)$,   as
$$\mathcal{\dot{W}}^{2\Theta,p}(\X):= \Big\{u\in L^{p}(\X):\ \mathcal{L}^{\Theta}u\in L^{p}(\X)\Big\}$$
endowed with the norm
$\|u\|_{\mathcal{\dot{W}}^{2\Theta,p}(\X)}=\|\L^{\Theta}u\|_{L^{p}(\X)},$
where $\Theta=\mathbf{s}$ if the operator $\mathcal{T}_t$ is
$H_{\alpha,t^{2\alpha}}$ or $\Theta=\widetilde{\mathbf{s}}$ if the
operator $\mathcal{T}_t$ is $P_{\sigma,t}$.
 The fractional Sobolev capacity of a compact
 set $\mathcal{K}\subset\X$, denoted by $Cap_{\mathcal{\dot{W}}^{2\Theta,p}}(\mathcal{K})$, is defined as
$$Cap_{\mathcal{\dot{W}}^{2\Theta,p}}(\mathcal{K}):=\inf\Big\{\|u\|_{\mathcal{\dot{W}}^{2\Theta,p}(\X)}^{p}:\ u\in C_{c}^{\infty}(\X),
\ u\geq 1_{\mathcal{K}}\Big\},$$
where $1_{\mathcal{K}}$ denotes the characteristic function of $\mathcal{K}$.

In Section  \ref{sec-3.1}, we provide several basic properties of
the fractional Sobolev type spaces. An interesting fact  is that
$C^{\infty}_{c}(\X)$ is dense in
$\mathcal{\dot{W}}^{2\mathbf{s},p}(\X)$, which is not true in
general metric measure spaces, see Proposition \ref{dense-1}. Unlike
the general metric measure spaces,  the  fractional gradients and
fractional divergences, weak $s$-gradients  can be well-defined on
stratified Lie groups, respectively. These differential operators
enable us to introduce  distributional fractional Sobolev spaces
$\mathbb{W}^{\Theta,p}(\X)$, see (\ref{ffs}). Based on the idea of
 \cite{Z23}, we obtain the equivalence of $\mathbb{W}^{\Theta,p}(\X)$ and $\mathcal{W}^{\Theta,p}(\X)$, see Proposition \ref{w-w}.  Section \ref{sec-3.2} mainly investigates the properties of the fractional Sobolev capacity $Cap_{\mathcal{\dot{W}}^{2\Theta,p}}(\cdot)$
and the Riesz capacity $Cap_{\Theta,p}(\cdot)$. By the aid of Proposition \ref{dense-1}, we can prove
 $Cap_{\mathcal{\dot{W}}^{2\Theta,p}}(\cdot)$ and $Cap_{\Theta,p}(\cdot)$ are equivalent, see Proposition \ref{relation}. Moreover,
as an application of  the equivalence obtained in Proposition \ref{relation},  two strong-type
capacitary inequalities related to
$Cap_{\mathcal{\dot{W}}^{2\Theta,p}}(\cdot)$ can be established; see Propositions
\ref{strong} \&\ \ref{strong-1}.

With the above preparation, we next state the main results of this
paper with the proofs given in Section \ref{sec-4}. Denote by $T(O)$
the tent based on an open subset $O$ of $\X$, i.e.,
$$T(O)=\Big\{(g,t)\in\X\times\mathbb{R}_{+}:\ B(g,r)\subseteq O\Big\}$$
with $B(x,r)$ being the open ball centered at $g\in\X$ with radius $r>0$. For $t\in(0,\infty)$, the $(p,\Theta)$-fractional
 capacity minimizing function associated with both $\mathcal{\dot{W}}^{2\Theta,p}(\X)$ and a nonnegative measure $\mu$ on
 $\X\times\mathbb{R}_{+}$, denoted by $c_{p}(\mu;t)$, is defined as
\begin{equation}\label{1.4}c_{p}(\mu;t):=\inf\Big\{Cap_{\mathcal{\dot{W}}^{2\Theta,p}}(O):\
\text{a bounded open set}\ O\subset\X,\ \mu(T(O))\geq t\Big\}.
\end{equation}
 Motivated by \cite{X1,zh-2,liu},  we characterize the measure
$\mu$ on $\X\times\mathbb{R}_{+}$ so that
$u\mapsto \mathcal{T}_tu$
is bounded from fractional Sobolev spaces
$\mathcal{\dot{W}}^{2\Theta,p}(\X)$ into the Lebesgue spaces
$L^{q}(\X\times\mathbb{R}_{+},\mu)$ on $\X\times\mathbb{R}_{+}$. Precisely,
\begin{theorem}\label{thm1}
Let $(s,\alpha,\sigma)\in(0,1)\times(0,1)\times(0,1)$,
$p\in(1,\min\{\mathcal{Q}/2\Theta,\Q/(1-2\Theta)\})$  and $\mu$ be a
nonnegative measure on $\X\times\mathbb{R}_{+}$.

$\mathrm{(a)}$ The following two statements are equivalent.

\item{\rm (i)} For $u\in \mathcal{\dot{W}}^{2\Theta,p}(\X)$,
$$\Big(\int_{\X\times\mathbb{R}_{+}}\big|\mathcal{T}_tu(g)\big|^{q}d\mu(g,t)\Big)^{1/q}\lesssim \|\nabla_{\X,w}^{2\Theta}u\|_{L^{p}(\X;H\X)}.$$

\item{\rm (ii)} The measure $\mu$ satisfies
\begin{equation*}
\infty>\left\{\begin{aligned}
&I_{p,q}(\mu):=\int^{\infty}_{0}\Big(\frac{t^{p/q}}{c_{p}(\mu;t)}\Big)^{q/(p-q)}\frac{dt}{t},\ &\ 0<q<p;\\
&\sup_{t>0}\frac{t^{1/q}}{(c_{p}(\mu;t))^{1/p}},\ &\ p\leq q<\infty.
\end{aligned}\right.
\end{equation*}

$\mathrm{(b)}$ When $q\in[p,\infty)$,  it further holds
$$(i)\Longleftrightarrow(ii)\Longleftrightarrow(iii)\Longleftrightarrow(iv),$$
where
\item{\rm (iii)}
$$\sup_{\lambda>0}\lambda(\mu(\{(g,t)\in\X\times\mathbb{R}_{+}:|\mathcal{T}_tu(g)|>\lambda\}))^{1/q}\lesssim \|\nabla_{\X,w}^{2\Theta}u\|_{L^{p}(\X;H\X)}\quad \forall\ u\in \mathcal{\dot{W}}^{2\Theta,p}(\X).$$

\item{\rm (iv)} For any bounded open set $O\subseteq \X$,
$$(\mu(T(O)))^{p/q}\lesssim Cap_{\mathcal{\dot{W}}^{2\Theta,p}}(O).$$
\end{theorem}

As the limiting case $t\rightarrow 0$ of Theorem \ref{thm1}, a
characterization of the measure $\mu$ on $\X$ ensuring the
continuity of the inclusion
$\mathcal{\dot{W}}^{2\Theta,p}(\X)\subset L^{q}(\X,\mu)$ is also
obtained as follows.

\begin{theorem}\label{thm2}
Let $(s,\alpha,\sigma)\in(0,1)\times(0,1)\times(0,1)$,
$p\in(1,\min\{\mathcal{Q}/2\Theta,\Q/(1-2\Theta)\})$. Given a
nonnegative measure $\nu$ on $\X$, let
$$h_{p}(t):=\inf\Big\{Cap_{\mathcal{\dot{W}}^{2\Theta,p}}(E):\ E\subset \X, \nu(E)\geq t\Big\}\ \ \forall\ t\in \mathbb{R}_{+}.$$
The following statements are equivalent:
\item{\rm (i)} For any $u\in \mathcal{\dot{W}}^{2s,p}(\X)$,
$$\Big(\int_{\X}|u(g)|^{q}d\nu(g)\Big)^{1/q}\lesssim \|\nabla_{\X,w}^{2\Theta}u\|_{L^{p}(\X;H\X)}.$$

\item{\rm (ii)} The measure $\nu$ satisfies
\begin{equation*}
\infty>                                                                                                                          \left\{\begin{aligned}
&\|h_{p}\|:=\Big(\int^{\infty}_{0}\frac{d\xi^{{p}/{(p-q)}}}{(h_{p}(\xi))^{{q}/{(p-q)}}}\Big)^{(p-q)/(pq)},\ & 0<q<p;\\
&\sup_{E\subset\X}\frac{(\nu(E))^{1/q}}{(Cap_{\mathcal{\dot{W}}^{2\Theta,p}}(E))^{1/p}},\
& p\leq q<\infty.
\end{aligned}\right.
\end{equation*}

\end{theorem}

Section \ref{sec-5} mainly investigate the functional and
geometrical aspects of the Besov capacity. In order to do this, we
first introduce  the Besov type space, which is defined by the
operator semigroup as follows:
$${B_{p,q,\mathcal{H}}^{{\alpha},\beta}(\X)} :=\Big\{u\in L^{p}(\X): N_{p,q,\mathcal{H}}^{{\alpha},\beta}(u):=\bigg(\int_{0}^{\infty}
                       \Big(\int_{\X}H_{\alpha,t}(|u-u(g)|^{p})(g)dg\Big)^{q/p}\frac{dt}{t^{{\beta q}/{2}+1}}
                    \bigg)^{1/q}<\infty \Big\}$$
endowed with the following norm
$$\|u\|_{B_{p,q,\mathcal{H}}^{{\alpha},\beta}(\X)}:=\|u\|_{L^{p}(\X)}+N_{p,q,\mathcal{H}}^{{\alpha},\beta}(u),$$
where $(p,q,\beta)\in[1,\infty)\times[1,\infty)\times(0,\infty)$. If $H_{\alpha,t}$ is replaced by $P_{\sigma,t}$ with $ \sigma\in
(0,1)$, we denote by $B_{p,q,\mathcal{P}}^{ {\sigma},\beta}(\X)$ and
$N_{p,q,\mathcal{P}}^{ {\sigma},\beta}$ the corresponding Besov
spaces and semi-norms, respectively.
Then in Section \ref{sec-5-1}, we investigate the dense subspaces
and the min-max property of
$B_{p,q,\mathcal{H}}^{{\alpha},\beta}(\X)$, respectively (see Proposition
\ref{desity} and Lemma \ref{max}). Moreover, we also
study the relationship between the fractional Sobolev space
$\mathcal{\dot{W}}^{2\Theta,p}(\X)$ and the Besov space
(see Corollary \ref{BW}).

Another  Besov space is defined by the difference as follows:
$$B_{p,q}^{\beta}(\X):=\Big\{u\in L^{p}(\X):
N_{p,q}^{\beta}(u)<\infty\Big\},$$ where
$$N_{p,q}^{\beta}(u):=\Big(\int_{0}^{\infty}\Big(\int_{\X}\int_{B(g,r)}\frac{|u(g)-u(g')|^{p}}{r^{2\beta
p+\Q}}dg'dg\Big)^{q/p}\frac{dr}{r}\Big)^{1/q}.$$
The norm of $u\in B_{p,q}^{\beta}(\X)$ is defined as
$$\|u\|_{B_{p,q}^{\beta}(\X)}:=\|u\|_{L^{p}(\X)}+N_{p,q}^{\beta}(u).$$ In Proposition
\ref{com-2}, we prove  that $B_{p,q,
\mathcal{H}}^{\alpha,2\beta}(\X)$ is identical to
$B_{p,q}^{\alpha\beta}(\X)$ under the condition that
$\beta\in(0,1/p)$.

In Section \ref{sec-6.1}, we introduce  two classes of Besov
capacities: for an arbitrary set $E\subset\X$, one is defined as
$$\textrm{Cap}_{B_{p,p,\mathcal{H}}^{\alpha,\beta}}(E):=\inf\Big\{\|u\|^{p}_{B_{p,p,\mathcal{H}}^{\alpha,\beta}(\X)}:
u\in B_{p,p,\mathcal{H}}^{\alpha,\beta}(\X)~\mbox{and}~
E\subset\{g\in\X:u(g)\geq1\}^{o}\Big\} $$ and the other is defined
as
$$\textrm{Cap}_{B_{p,p,\mathcal{H}}^{\alpha,\beta}}^{*}(E):=\inf\Big\{\|u\|^{p}_{B_{p,p,\mathcal{H}}^{\alpha,\beta}(\X)}: u\in C_{c}^{\infty}(\X)~\mbox{and}~ E\subset\{g\in\X:u(g)\geq1\}^{o}\Big\}.$$
Similarly to the fractional Sobolev capacity
$Cap_{\mathcal{\dot{W}}^{2\Theta,p}}(\cdot)$, we first investigate
the measure-theoretic properties of the Besov capacity
$\textrm{Cap}_{B_{p,p,\mathcal{H}}^{\alpha,\beta}}(\cdot)$ (see
Lemma \ref{CP}). Afterwards we give the equivalence of these two
classes of Besov capacities (see Lemma \ref{EquiC}). By the aid of
the strong-type capacitary inequality, we obtain

\begin{theorem}\label{thm3}
Let $\alpha\in(0,1)$ and  $p\in(1,2\alpha)$. Given a nonnegative
measure $\nu$ on $\X$, let
$$h_{p,\nu}(t):=\inf\Big\{\textrm{Cap}_{B_{p,p,\mathcal{H}}^{\alpha,\beta}}(E):\ E\subset \X, \nu(E)\geq t\Big\}\ \ \forall\ t\in \mathbb{R}_{+}.$$
  Assume that $\beta\in(0,1/\alpha)$. For the Besov capacity $\textrm{Cap}_{B_{p,p,\mathcal{H}}^{\alpha,\beta}}(\cdot)$,
the following statements are equivalent:
\item{\rm (i)}
$$\Big(\int_{\X}|u(g)|^{q}d\nu(g)\Big)^{1/q}\lesssim \|u\|_{B_{p,p,\mathcal{H}}^{\alpha,\beta}(\X)}\ \
\ \ \forall\ u\in B_{p,p,\mathcal{H}}^{\alpha,\beta}(\X).$$

\item{\rm (ii)} The measure $\nu$ satisfies
\begin{equation*}
\infty>                                                                                                                          \left\{\begin{aligned}
&\|\nu\|_{p,q}:=\Big(\int^{\infty}_{0}\Big(\frac{t^{p/q}}{h_{p,\nu}(t)}\Big)^{q/(p-q)}\frac{dt}{t}\Big)^{(p-q)/p},\ & 0<q<p;\\
&\sup_{E\subset\X}\frac{(\nu(E))^{1/q}}{(\textrm{Cap}_{B_{p,p,\mathcal{H}}^{\alpha,\beta}}(E))^{1/p}},\
&  p\leq q<\infty.
\end{aligned}\right.
\end{equation*}

%
\end{theorem}

\begin{remark}
 \item{(i)} Our approach in this article combines semigroup theory as well as the nonlocal calculus for $\L$ to
  overcome the difficulties caused by the lack of the Fourier transform. Moreover, we point out that, even on the setting $\mathbb{R}^{n}$,
  our main results obtained in Theorems \ref{thm1} \&\ \ref{thm2} are new and generalize the former results obtained by Xiao
   \cite[Theorems 1.1, 1.2]{X1}. The result of Theorem \ref{thm3} is also  new.

\item{(ii)} Consider the following heat equation:
\begin{equation}\label{eq-0.1}
\left\{\begin{aligned}
 &(\partial_{t}+\L^{\alpha})f(g,t)=0,&\  (g,t)\in \X\times\mathbb{R}_{+};\\
 &f(g,0)=u(g),&\  g\in \X.
  \end{aligned}
\right.\end{equation}
For $u\in L^{2}(\X)$, the solution to the above heat equation (\ref{eq-0.1}) can be
represented as
$$f(g,t):=H_{\alpha,t}u(g):=\int_{\X}K_{\alpha,t}(g'^{-1}g)u(g')dg' \ \ \ \ \forall\ (g,t)\in\X\times\mathbb{R}_{+},$$
where $K_{\alpha,t}(\cdot)$ is the heat kernel of the semigroup $\{H_{\alpha,t}\}_{t>0}$. By the aid of fractional Sobolev capacities,
Theorem \ref{thm1} provide geometric characterizations of Borel measures $\mu$ such that fractional Sobolev type spaces on
 $\X$ can be extended to the weighted Lebesgue spaces related to $\mu$ on $\X\times\mathbb R_{+}$ via the family of dissipative operators:
$u\mapsto f(g,t^{2\alpha}) $ with $u$ belonging to the fractional
Sobolev spaces, respectively. Also, Theorem \ref{thm2} can be seen
as the limit of Theorem \ref{thm1}, which gives the capacitary
characterization of the embeddings from fractional Sobolev type
spaces on $\X$ to the weighted Lebesgue spaces on $\X$.
\end{remark}

\subsection{Notation}
Throughout this paper,  all notations will be given as needed.

\begin{itemize}
  \item $\mathbb{N}:=\{1,2,3,\ldots\}$, $\mathbb{Z}:=\{0,\pm1,\pm2, \ldots\}$. For any set $E\subset \mathcal{X},$ denote by $\mathbf{1}_{E}$ the characteristic function of $E.$
  \item We denote by $C_{c}^{\infty}(\X)$ the space of all infinitely differentiable functions on $\X$ with compact
  support, and by $\mathscr{S}(\X)$ the Schwartz class on $\X.$
  \item For every multi-index $\beta=(\beta_{1},\ldots,\beta_{n})\in\mathbb{N}^{n}$, let $|\beta|:=\beta_{1}+\cdots+\beta_{n}$ and $D^{\beta}:=\mathbf{X}_{1}^{\beta_{1}}\cdots\mathbf{X}_{n}^{\beta_{n}}.$ We also use the notation $(\partial/\partial x_{i})^{\beta}$ or $\partial^{\beta}$ to denote differentiation with respect to the standard basis of $\mathbb{R}^{n}$.
  \item The convolution of two functions $f,u:\X\rightarrow\mathbb{R}$ is defined whenever it makes sense by
  $$(f\ast u)(g)=\int_{\X}f(gg'^{-1})u(g')dg'=\int_{\X}f(g')u(g'^{-1}g)dg'.$$
  \item The symbol $\thicksim$ between two positive expressions $u,v$ means that their ratio
  $\frac{u}{v}$ is bounded from above and below by positive constants. The symbol $\lesssim$ (respectively $\gtrsim$)  between two nonnegative expressions $u,v$
  means that there exists a constant $C>0$ such that $u\leq Cv$ (respectively $u\geq Cv$).
\end{itemize}

\section{Preliminaries}\label{sec-2}
\subsection{Stratified Lie groups}\label{sec-2.1}

We use the same notations and terminologies for stratified Lie
groups as those in \cite{F75}. A Lie group $\X$ is called stratified
if it is nilpotent, connected and simple connected, and its Lie
algebra $\mathfrak{g}$
 is equipped with a family of dilations: $\{\delta_{r}:r>0\}$ and
$\mathfrak{g}$ can be expressed as  a direct sum
$\oplus_{j=1}^{m}\mathfrak{g}_{j}$ such that
$[\mathfrak{g}_{1},\mathfrak{g}_{j}]\subset g_{1+j}$ for $j<m$,
$[\mathfrak{g}_{1}, \mathfrak{g}_{m}]=0$, and
$\delta_{r}(\mathbf{X})=r^{j}\mathbf{X}$ for $\mathbf{X}\in
\mathfrak{g}_{j}$.

 $\mathcal{Q}=\sum_{j=1}^{m}jn_{j}$ is called the homogeneous dimension of $\X$, where $n_{j}=\dim\mathfrak{g}_{j}$. We always assume $\Q\geq 3$.
 $\X$
 is topologically identified with $\mathfrak{g}$ via the exponential map  $\mathrm{exp}: \mathfrak{g}\mapsto\X$ and $\delta_{r}$ is also viewed
  as an automorphism of $\X$. We fix a homogeneous norm of $\X$, which satisfies the generalized triangle inequalities
\begin{equation}\label{2.11}                                                                                                                           \left\{\begin{aligned}
&|gg'|\leq \gamma(|g|+|g'|)\ \ \ \ \forall\ g,g'\in\X;\\
&||gg'|-|g||\leq\gamma|g'|\ \ \ \ \forall\ g,g'\in\X,|g'|\leq|g|/2,
\end{aligned}\right.
\end{equation}
where $\gamma\geq 1$ is a constant. The homogeneous norm induces a
quasi-metric $d$ which is defined by $d(g,g'):=|g'^{-1}g|.$ In
particular,
$$d(e,g)=|g|\ \ \ \ \text{and}\ \ \ \ d(g,g')=d(e,g'^{-1}g),$$ where
 $e$ is  the identity element of $\X$. The ball of
radius $r$ centered at $g$ is denoted by
$B(g,r)=\{g'\in\X:d(g,g')<r\}.$ The Haar measure on $\X$ is simply
the Lebesgue measure on $\mathbb{R}^{n}$ under the identification of
$\X$ with $\mathfrak{g}$ and the identification of $\mathfrak{g}$
with $\mathbb{R}^{n}$, where $n=\sum^{m}_{j=1}n_{j}$. The measure of
$B(g,r)$ is $|B(g,r)|=b'r^{\mathcal{Q}}$ dependent of $\mathcal{Q}$,
 where $b'$ is a constant.

We identify $\mathfrak{g}$ with $\mathfrak{g}_{L}$, the Lie algebra of left-invariant vector fields on $\X$.
 Let $\{\mathbf{X}_{j}: j=1,...,n_{1}\}$ be a basis of $\mathfrak{g}_{1}$. Let $\widetilde{\mathbf{X}_{j}}$
 be the right-invariant vector field on $\X$. For any $u\in C^{1}(\X)$, we set $\widetilde{u}(g):=u(g^{-1})$.
 Then $\widetilde{\mathbf{X}_{j}}$ and $\mathbf{X}_{j}$ are related by $\widetilde{\mathbf{X}_{j}}u(g)=-\mathbf{X}_{j}\widetilde{u}(g^{-1})$,
  $\forall g\in\X$.
The sub-Laplacian $\L$ is defined by
$\L:=-\sum^{n_{1}}_{j=1}\mathbf{X}_{j}^{2}$. Below, we will use the
sub-Riemannian structure of $\X$. The horizontal bundle $H\X$ is the
subbundle of the tangent boudle $T\X$ spanned by the vector fields
$\mathbf{X}_{1},...,\mathbf{X}_{n_{1}}$, i.e., the fibers of $H\X$
are denoted as
$$H\X_{g}:=\text{span}\{\mathbf{X}_{1}(g),...,\mathbf{X}_{n_{1}}(g)\}  \ \ \ \forall\ g\in \X.$$
Then for any $g\in \X$, we can endow $H\X_{g}$ with a norm
$|\cdot|_{g}$ and a scalar product $\langle\cdot,\cdot\rangle_{g}$
making $\{\mathbf{X}_{1}(g),...,\mathbf{X}_{n_{1}}(g)\}$ an
orthonormal basis. In the sequel, if the identification of the base
point is clear or not important, we often omit $g$ in subscripts.
The $H\X$-valued operator
$$\nabla_{\X}u:=\sum_{j=1}^{n_{1}}(\mathbf{X}_{j}u)\mathbf{X}_{j} \ \ \ \ \forall\ u\in C^{1}(\X)$$
is called the horizontal gradient. And the real-valued operator
$$\text{div}_{\X}\varphi:=\sum_{j=1}^{n_{1}}\mathbf{X}_{j}\varphi_{j} \ \ \ \ \forall\ \varphi=\sum_{j=1}^{n_{1}}\varphi_{j}\mathbf{X}_{j}\in C^{1}(\X;H\X)$$
is called the horizontal divergence.

From \cite[Chapter 1]{F82}, as for the interaction of convolution and differential operators, for $j=1,\ldots,n_{1}$, we have
\begin{equation}\label{e-2.6}
  \mathbf{X}_{j}(f\ast u)=f\ast(\mathbf{X}_{j}u),\ \widetilde{\mathbf{X}_{j}}(f\ast u)=(\widetilde{\mathbf{X}_{j}}f)\ast u,\ (\mathbf{X}_{j} f)\ast u=f\ast(\widetilde{\mathbf{X}_{j}}u),
\end{equation}
where $f$ and $g$ are smooth functions with suitable decay at infinity. Moreover by a standard approximation argument, (\ref{e-2.6}) remains valid if $f,u\in \text{Lip}(\X)$ and either of them has compact support. Generally, the convolution on $\X$ is not commutative. However, there are the following simple facts: given measurable functions $f,u,h:\X\rightarrow\mathbb{R}$, suppose that two of them belong to $\text{Lip}_{c}(\X)$ and the third one is locally integrable, if additionally the function $h$ satisfies $h(g)=h(g^{-1})$ for any $g\in \X$, then using Fubini's theorem, we have
\begin{equation}\label{e-2.7}
 \left\{\begin{aligned}
 \int_{\X}f(g)(h\ast u)(g)dg&=\int_{\X}u(g)(h\ast f)(g)dg;\\
 \int_{\X}f(g)(u\ast h)(g)dg&=\int_{\X}u(g)(f\ast h)(g)dg.\end{aligned}\right.
\end{equation}

\subsection{Properties and kernel estimates of $\{e^{-t\mathcal{L}}\}_{t>0}$}\label{sec-2.3}

In this subsection,  we  collect several properties of $e^{-t\L}$
which will be used in the rest of the paper mostly without proofs.
The reader can refer to the references \cite{F75,F82,VSC}.
\begin{lemma}\label{heat-1}
The heat semigroup $\{e^{-t\L}\}_{t>0}$ satisfies  the following
properties:

\item{\rm (i)} $e^{-t\L}u(g)=u\ast e^{-t\L}(g)$, where $e^{-t\L}(g)$ is $C^{\infty}$ on $\X\times(0,\infty)$, $\int_{\X}e^{-t\L}(g)dg=1$
 for all $t$, and $e^{-t\L}(g)\geq 0$  for all $g$ and $t$.

\item{\rm (ii)} Semigroup property:
$e^{-t\mathcal{L}}e^{-t'\mathcal{L}}=e^{-(t+t')\mathcal{L}} \ \ \
\forall\ t,t'>0.$

\item{\rm (iii)} $\{e^{-t\L}\}_{t>0}$ is a contraction semigroup on
$L^{p}$ for $1\leq p\leq\infty$, which is strongly continuous for
$1\le p<\infty$.

 \item{\rm (iv)} $e^{-t\L}$ is self-adjoint.
  \end{lemma}
 From the spectral decomposition,
it is also easy to see that the sub-Laplace operator $\mathcal{L}$ is
  the generator of $\{e^{-t\L}\}_{t>0}$, that is,  for
$u\in\mathscr{D}(\mathcal{L})$ (the domain of $\mathcal{L}$),
$$\lim_{t\rightarrow0^{+}}\Big\|\dfrac{e^{-t\mathcal{L}}u-u}{t}-\mathcal{L}u\Big\|_{L^{2}(\X)}=0.$$
This implies that for
$t>0,~e^{-t\mathcal{L}}\mathscr{D}(\mathcal{L})\subset\mathscr{D}(\mathcal{L}),$
and that for $u\in\mathscr{D}(\mathcal{L}),$
\begin{equation}\label{D}
\partial_{t}e^{-t\mathcal{L}}u=e^{-t\mathcal{L}}\mathcal{L}u=\mathcal{L}e^{-t\mathcal{L}}u,
\end{equation}
where the derivative in the left-hand side of identities (\ref{D})
is taken in the sense of $L^{2}(\X).$ Please refer to the books
\cite{J,LB} for  its main properties and regularity.  The following
lemmas for the estimates of the heat kernel  can be seen from
\cite[Section IV.4]{VSC} on the nilpotent Lie group, which includes
the stratified Lie group.

\begin{lemma}\label{upper}{\rm (\cite[Theorems IV. 4.2\ \&\ IV. 4.3]{VSC})}
For all $t>0$ and  all $g\in\X,$ there exists a  constant  $c\geq1$
such that the kernel $e^{-t\L}(\cdot)$ of the heat semigroup
satisfies
$$c^{-1}t^{-\Q/2}e^{-d(e,g)^{2}/c^{-1}t}\leq e^{-t\L}(g)\leq c^{-1}t^{-\Q/2}e^{-d(e,g)^{2}/ct}.$$
\end{lemma}

\begin{lemma}\label{DH}{\rm (\cite[Theorem IV.4.2]{VSC})}
For every non-negative integer $k$ and $\beta\in\mathbb{N}^{n}$,
there exists $c>0$ such that for every $g\in\X$ and $t>0$,
$$\Big|\frac{\partial^{k}}{\partial t^{k}}D^{\beta}e^{-t\L}(g)\Big|\leq ct^{-(\Q+|\beta|+2k)/2}e^{-d(e,g)^{2}/t}.$$
\end{lemma}

By the stratified  mean value theorem in \cite{F82} and Lemma
\ref{DH}, we obtain

\begin{lemma}\label{holder}
For all $t>0$ and   all $g,h\in\X,$ there exist positive constants
$C,c$ such that
\begin{equation*}
  \Big|e^{-t\L}(gh)-e^{-t\L}(g)\Big|\leq
  \begin{cases}
 & Cd(e,h)t^{-(\Q+1)/2},\\
 & Cd(e,h)t^{-(\Q+1)/2}e^{-cd(e,g)^{2}/t},\ \ if \ \ d(e,h)\leq d(e,g)/2.
  \end{cases}
\end{equation*}
\end{lemma}

\section{Fractional heat semigroups and Caffarelli-Silvestre type extensions}\label{sec-2-3}
In this section, we first give some properties of the fractional
heat semigroup $H_{\alpha,t}$ and the Caffarelli-Silvestre type
extension $P_{\sigma,t}$. Then we   investigate the fractional
powers of the sub-Laplace operator, Riesz potential operator and
fractional Sobolev spaces, respectively.

\subsection{Fractional heat semigroups}\label{sec-2-3-1}
It is well known that $\L^{\alpha}$ is the generator of a Markovian semigroup $\{H_{\alpha,t}\}_{t>0}$ which is related to $\{e^{-t\L}\}_{t>0}$ by the subordination formula
\begin{equation}\label{eq-heat-1}
  H_{\alpha,t}u=\int_{0}^{\infty}\eta_{t}^{\alpha}(s)e^{-s\L}uds=\int_{0}^{\infty}\eta_{1}^{\alpha}(\tau)e^{-\tau t^{1/\alpha}\L}ud\tau,
\end{equation}
where
\begin{equation*}
  \eta^{\alpha}_{t}(\lambda)=
  \left\{\begin{aligned}
&\frac{1}{2\pi i}\int_{\iota-i\infty}^{\iota+i\infty}e^{z\lambda-tz^{\alpha}}dz,\ & \lambda\geq 0;\\
&0,\ & \lambda<0
  \end{aligned}\right.
\end{equation*}
for $\iota>0$, $t>0$ and $0<\alpha<1$. By \cite[Chapter IX]{Y80}, we know that $\eta_{t}^{\alpha}(\lambda)$ is non-negative
 for $\lambda\geq 0$ and
$$\int_{0}^{\infty}\eta_{t}^{\alpha}(s)e^{-s\lambda}ds=e^{-t\lambda^{\alpha}}
$$ for $\lambda>0$.
Moreover,
\begin{equation*}\label{2.16}
  \eta_{t}^{\alpha}(s)\leq \min\Big\{\frac{1}{t^{1/\alpha}},\ \frac{t}{s^{1+\alpha}}\Big\}
\end{equation*}
and
\begin{equation*}\label{3.1}
  \int_{0}^{\infty}\eta_{t}^{\alpha}(s)ds=1.
\end{equation*}
For $-\infty<\delta<\alpha$,
\begin{equation}\label{2.17}
  \int_{0}^{\infty}\eta_{t}^{\alpha}(s)s^{\delta}ds=\frac{\Gamma(1-\delta/\alpha)}{\Gamma(1-\delta)}t^{\delta/\alpha}.
\end{equation}
If $\delta\geq \alpha$,
\begin{equation*}\label{2.18}
  \int_{0}^{\infty}\eta_{t}^{\alpha}(s)s^{\delta}ds=+\infty.
\end{equation*}

Using Lemma \ref{heat-1}, (\ref{eq-heat-1}) and the properties of
$\eta_{t}^{\alpha}(\cdot)$, we have
\begin{lemma}\label{fracheat-1}
Let $\alpha\in(0,1)$. The fractional heat semigroup $\{H_{\alpha,t}\}_{t>0}$ satisfies  the following
properties:

\item{\rm (i)} $H_{\alpha,t}u(g)=u\ast K_{\alpha,t}(g)$, where $K_{\alpha,t}(g)$ is $C^{\infty}$ on
 $\X\times(0,\infty)$, $\int_{\X}K_{\alpha,t}(g)dg=1$ for all $t$, and $K_{\alpha,t}(g)\geq 0$ for all $g$ and $t$.

\item{\rm (ii)} Semigroup property:
$H_{\alpha,t}H_{\alpha,t'}=H_{\alpha,t+t'}\ \forall\ t,t'>0.$

\item{\rm (iii)} $\{H_{\alpha,t}\}_{t>0}$ is a contraction semigroup on
$L^{p}(\X)$, $1\leq p\leq\infty$, which is strongly continuous for $1\le
p<\infty$.

 \item{\rm (iv)} $H_{\alpha,t}$ is self-adjoint.
  \end{lemma}
Next we give some  estimates for the fractional heat kernel $K_{\alpha,t}(\cdot)$.
\begin{proposition}\label{pro-frac} Let $\alpha\in(0,1)$.

\item{\rm (i)} For every $t>0$ and $g\in\X$,
$$\min\Big\{t^{-\Q/2\alpha},\frac{t}{d(e,g)^{\Q+2\alpha}}\Big\}\lesssim K_{\alpha,t}(g)
\lesssim
\min\Big\{t^{-\Q/2\alpha},\frac{t}{d(e,g)^{\Q+2\alpha}}\Big\},$$
which implies that
$$\frac{t}{(t^{1/2\alpha}+d(e,g))^{\Q+2\alpha}}\lesssim K_{\alpha,t}(g)\lesssim \frac{t}{(t^{1/2\alpha}+d(e,g))^{\Q+2\alpha}}.$$

{\rm (ii)} For every $t>0$, $g,h\in\X$ and $d(e,h)\leq d(e,g)/2$,
 $$\Big|K_{\alpha,t}(gh)-K_{\alpha,t}(g)\Big|\lesssim \frac{d(e,h)}{t^{1/2\alpha}}\min\Big\{t^{-\Q/2\alpha},\frac{t^{1+1/2\alpha}}
 {d(e,g)^{2\alpha+\Q+1}}\Big\},$$
which implies that
$$\Big|K_{\alpha,t}(gh)-K_{\alpha,t}(g)\Big|\lesssim \frac{d(e,h)}{t^{1/2\alpha}}\frac{t}{(t^{1/2\alpha}+d(e,g))^{\Q+2\alpha}}.$$
\end{proposition}
\begin{proof}
The proof of (i) can be seen from \cite[Remark 4]{F18}.

(ii) By  (\ref{eq-heat-1}) and Lemma \ref{holder}, we can get
\begin{eqnarray*}
|K_{\alpha,t}(gh)-K_{\alpha,t}(g)|
&=&\Big|\int^{\infty}_{0}\eta^{\alpha}_{t}(s)\Big(e^{-s\L}(gh)-e^{-s\L}(g)\Big)ds\Big|\\
&\lesssim& \int^{\infty}_{0}\frac{t}{s^{1+\alpha}}d(e,h)s^{-(\Q+1)/2}e^{-cd(e,g)^{2}/s}ds,
\end{eqnarray*}
which, together with the change of variables, gives
\begin{eqnarray*}
|K_{\alpha,t}(gh)-K_{\alpha,t}(g)|
&\lesssim&
\frac{d(e,h)}{{t^{1/2\alpha}}}t^{-\Q/2\alpha}\int^{\infty}_{0}u^{-(1+\alpha+(\Q+2)/2)}e^{-cd(e,g)^{2}/t^{1/\alpha}u}du.
\end{eqnarray*}
Let $d(e,g)^{2}/(t^{1/\alpha}u)=r^{2}$. Then
\begin{eqnarray*}
|K_{\alpha,t}(gh)-K_{\alpha,t}(g)|
&\lesssim& \frac{d(e,h)}{{t^{1/2\alpha}}}\frac{t^{1+1/2\alpha}}{d(e,g)^{2\alpha+\Q+1}}\int_{0}^{\infty}e^{-cr^{2}}r^{2\alpha+\Q}dr\\
&\lesssim& \frac{d(e,h)}{{t^{1/2\alpha}}}\frac{t^{1+1/2\alpha}}{d(e,g)^{2\alpha+\Q+1}}.
\end{eqnarray*}

Therefore, we obtain
$$|K_{\alpha,t}(gh)-K_{\alpha,t}(g)|\lesssim \frac{d(e,h)}{{t^{1/2\alpha}}}\frac{t^{1+1/2\alpha}}{d(e,g)^{2\alpha+\Q+1}}.$$

On the other hand, noticing that
$$|K_{\alpha,t}(gh)-K_{\alpha,t}(g)|\lesssim \int_{0}^{\infty}\eta_{t}^{\alpha}(s)d(e,h)s^{-(\Q+1)/2}ds,$$
 by (\ref{eq-heat-1}) and (\ref{2.17}), we obtain
\begin{eqnarray*}
  |K_{\alpha,t}(gh)-K_{\alpha,t}(g)| &\lesssim&\frac{d(e,h)}{t^{1/2\alpha}}t^{-\Q/2\alpha} \int_{0}^{\infty}\eta_{1}^{\alpha}(\tau)\tau^{-(\Q+1)/2}d\tau  \\
   &\lesssim&\frac{d(e,h)}{t^{1/2\alpha}}t^{-\Q/2\alpha}.
\end{eqnarray*}

Hence, in any case, we have
$$\Big|K_{\alpha,t}(gh)-K_{\alpha,t}(g)\Big|\lesssim \frac{d(e,h)}{t^{1/2\alpha}}\frac{t}{(t^{1/2\alpha}+d(e,g))^{\Q+2\alpha}}.$$

\end{proof}
By the property of the semigroup $\{H_{\alpha,t}\}_{t>0}$, we can show
that the condition $d(e,h)<d(e,g)/2$ in Lemma \ref{holder} can be
replaced by $d(e,h)<t^{1/2\alpha}$. Precisely,
\begin{lemma}\label{le-2.1}
 For all $t>0$ and $d(e,h)<t^{1/2\alpha}$,
$$\Big|K_{\alpha,t}(gh)-K_{\alpha,t}(g)\Big|\lesssim \frac{d(e,h)}{t^{1/2\alpha}}\frac{t}{(t^{1/2\alpha}+d(e,g))^{\Q+2\alpha}}.$$
\end{lemma}
\begin{proof}
Assume that $d(e,h)<t^{1/2\alpha}$. It is obvious that if $d(e,h)<t^{1/2\alpha}<d(e,g)/2$ or $d(e,h)<d(e,g)/2<t^{1/2\alpha}$,  this lemma holds
 by Proposition \ref{pro-frac}. We only consider the case $d(e,g)/2<d(e,h)<t^{1/2\alpha}$. By the property of the semigroup $\{H_{\alpha,t}\}_{t>0}$,
 we divide $\Big|K_{\alpha,t}(gh)-K_{\alpha,t}(g)\Big|$ as follows:
\begin{eqnarray*}
\Big|K_{\alpha,t}(gh)-K_{\alpha,t}(g)\Big|&=&\left|\int_{\X}\Big(K_{t/2}(u^{-1}(gh))-K_{t/2}(u^{-1}g)\Big)K_{t/2}(u)du\right|\\
&\leq& S_{1,1}+S_{1,2}+S_{2},
\end{eqnarray*}
where
$$\left\{\begin{aligned}
S_{1,1}&:=\int_{d(gh,u)<2d(e,h)}\Big|K_{\alpha,t/2}(u^{-1}(gh))\Big|\cdot\Big|K_{\alpha,t/2}(u)\Big|du;\\
S_{1,2}&:=\int_{d(gh,u)<2d(e,h)}\Big|K_{\alpha,t/2}(u^{-1}g)\Big|\cdot\Big|K_{\alpha,t/2}(u)\Big|du;\\
S_{2}&:=\int_{d(gh,u)\geq2d(e,h)}\Big|K_{\alpha,t/2}(u^{-1}(gh))-K_{\alpha,t/2}(u^{-1}g)\Big|\cdot\Big|K_{\alpha,t/2}(u)\Big|du.
\end{aligned}\right.$$
By Proposition \ref{pro-frac}, we show that
\begin{eqnarray*}
S_{1,1}\lesssim\int_{d(gh,u)<2d(e,h)}t^{-\Q/2\alpha}t^{-\Q/2\alpha}du
\lesssim
t^{-\Q/2\alpha}\Big(\frac{d(e,h)}{t^{1/2\alpha}}\Big)^{\Q}.
\end{eqnarray*}
Then by  the fact that $d(e,g)/2<t^{1/2\alpha}$ and $\Q\geq 3$, we have
\begin{eqnarray*}
  S_{1,1} &\lesssim& t^{-\Q/2\alpha}\frac{d(e,h)}{t^{1/2\alpha}} \\
   &\lesssim& \frac{d(e,h)}{t^{1/2\alpha}}\frac{t}{(t^{1/2\alpha}+d(e,g))^{\Q+2\alpha}}.
\end{eqnarray*}

Note that, by the triangle
inequality, $d(g,u)<3d(e,h)$  for $d(gh,u)<2d(e,h)$. Hence, the term $S_{1,2}$ can be
dealt with similarly to the case of the term $S_{1,1}$. Below we
estimate the term $S_{2}$. Since $d(e,h)\leq d(gh,u)/2$ and
$d(e,g)/2<d(e,h)$, it follows from Proposition \ref{pro-frac} that
\begin{eqnarray*}
S_{2}&\lesssim&\int_{d(gh,u)\geq2d(e,h)}\frac{d(e,h)}{t^{1/2\alpha}}\frac{t}{(t^{1/2\alpha}+d(gh,u))^{\Q+2\alpha}} \Big|K_{\alpha,t/2}(u)\Big|du\\
&\lesssim&\frac{d(e,h)}{t^{1/2\alpha}}\frac{t}{(t^{1/2\alpha}+d(e,g))^{\Q+2\alpha}}\int_{d(gh,u)\geq2d(e,h)}\Big|K_{\alpha,t/2}(u)\Big|du.
\end{eqnarray*}
By Proposition \ref{pro-frac}, a direct computation gives
\begin{eqnarray*}
 \int_{d(gh,u)\geq2d(e,h)}\Big|K_{\alpha,t/2}(u)\Big|du
&&\lesssim \int_{\X}\frac{t}{(t^{1/2\alpha}+d(e,u))^{\Q+2\alpha}}du\\
&&\lesssim \int_{d(e,u)<t^{1/2\alpha}}t^{-\Q/2\alpha}du\\
&&\quad+\sum^{\infty}_{k=0}\int_{ d(e,u)\simeq  2^{k}t^{1/2\alpha}}\frac{t}{(t^{1/2\alpha}+d(e,u))^{\Q+2\alpha}}du\\
&&\lesssim 1.
\end{eqnarray*}
This completes the proof of Lemma \ref{le-2.1}.
\end{proof}

\begin{proposition}\label{continuous}
If $p\in[1,\infty]$ and $u\in L^{p}(\X)$, then the function $g\mapsto H_{\alpha,t}u(g)$ is continuous on $\X$ for fixed $t>0$.
\end{proposition}
\begin{proof}
To begin with, for fixed $t>0$, choose $g, g_{0}\in \X$. Write
\begin{eqnarray*}
  \Big|H_{\alpha,t}u(g)-H_{\alpha,t}u(g_{0})\Big| &\leq& \int_{\X}\Big|K_{\alpha,t}(g'^{-1}g)-K_{\alpha,t}(g'^{-1}g_{0})\Big||u(g')|dg'.  \\
  \end{eqnarray*}
Below we consider two cases:

{\it Case 1: $d(g,g_{0})<t^{1/2\alpha}$.} In this case, by
Proposition \ref{pro-frac} and H\"{o}lder's inequality, we have
\begin{eqnarray*}
  \Big|H_{\alpha,t}u(g)-H_{\alpha,t}u(g_{0})\Big| &\lesssim& \int_{\X}\frac{d(g,g_{0})}{t^{1/2\alpha}}\frac{t}{(t^{1/2\alpha}+d(g,g'))^{\Q+2\alpha}}u(g')dg'  \\
   &\lesssim& \frac{d(g,g_{0})}{t^{1/2\alpha}}t^{-\mathcal{Q}/2\alpha}\int_{\X}\frac{1}{(1+d(g,g')/t^{1/2\alpha})^{\Q+2\alpha}}u(g')dg'  \\
   &\lesssim& \frac{d(g,g_{0})}{t^{1/2\alpha}}t^{-\mathcal{Q}/2\alpha}\|u\|_{L^{p}(\X)}\Big(\int_{\X}\frac{1}{(1+d(g,g')/t^{1/2\alpha})^{(\Q+2\alpha)p'}}dg'\Big)^{1/p'} \\
   &\lesssim&  \|u\|_{L^{p}(\X)}t^{-(1+\mathcal{Q}/p)/2}d(g,g_{0}).
\end{eqnarray*}

{\it Case 2: $d(g,g_{0})\geq t^{1/2\alpha}$.} For this case, it can be deduced from
Proposition \ref{pro-frac} that
\begin{eqnarray*}
   &&\Big|K_{\alpha,t}(g'^{-1}g)-K_{\alpha,t}(g'^{-1}g_{0})\Big|  \\
   &&\lesssim \Big(\frac{d(g,g_{0})}{t^{1/2\alpha}}\Big)\Big\{\frac{t}{(t^{1/2\alpha}+d(g,g'))^{\Q+2\alpha}}
   +\frac{t}{(t^{1/2\alpha}+d(g_{0},g'))^{\Q+2\alpha}}\Big\}.
\end{eqnarray*}
Thus we have
\begin{equation*}
  \Big|H_{\alpha,t}u(g)-H_{\alpha,t}u(g_{0})\Big| \lesssim M_{1}+M_{2},
\end{equation*}
where
$$\left\{\begin{aligned}
M_{1}&:=\int_{\X}\Big(\frac{d(g,g_{0})}{t^{1/2\alpha}}\Big)\frac{t}{(t^{1/2\alpha}+d(g,g'))^{\Q+2\alpha}}u(g')dg';\\
M_{2}&:=\int_{\X}\Big(\frac{d(g,g_{0})}{t^{1/2\alpha}}\Big)\frac{t}{(t^{1/2\alpha}+d(g_{0},g'))^{\Q+2\alpha}}u(g')dg'.
\end{aligned}\right.$$
Then similarly to the procedure of Case 1, we deduce
$$M_{1}+M_{2}\lesssim \|u\|_{L^{p}(\X)}t^{-(1+\mathcal{Q}/p)/2\alpha}d(g,g_{0}).$$
This means
$$ \Big|H_{\alpha,t}u(g)-H_{\alpha,t}u(g_{0})\Big| \lesssim \|u\|_{L^{p}(\X)}t^{-(1+\mathcal{Q}/p)/2\alpha}d(g,g_{0}),$$
which indicates that the function $g\mapsto e^{-t\L}u(g)$ is continuous on $\X$.
\end{proof}
\begin{remark}
A function $u: \X\rightarrow[-\infty,+\infty]$ is said to be lower semicontinuous on $\X$ if for any $g\in \X$ and $\epsilon>0$, there exists a neighborhood $U_{\epsilon}$ of $g$ such that $u(g')\geq u(g)-\epsilon$ for any $g'\in U_{\epsilon}$ when $u(g)<\infty$, and $u(g')$ tends to $\infty$ as $g'$ tends to $g$ when $u(g)=\infty$. In particular, the supremum of a sequence of continuous functions is lower semicontinuous.

\end{remark}

\subsection{Caffarelli-Silvestre type extensions}\label{sec-2-3-2}

In this section, we will focus   on the corresponding
Caffarelli-Silvestre type extension on the stratified Lie group,
which for case of notation we henceforth denote by $P_{\sigma,t}$
with  $\sigma\in(0,1)$. From \cite[(26)]{F15}, we know that for any
$\sigma\in(0,1)$, the integral kernel of $P_{\sigma, t}$, called the
fractional Poisson kernel,  can be written as
\begin{equation}\label{Poisson}
  P_{\sigma,t}(g)=\frac{t^{2\sigma}}{4^{\sigma}\Gamma(\sigma)}\int_{0}^{\infty}r^{-(1+\sigma)}e^{-t^{2}/4r}e^{-r\L}(g)dr
  =\frac{1}{\Gamma(\sigma)}\int_{0}^{\infty}e^{-r}e^{-t^{2}\L/4r}(g)\frac{dr}{r^{1-\sigma}}.
\end{equation}

 Using the above subordination formula (\ref{Poisson}) and Lemma \ref{heat-1}, we  can also get the following properties of $\{P_{\sigma,t}\}_{t>0}$.
\begin{lemma}\label{poisson-1}
The Caffarelli-Silvestre type extension $\{P_{\sigma,t}\}_{t>0}$
satisfies  the following properties:

\item{\rm (i)} $P_{\sigma,t}u(g)=u\ast P_{\sigma,t}(g)$, where $P_{\sigma,t}(g)$ is $C^{\infty}$ on $\X\times(0,\infty)$,
 $\int_{\X}P_{\sigma,t}(g)dg=1$ for all $t$, and  $P_{\sigma,t}(g)\geq 0$ for all $g$ and $t$.

\item{\rm (ii)} $\{P_{\sigma,t}\}_{t>0}$ is a contraction semigroup on $L^{p}(\X)$, $1\leq p\leq\infty$, which is strongly continuous for $p<\infty$.
\end{lemma}
Now we give a pointwise estimate of the kernel
$P_{\sigma,t}(\cdot)$.
\begin{proposition}\label{pro-1}
 Let $0<\sigma<1$.

 \item{\rm (i)} For all $(g,t)\in\X\times(0,\infty)$,
\begin{equation*}\label{k-1}
\frac{t^{2\sigma}}{(t^{2}+d(e,g)^{2})^{\Q/2+\sigma}}  \lesssim P_{\sigma,t}(g)\lesssim \frac{t^{2\sigma}}{(t^{2}+d(e,g)^{2})^{\Q/2+\sigma}}.
\end{equation*}

\item{\rm (ii)} For every $t>0$ and $d(e,h)<d(e,g)/2$,
\begin{equation}\label{eq-3.2}
\Big|P_{\sigma,t}(gh)-P_{\sigma,t}(g)\Big|\lesssim \Big(\frac{d(e,h)}{t}\Big) \frac{t^{2\sigma}}{(t^{2}+d(e,g)^{2})^{\Q/2+\sigma}}.
\end{equation}
\end{proposition}
\begin{proof}
 It can be deduced from  (\ref{Poisson}) and Lemma \ref{upper} that
$$P_{\sigma,t}(g)\lesssim \int^{\infty}_{0}e^{-r}r^{\sigma-1}\Big(\frac{t^{2}}{4r}\Big)^{-\Q/2}\exp\Big(-c\Big(\frac{d(e,g)^{2}}{t^{2}/(4r)}\Big)\Big)dr.$$
Letting $r(1+d(e,g)^{2}/t^{2})=u$, we  can obtain
\begin{eqnarray*}
  P_{\sigma,t}(g) &\lesssim&\int^{\infty}_{0}e^{-(r+d(e,g)^{2}r/t^{2})}r^{\sigma-1+\Q/2}t^{-\Q}dr  \\
   &\lesssim& \frac{t^{-\Q}}{(1+d(e,g)^{2}/t^{2})^{\sigma+\Q/2}}\int_{0}^{\infty}e^{-u}u^{\sigma-1+\Q/2}du  \\
   &\lesssim& \frac{t^{2\sigma}}{(t^{2}+d(e,g)^{2})^{\Q/2+\sigma}}.
\end{eqnarray*}

For lower bounds of $P_{\sigma,t}$, we also have
$$P_{\sigma,t}(g)\gtrsim\int^{1}_{0}e^{-r}r^{\sigma-1}t^{-\Q}r^{\Q/2}
\exp\Big(-c\Big(\frac{d(e,g)^{2}}{t^{2}/(4r)}\Big)\Big)dr.$$
Let $r\Big(1+d(e,g)^{2}/t^{2}\Big)=u$. It holds
\begin{eqnarray*}
 P_{\sigma,t}(g)
   &\gtrsim&t^{-\Q}\frac{1}{(1+d(e,g)^{2}/t^{2})^{\Q/2+\sigma}}\int_{0}^{1}e^{-cu}u^{\Q/2+\sigma-1}du \\
   &\gtrsim& \frac{t^{2\sigma}}{(t^{2}+d(e,g)^{2})^{\Q/2+\sigma}}.
\end{eqnarray*}

(ii) It follows from (\ref{Poisson}) and Lemma \ref{holder} that
\begin{eqnarray*}
  &&\Big|P_{\sigma,t}(gh)-P_{\sigma,t}(g)\Big|\\ &&\sim\Big|\int^{\infty}_{0}e^{-r}\Big(e^{-t^{2}\L/(4r)}(gh)-e^{-t^{2}\L/(4r)}(g)\Big)\frac{dr}{r^{1-\sigma}}\Big|\\
   &&\lesssim \int^{\infty}_{0}e^{-r}\Big(\frac{d(e,h)}{(t^{2}/(4r))^{1/2}}\Big)\Big(\frac{t^{2}}{4r}\Big)^{-\Q/2}
   \exp\Big(-c\Big(\frac{d(e,g)^{2}}{t^{2}/(4r)}\Big)\Big)\frac{dr}{r^{1-\sigma}}.
\end{eqnarray*}
Then the proof is similar to that of (i), so we omit it.
\end{proof}

\begin{remark}\label{333}
\item{(i)} The condition $d(e,h)<d(e,g)/2$ can be replaced by $d(e,h)<t$. Since $\{P_{t,\sigma}\}_{t>0}$ does not satisfy the semigroup property, we cannot prove it by the method of Lemma \ref{le-2.1}.
In fact, assume that $d(e,h)<t$. The following two cases should be considered.

{\it Case 1: $d(e,h)<t<d(e,g)/2$ or $d(e,h)<d(e,g)/2<t$}. By Proposition \ref{pro-1}, (\ref{eq-3.2}) still holds.

{\it Case 2: $d(e,g)/2<d(e,h)<t$.} Split
$$\Big|P_{\sigma,t}(gh)-P_{\sigma,t}(g)\Big|\lesssim M_{1}+M_{2},$$
where
$$\left\{\begin{aligned}
M_{1}&:=\int_{\{r: (t^{2}/(4r))>d(e,h)^{2}\}}e^{-r}\Big|e^{-t^{2}\L/(4r)}(gh)-e^{-t^{2}\L/(4r)}(g)\Big|\frac{dr}{r^{1-\sigma}};\\
M_{2}&:=\int_{\{r: (t^{2}/(4r))\leq
d(e,h)^{2}\}}e^{-r}\Big|e^{-t^{2}\L/(4r)}(gh)-e^{-t^{2}\L/(4r)}(g)\Big|\frac{dr}{r^{1-\sigma}}.
\end{aligned}\right.$$
For $M_{1}$, by Lemma \ref{holder},
\begin{eqnarray*}
M_{1}&\lesssim& \int^{\infty}_{0}e^{-r}\Big(\frac{d(e,h)}{(t^{2}/(4r))^{1/2}}\Big)\Big(\frac{t^{2}}{4r}\Big)^{-\Q/2}\frac{dr}{r^{1-\sigma}}\\
   &\lesssim& \Big(\frac{d(e,h)}{t}\Big)t^{-\Q}\int^{\infty}_{0}e^{-r}r^{1/2+\Q/2+\sigma-1}dr\\
   &\lesssim& \Big(\frac{d(e,h)}{t}\Big)t^{-\Q}.
\end{eqnarray*}

For $M_{2}$, we further divide $M_{2}$ as $M_{2}\lesssim M_{2,1}+M_{2,2}$, where
$$\left\{\begin{aligned}
M_{2,1}&:=\int_{\{r: (t^{2}/(4r))\leq d(e,h)^{2}\}}e^{-r}\Big|e^{-t^{2}\L/(4r)}(gh)\Big|\frac{dr}{r^{1-\sigma}};\\
M_{2,2}&:=\int_{\{r: (t^{2}/(4r))\leq
d(e,h)^{2}\}}e^{-r}\Big|e^{-t^{2}\L/(4r)}(g)\Big|\frac{dr}{r^{1-\sigma}}.
\end{aligned}\right.$$

Noting that $(t^{2}/4r)\leq d(e,h)^{2}$, we can deduce from Lemma
\ref{upper} that
\begin{eqnarray*}
M_{2,1}&\lesssim&\int^{\infty}_{0}e^{-r}r^{\sigma-1}\Big(\frac{t^{2}}{4r}\Big)^{-\Q/2}
\exp\Big(-c\Big(\frac{d(e,gh)^{2}}{t^{2}/(4r)}\Big)\Big)dr\\
&\lesssim&\int^{\infty}_{0}e^{-r}r^{\sigma-1}\Big(\frac{d(e,h)}{(t^{2}/(4r))^{1/2}}\Big)\Big(\frac{t^{2}}{4r}\Big)^{-\Q/2}dr\\
&\lesssim& \Big(\frac{d(e,h)}{t}\Big)t^{-\Q}\int^{\infty}_{0}e^{-r}r^{1/2+\Q/2+\sigma-1}dr\\
   &\lesssim& \Big(\frac{d(e,h)}{t}\Big)t^{-\Q}.
\end{eqnarray*}
The term $M_{2,2}$ can be dealt with similar to the term $M_{2,1}$.
Then it can be deduced from the condition $d(e,g)/2<t$
that
\begin{eqnarray*}
\Big|P_{\sigma,t}(gh)-P_{\sigma,t}(g)\Big|&\lesssim& \Big(\frac{d(e,h)}{t}\Big)t^{-\Q}\\
&\lesssim&\Big(\frac{d(e,h)}{t}\Big) \frac{t^{2\sigma}}{(t^{2}+d(e,g)^{2})^{\Q/2+\sigma}}.
\end{eqnarray*}

\item{(ii)} Similarly to the proof of Proposition \ref{continuous}, it can be proved that
function $g\mapsto P_{\sigma,t}u(g)$ is continuous on $\X$ if $p\in[1,\infty]$ and $u\in L^{p}(\X)$.

\end{remark}

\subsection{Fractional powers of the sub-Laplace operator}\label{sec-2.3}

According to \cite{G18},  the fractional power $\mathcal{L}^{s}$
  can be defined as follows.
\begin{definition}\label{NO}
Let $s\in(0,1).$ For any $u\in\mathscr{S}(\X),$ we have
\begin{equation}\label{-G}
\mathcal{L}^{s}u(g)=-\dfrac{s}{\Gamma(1-s)}\int_{0}^{\infty}\big(e^{-t\mathcal{L}}u(g)-u(g)\big)\dfrac{dt}{t^{1+s}}.
\end{equation}
\end{definition}
\begin{remark}\label{re-cc}
\item{(i)} For $s\in(0,1)$, $\alpha\in(0,1)$, $\sigma\in(0,1)$ and
$u\in\mathscr{S}(\X),$ using the formula (\ref{-G}),  we can see
that
\begin{equation}\label{fraheat-1}
 \mathcal{L}^{\mathbf{s}}u(g)=-\dfrac{s}{\Gamma(1-s)}\int_{0}^{\infty}\big(H_{\alpha,t}u(g)-u(g)\big)\dfrac{dt}{t^{1+s}}
\end{equation}
and
\begin{equation}\label{frapoisson-1}
 \mathcal{L}^{\widetilde{\mathbf{s}}}u(g)=-\dfrac{s}{\Gamma(1-s)}\int_{0}^{\infty}\big(P_{\sigma,t}u(g)-u(g)\big)\dfrac{dt}{t^{1+s}},
\end{equation} where   $\mathbf{s} $ and $\widetilde{\mathbf{s}} $ are given in (\ref{eq1.1}).

\item{(ii)} Observe that the right-hand sides of (\ref{fraheat-1}) and (\ref{frapoisson-1}) are  convergent integrals in $L^{p}(\X)$
 for any $p\in[1,\infty].$ We give only the proof of (\ref{fraheat-1}). (\ref{frapoisson-1}) can be proved similarly. In order  to prove
 (\ref{fraheat-1}),
we first  need to verify that
\begin{equation}\label{3.11}\|\mathcal{L}^{\alpha}u\|_{L^{p}(\X)}<\infty. \end{equation}   In fact, we can write $\L^{\alpha}u$ as
$$\int_{0}^{\infty}\big(e^{-t\L}u(g)-u(g)\big)\dfrac{dt}{t^{1+\alpha}}=\int_{0}^{1}\big(e^{-t\L}u(g)-u(g)\big)\dfrac{dt}{t^{1+\alpha}}
+\int_{1}^{\infty}\big(e^{-t\L}u(g)-u(g)\big)\dfrac{dt}{t^{1+\alpha}}.$$
In the first integral,
\begin{eqnarray*}
e^{-t\L}u(g)-u(g)=\int_{0}^{t}\frac{d}{dt'}e^{-t'\L}u(g)dt'
=\int_{0}^{t}\mathcal{L}e^{-t'\L}u(g)dt'=\int_{0}^{t}e^{-t'\L}\mathcal{L}u(g)dt'.
\end{eqnarray*}
Therefore, for any $t\in[0,1],$   using Lemma \ref{heat-1}
(iii), we obtain
\begin{eqnarray*}
\|e^{-t\L}u-u\|_{L^{p}(\X)}\leq \int_{0}^{t}\|e^{-t'\L}\mathcal{L}u\|_{L^{p}(\X)}dt'
\leq t\|\mathcal{L}u\|_{L^{p}(\X)}.
\end{eqnarray*}
This implies that
$$t^{-1-\alpha}\|e^{-t\L}u-u\|_{L^{p}(\X)}\lesssim t^{-\alpha}\in L^{1}(0,1).$$
For the second integral, by a similar method mentioned above, we
have
\begin{eqnarray*}
t^{-1-\alpha}\|e^{-t\L}u-u\|_{L^{p}(\X)}&\leq&
t^{-1-\alpha}(\|e^{-t\L}u\|_{L^{p}(\X)}+\|u\|_{L^{p}(\X)})\\
&\lesssim& \|u\|_{L^{p}(\X)}t^{-1-\alpha}\in L^{1}(1,\infty).
\end{eqnarray*}

We next show that the right-hand side of (\ref{fraheat-1}) is a
convergent integral in $L^{p}(\X)$
 for any $p\in[1,\infty].$ We first   write (\ref{fraheat-1}) as
$$\int_{0}^{\infty}\big(H_{\alpha,t}u(g)-u(g)\big)\dfrac{dt}{t^{1+s}}=\int_{0}^{1}\big(H_{\alpha,t}u(g)-u(g)\big)\dfrac{dt}{t^{1+s}}
+\int_{1}^{\infty}\big(H_{\alpha,t}u(g)-u(g)\big)\dfrac{dt}{t^{1+s}}.$$
In the first integral,
\begin{eqnarray*}
H_{\alpha,t}u(g)-u(g)=\int_{0}^{t}\frac{d}{dt'}H_{\alpha,t'}u(g)dt'
=\int_{0}^{t}\mathcal{L}^{\alpha}H_{\alpha,t'}u(g)dt'=\int_{0}^{t}H_{\alpha,t'}\mathcal{L}^{\alpha}u(g)dt'.
\end{eqnarray*}
Therefore, for any $t\in[0,1],$   using Lemma \ref{fracheat-1} (iii)
and (\ref{3.11}), we obtain
\begin{eqnarray*}
\|H_{\alpha,t}u-u\|_{L^{p}(\X)}\leq
\int_{0}^{t}\|H_{\alpha,t'}\mathcal{L}^{\alpha}u\|_{L^{p}(\X)}dt'
\leq t\|\mathcal{L}^{\alpha}u\|_{L^{p}(\X)}\lesssim t,
\end{eqnarray*}
whih implies that
$$t^{-1-s}\|H_{\alpha,t}u-u\|_{L^{p}(\X)}\lesssim t^{-s}\in L^{1}(0,1).$$
For the second integral, by a similar method mentioned above, we
have
\begin{eqnarray*}
t^{-1-s}\|H_{\alpha,t}u-u\|_{L^{p}(\X)}&\leq&
t^{-1-s}(\|H_{\alpha,t}u\|_{L^{p}(\X)}+\|u\|_{L^{p}(\X)})\\
&\lesssim& \|u\|_{L^{p}(\X)}t^{-1-s}\in L^{1}(1,\infty).
\end{eqnarray*} Then we deduce the desired result.

\item{(iii)} From the proof of \cite[Theorem 4.10]{F75}, $\L^{1/2}$ has the following decomposition
\begin{equation}\label{eq-3.22}
  \L^{1/2}=\sum_{j=1}^{n_{1}}T_{j}\mathbf{X}_{j},
\end{equation}
where $\{T_{j}\}_{j=1}^{n_{1}}$ are bounded operators on $L^{p}(\X)$.

\item{(iv)} Let $\L_{p}$ be the maximal restriction of $\L$ to $L^{p}(\X)$. By \cite[Theorem 3.15]{F75}, $\L_{p}^{\Theta/2}$ is a closed operator on $L^{p}(\X)$, where $\Theta=\mathbf{s}$ or $\Theta=\widetilde{\mathbf{s}}$.
\end{remark}

\begin{definition}\label{def-sobolev}
Let $p\in[1,\infty)$,
$(s,\alpha,\sigma)\in(0,1)\times(0,1)\times(0,1) $,
$\mathbf{s}=s\alpha$ and $\widetilde{\mathbf{s}}=s\sigma$.   The
fractional Sobolev spaces $\mathcal{W}^{2\mathbf{s},p}(\X)$ related
to $\mathcal{L}^{\mathbf{s}}$ is defined by
$$\mathcal{W}^{2\mathbf{s},p}(\X)=\Big\{u\in L^{p}(\X):\ \mathcal{L}^{\mathbf{s}}u\in L^{p}(\X)\Big\}$$
 endowed with the norm
$$\|u\|_{\mathcal{W}^{2\mathbf{s},p}(\X)}:=\|u\|_{L^{p}(\X)}+\|\mathcal{L}^{\mathbf{s}}u\|_{L^{p}(\X)}.$$
Correspondingly, the homogeneous fractional Sobolev space
$\mathcal{\dot{W}}^{2\mathbf{s},p}(\X)$ is defined in the same way
as for $\mathcal{W}^{2\mathbf{s},p}(\X)$ endowed with the norm
$$\|u\|_{\mathcal{\dot{W}}^{2\mathbf{s},p}(\X)}=\|\mathcal{L}^{\mathbf{s}}u\|_{L^{p}(\X)}.$$
The definition of the fractional Sobolev space
$\mathcal{W}^{2\widetilde{\mathbf{s}},p}(\X)$ related to
$\mathcal{L}^{\widetilde{\mathbf{s}}}$ can also be similar to that
of $\mathcal{W}^{2\mathbf{s},p}(\X)$.
\end{definition}

\begin{remark}
\item{(i)} If we also denote the domain of $\mathcal{L}^{\mathbf{s}}$ and
$\mathcal{L}^{\widetilde{\mathbf{s}}}$ in $L^{p}(\X)$ as
$\mathcal{W}^{2\mathbf{s},p}(\X)$ and
$\mathcal{W}^{2\widetilde{\mathbf{s}},p}(\X)$, then operators
$\mathcal{L}^{\mathbf{s}}$ and
$\mathcal{L}^{\widetilde{\mathbf{s}}}$ can be extended to a closed
operator on $\mathcal{W}^{2\mathbf{s},p}(\X)$ and
$\mathcal{W}^{2\widetilde{\mathbf{s}},p}(\X)$, respectively, see
\cite[Lemma 2.1]{B}. Therefore, $\mathcal{W}^{2s,p}(\X)$ endowed
with the norm $\|\cdot\|_{\mathcal{W}^{2s,p}(\X)}$ and
$\mathcal{W}^{2\widetilde{\mathbf{s}},p}(\X)$
 endowed with the norm $\|\cdot\|_{\mathcal{W}^{2\widetilde{\mathbf{s}},p}(\X)}$
become   Banach spaces.

\item{(ii)} From Remark \ref{re-cc} (ii), we know that
$\mathscr{S}(\X)\subset \mathcal{W}^{2\mathbf{s},p}(\X),$ that is,
$\mathcal{L}^{\mathbf{s}}(\mathscr{S}(\X))\subset L^{p}(\X).$ The above conclusion holds true for $\mathcal{W}^{2\widetilde{\mathbf{s}},p}(\X)$ as well.

\end{remark}

Next, we recall some results about  the Riesz potential operators
$\mathscr{I}_{\tilde{\alpha}}$ (see \cite{F75,Z23}).
\begin{definition}\label{def1}
Let $0<\tilde{\alpha}<\mathcal{Q}.$ The Riesz kernel of order
$\tilde{\alpha}$ is defined as
\begin{equation}\label{Ia}
I_{\tilde{\alpha}}(g)=\dfrac{1}{\Gamma({\tilde{\alpha}}/{2})}\int_{0}^{\infty}t^{{\tilde{\alpha}}/{2}-1}e^{-t\L}(g)dt
\quad \forall~ g\in\X.
\end{equation}
\end{definition}

It follows from  \cite{F75}  that the Riesz potential operator
$\mathscr{I}_{\tilde{\alpha}}$ is the inverse of the nonlocal
operator $\mathcal{L}^{\tilde{\alpha}/2},$ which is defined as
\begin{equation}\label{FFF}
\mathscr{I}_{\tilde{\alpha}}u(g)=(u\ast I_{\tilde{\alpha}})(g) :=\int_{\X}u(g')
I_{\tilde{\alpha}}(g'^{-1}g)dg'
\end{equation}for any $u\in \mathscr{S}(\X).$ Moreover, exchanging the
 order of integration in   (\ref{FFF}), we obtain
\begin{equation}\label{R}
\mathscr{I}_{\tilde{\alpha}}u(g)=\dfrac{1}{\Gamma({\tilde{\alpha}}/{2})}\int_{0}^{\infty}t^{{\tilde{\alpha}}/{2}-1}e^{-t\mathcal{L}}u(g)dt
\quad \forall~ u\in\mathscr{S}(\X).
\end{equation}
\begin{remark}
\item{(i)} Form   (\ref{R}), we can see that for
$0<\tilde{\beta},\tilde{\gamma}<\mathcal{Q}$ and
$(\alpha,\sigma)\in(0,1)\times(0,1)$,
$$\mathscr{I}_{\tilde{\beta}}u(g)=\dfrac{1}{\Gamma({\tilde{\alpha}}/{2})}\int_{0}^{\infty}t^{{\tilde{\alpha}}/{2}-1}H_{\alpha,t}u(g)dt$$
and
$$\mathscr{I}_{\tilde{\gamma}}u(g)=\dfrac{1}{\Gamma({\tilde{\alpha}}/{2})}\int_{0}^{\infty}t^{{\tilde{\alpha}}/{2}-1}P_{\sigma,t}u(g)dt,$$
where $\tilde{\beta}=\alpha\tilde{\alpha}$ and $\tilde{\gamma}=\sigma\tilde{\alpha}.$

\item{(ii)} From the spectral analysis, we can also  see that
\begin{equation}\label{eq-3.1-1}
u=\mathscr{I}_{2\mathbf{s}}\L^{\mathbf{s}}u=\L^{\mathbf{s}}\mathscr{I}_{2\mathbf{s}}u.
\end{equation}
The above conclusion (\ref{eq-3.1-1}) holds true for
$\mathscr{I}_{2\widetilde{\mathbf{s}}}$ as well.

\item{(iii)} By  Lemma \ref{holder} and (\ref{Ia}), for $d(e,h)\leq d(e,g)/2$, we obtain
$$|I_{\tilde{\beta}}(gh)-I_{\tilde{\beta}}(g)|\lesssim d(e,g)^{\tilde{\beta}-\mathcal{Q}-1}d(e,h),$$
 which implies $I_{\tilde{\beta}}(\cdot)$
is continuous on $\X$. Then by \cite[Proposition 2.3.2]{ada}, we
know that for a measurable  nonnegative function $u$, $g\mapsto
\mathscr{I}_{\tilde{\beta}}u(g)$ is lower semicontinuous on $\X$.

\end{remark}

Similarly to \cite[Propositions 1.10 \&\ 1.11]{F75}, it can be proved  that the following
Hardy-Littlewood-Sobolev
 inequalities for potential operators $\mathscr{I}_{\tilde{\beta}}$ and $\mathscr{I}_{\tilde{\gamma}}$ are valid,
which play  the  key role in the proof of  Proposition
\ref{relation} below.
\begin{proposition}\label{HLS} Let
$(\alpha,\sigma,\tilde{\alpha})\in(0,1)\times(0,1)\times
(0,\mathcal{Q})$ and assume that $\theta$ be either
$\tilde{\beta}=\alpha\tilde{\alpha}$ or
$\tilde{\gamma}=\sigma\tilde{\alpha}.$ Let $p\in
[1,\mathcal{Q}/\theta)$ and $1/p-1/q=\theta/\mathcal{Q}.$ The
following statements hold:
  \item{\rm (i)} If $p=1,$   there exists a positive constant $C$  such
   that for any $u\in L^{1}(\X)$,
  $$\sup_{\lambda>0}\lambda|\{g\in\X:\ |\mathscr{I}_{\theta}u(g)|>\lambda\}|^{{(\mathcal{Q}-\theta)}/{\mathcal{Q}}}
  \leq C\|u\|_{L^{1}(\X)}.$$
  \item{\rm (ii)} If $p\in(1,\mathcal{Q}/\theta)$,   there exists a positive constant $C$
   such that for any $u\in L^{p}(\X)$,
  $$\|\mathscr{I}_{\theta}u\|_{L^{q}(\X)}\leq C\|u\|_{L^{p}(\X)}.$$

\end{proposition}
\begin{remark}\label{EL1}
By Proposition \ref{HLS} and  Fubini's Theorem, we obtain the
following generalization of (\ref{e-2.7}) for the case
$h=I_{\theta}$. Precisely, let $1\leq p<\Q/\theta$. Then for any
$f\in L^{p}(\X)$ and $u\in \text{Lip}_{c}(\X)$,
$$\int_{\X}f(g) \mathscr{I}_{\theta}u(g) dg=\int_{\X}u(g) \mathscr{I}_{\theta}f(g) dg.$$

\end{remark}

\section{Fractional Sobolev capacities and Riesz capacities}\label{sec-3}

\subsection{Characterizations of fractional Sobolev spaces}\label{sec-3.1}
We first give some results about the dense subspaces of fractional
Sobolev spaces, which are related to the fractional Sobolev
capacity. In the setting of nilpotent Lie groups, the density of the spaces of $C^{\infty}$ functions with compact support in Sobolev type spaces was proved by
Folland in \cite{F75} (see \cite[Theorem 4.5]{F75}). However, in order to investigate the extension problem of
$\mathcal{\dot{W}}^{2\mathbf{s},p}(\X)$ via fractional Sobolev type capacities, for any compact subset $K$, we need to find a approximate sequence  $\{u_{j}\geq
1_{\mathcal{K}}\}\subset C_{c}^{\infty}(\X)$ for any $u\in \mathcal{\dot{W}}^{2\mathbf{s},p}(\X)$ satisfying $u>1_{\mathcal{K}}$. Hence, in the proof of Proposition \ref{dense-1} below,  we apply another a constructive
approach.
\begin{proposition}\label{dense-1}
Let $(s,\alpha,\sigma)\in(0,1)\times(0,1)\times(0,1)$ and
$p\in(1,\infty).$ The indices $\mathbf{s}$ and $\widetilde{\mathbf{s}}$ are
given in (\ref{eq1.1}).

\item{\rm (i)} For any $u\in \mathcal{\dot{W}}^{2\mathbf{s},p}(\X)$, there exists a sequence of functions $\{u_{j}\}_{j=1}^{\infty}\subset C_{c}^{\infty}(\X)$ such that
\begin{equation}\label{u-u}
  \lim_{j\rightarrow\infty}\|u_{j}-u\|_{\mathcal{\dot{W}}^{2\mathbf{s},p}(\X)}=0.
\end{equation}
Especially, if $u\geq 0$, then all the functions $\{u_{j}\}_{j=1}^{\infty}$ in (\ref{u-u}) can be chosen to be non-negative.
Moreover, if $\mathcal{K}\subset\X$ is a compact set and $u\geq 1_{K}$, then the approximation sequence
 $\{u_{j}\}_{j=1}^{\infty}$ in $C_{c}^{\infty}(\X)$ can be chosen to satisfy $u_{j}\geq 1_{\mathcal{K}}$ for all $j\in\mathbb{N}.$

\item{\rm (ii)} For the fractional Sobolev spaces $\mathcal{\dot{W}}^{2\widetilde{\mathbf{s}},p}(\X)$, the above results also hold.
\end{proposition}

\begin{proof}
We only give the proof of (i). The proof of (ii) can be similarly
obtained, we omit the details.  Without loss
of generality, we may assume $u\in
\mathcal{\dot{W}}^{2\mathbf{s},p}(\X)$ such that
$\|u\|_{\mathcal{\dot{W}}^{2\mathbf{s},p}(\X)}=1$. Denote by
$p'={p}/{(p-1)}$ the conjugate number of $p$.

{\it Step 1: Approximation $u\in \mathcal{\dot{W}}^{2\mathbf{s},p}(\X)$ by $C^{\infty}(\X)\cap\mathcal{\dot{W}}^{2\mathbf{s},p}(\X)$-functions.} Suppose that $\tau\in C_{c}^{\infty}(\X)$ is a non-negative smooth radial function on $\X$ satisfying
$$\left\{\begin{aligned}
&\text{ supp }\tau\subset B(e,1);\\
&\int_{\X}\tau(g)dg=1;\\
&\tau_{\epsilon}(\cdot)=\epsilon^{-\mathcal{Q}}\tau(\delta_{\epsilon^{-1}}(\cdot))
\ \text{for all}\ \epsilon\in(0,\infty).
\end{aligned}\right.$$
Clearly,   $\|\tau_{\epsilon}\ast u-u\|_{L^{p}(\X)}\rightarrow 0$ as
$\epsilon\rightarrow 0$. We claim that
\begin{equation}\label{u-u-1}
  \L^{\mathbf{s}}(\tau_{\epsilon}\ast u)=\tau_{\epsilon}\ast(\L^{\mathbf{s}}u)\ in\ \mathscr{S}'(\X).
\end{equation}
Once (\ref{u-u-1}) is proved, we can apply the Young inequality to obtain
$$\|\tau_{\epsilon}\ast(\L^{\mathbf{s}}u)\|_{L^{p}(\X)}\leq \|\tau_{\epsilon}\|_{L^{1}}\|\L^{\mathbf{s}}u\|_{L^{p}(\X)}\sim \|u\|_{\mathcal{\dot{W}}^{2\mathbf{s},p}(\X)}<\infty,$$
and then use the density of $\mathscr{S}(\X)$ in $L^{p'}(\X)=(L^{p}(\X))^{\ast}$ to get
\begin{eqnarray*}
   &&\|\L^{\mathbf{s}}(\tau_{\epsilon}\ast u)-\tau_{\epsilon}\ast (\L^{\mathbf{s}}u)\|_{L^{p}(\X)}  \\
   &&=\sup\Big\{\langle \L^{\mathbf{s}}(\tau_{\epsilon}\ast u)-\tau_{\epsilon}\ast (\L^{\mathbf{s}}u), \varphi\rangle:\ \varphi\in\mathscr{S}(\X),\ \|\varphi\|_{L^{p'}(\X)}\leq 1\Big\}=0,
\end{eqnarray*}
which implies that the identity in (\ref{u-u-1}) also holds in $L^{p}(\X)$, whence
$$\lim_{\epsilon\rightarrow 0}\|\L^{\mathbf{s}}(\tau_{\epsilon}\ast u)-\L^{\mathbf{s}}u\|_{L^{p}(\X)}=\lim_{\epsilon\rightarrow 0}\|\tau_{\epsilon}\ast (\L^{\mathbf{s}}u)-\L^{\mathbf{s}}u\|_{L^{p}(\X)}=0.$$
This shows that $u$ can be approximated by $C^{\infty}$-functions $\tau_{\epsilon}\ast u$ in $\mathcal{\dot{W}}^{2\mathbf{s},p}(\X)$.

It remains to prove (\ref{u-u-1}). For any $\varphi\in\mathscr{S}(\X)$, we have
\begin{eqnarray}\notag
  \langle \L^{\mathbf{s}}(\tau_{\epsilon}\ast u),\varphi\rangle &=& \langle \tau_{\epsilon}\ast u, \L^{\mathbf{s}}\varphi\rangle \\ \notag
   &=& \int_{\X}\tau_{\epsilon}\ast u(g)\L^{\mathbf{s}}\varphi(g)dg \\ \label{u-u-2}
   &=& \int_{\X}\Big(\int_{\X}\tau_{\epsilon}(gg'^{-1})u(g')\L^{\mathbf{s}}\varphi(g)dg'\Big)dg.
\end{eqnarray}
Then we apply Fubini's theorem to obtain
\begin{eqnarray*}
   &&\int_{\X}\Big(\int_{\X}\tau_{\epsilon}(gg'^{-1})u(g')\L^{\mathbf{s}}\varphi(g)dg'\Big)dg  \\
   &&=\int_{\X}\Big(\int_{\X}\tau_{\epsilon}(gg'^{-1})u(g')\L^{\mathbf{s}}\varphi(g)dg\Big)dg'  \\
   &&=\int_{\X}\tau_{\epsilon}\ast(\L^{\mathbf{s}}\varphi)(g')u(g')dg'.
\end{eqnarray*}
Since $\tau_{\epsilon}\ast\varphi\in\mathscr{S}(\X)$, we easily obtain
$$\tau_{\epsilon}\ast (\L^{\mathbf{s}}\varphi)(g')=\L^{\mathbf{s}}(\tau_{\epsilon}\ast\varphi)(g').$$
Combining the last two formulas with (\ref{u-u-2}), we further apply
the fact that $\L^{\mathbf{s}}u\in L^{p}(\X)$ and Fubini's Theorem
to deduce
\begin{eqnarray*}
  \langle\L^{\mathbf{s}}(\tau_{\epsilon}\ast u), \varphi\rangle &=& \int_{\X}\L^{\mathbf{s}}(\tau_{\epsilon}\ast\varphi)(g')u(g')dg'=\langle u,\L^{\mathbf{s}}(\tau_{\epsilon}\ast\varphi)\rangle \\
   &=& \langle\L^{\mathbf{s}}u, \tau_{\epsilon}\ast\varphi\rangle=\langle\tau_{\epsilon}\ast(\L^{\mathbf{s}}u),\varphi\rangle,
\end{eqnarray*}
which gives (\ref{u-u-1}).

{\it Step 2: Approximation $u\in C^{\infty}(\X)\cap\mathcal{\dot{W}}^{2\mathbf{s},p}(\X)$ by $C_{c}^{\infty}(\X)$-functions.} Let $\eta$ be a radial decreasing function in $C_{c}^{\infty}(\X)$ such that
\begin{equation*}\label{u-u-3}
 \left\{\begin{aligned}
&0\leq \eta(g)\leq 1 \ \text{for all}\ g\in \X;\\
&\eta(g)=1\ \text{on}\ B(e,1);\\
&\text{supp}\ \eta\subset B(0,2);\\
&\eta_{N}(\cdot)=\eta(\delta_{N^{-1}}(\cdot))\ \text{for any}\ N>0.
\end{aligned}\right.
\end{equation*}
Clearly, $\{u\eta_{N}\}_{N=1}^{\infty}\subset C_{c}^{\infty}(\X)$ and $\|u\eta_{N}-u\|_{L^{p}(\X)}\rightarrow 0$ as $N\rightarrow\infty$. So it suffices to show
\begin{equation}\label{u-u-4}
  \lim_{N\rightarrow\infty}\|\L^{\mathbf{s}}(u\eta_{N}-u)\|_{L^{p}(\X)}=0.
\end{equation}
To prove (\ref{u-u-4}), for any $g\in\X$, we write
\begin{eqnarray*}\notag
   &&\L^{\mathbf{s}}(u\eta_{N}-u)(g)  \\ \notag
   &&=C\int_{\X}\Big(u(g)(\eta_{N}(g)-1)-u(g')(\eta_{N}(g')-1)\Big)\int_{0}^{\infty}t^{-(1+s)}K_{\alpha,t}(g'^{-1}g)dtdg'  \\ \notag
   &&=C\Big((\eta_{N}(g)-1)\int_{\X}(u(g)-u(g'))\int_{0}^{\infty}t^{-(1+s)}K_{\alpha,t}(g'^{-1}g)dtdg'\\ \notag
   &&\quad +\int_{\X}u(g')(\eta_{N}(g)-\eta_{N}(g'))\int_{0}^{\infty}
   t^{-(1+s)}K_{\alpha,t}(g'^{-1}g)dtdg'\Big)  \\ \label{u-u-5}
   &&=(\eta_{N}(g)-1)\L^{\mathbf{s}}u(g)+C\int_{\X}u(g')(\eta_{N}(g)-\eta_{N}(g'))\int_{0}^{\infty}t^{-(1+s)}K_{\alpha,t}(g'^{-1}g)dtdg'.\
\end{eqnarray*}
Via applying $\L^{\mathbf{s}}u\in L^{p}(\X)$, we obtain
$$\lim_{N\rightarrow\infty}\|(\eta_{N}-1)\L^{\mathbf{s}}u\|_{L^{p}(\X)}\leq \lim_{N\rightarrow\infty}\Big(\int_{d(e,g)>N}|\L^{\mathbf{s}}u(g)|^{p}dg\Big)^{1/p}=0.$$
Consequently, (\ref{u-u-4}) follows from verifying
\begin{equation}\label{u-u-6}
  \lim_{N\rightarrow\infty}\int_{\X}|u(g')||\eta_{N}(g)-\eta_{N}(g')|\int_{0}^{\infty}t^{-(1+s)}K_{\alpha,t}(g'^{-1}g)dtdg'=0\ in\ L^{p}(\X).
\end{equation}

In order to prove (\ref{u-u-6}), for any $g\in \X$, we use the
stratified  mean value theorem and properties of $\eta$ to find
\begin{eqnarray*}
   &&\int_{\X}|\eta(g)-\eta(g')|\int_{0}^{\infty}t^{-(1+s)}K_{\alpha,t}(g'^{-1}g)dtdg'  \\
   &&\leq \|\nabla_{\X}\eta\|_{\infty}\int_{d(g,g')<1}d(g,g')^{1-2\alpha s-\mathcal{Q}}dg'+\int_{d(g,g')\geq 1}d(g,g')^{-2\alpha s-\mathcal{Q}}dg'\lesssim 1,
\end{eqnarray*}
whence
\begin{eqnarray*}\notag
   &&\int_{\X}|\eta_{N}(g)-\eta_{N}(g')|\int_{0}^{\infty}t^{-(1+s)}K_{\alpha,t}(g'^{-1}g)dtdg'  \\ \label{u-u-7}
   &&\lesssim N^{-2\alpha s}\int_{\X}\frac{|\eta(\delta_{N^{-1}}(g))-\eta(g')|}{d(\delta_{N^{-1}}(g),g')^{\mathcal{Q}+2\alpha s}}dg'\lesssim N^{-2\alpha s}.
\end{eqnarray*}
The last estimation, along with   H\"{o}lder's inequality, derives
\begin{eqnarray*}
  &&\int_{\X}|u(g')||\eta_{N}(g)-\eta_{N}(g')|\int_{0}^{\infty}t^{-(1+s)}K_{\alpha,t}(g'^{-1}g)dtdg' \\ &&\lesssim\int_{\X}\frac{|u(g')||\eta_{N}(g)-\eta_{N}(g')|}{d(g,g')^{\mathcal{Q}+2\alpha s}}dg' \\
   &&\lesssim\Big(\int_{\X}\frac{|u(g')|^{p}|\eta_{N}(g)-\eta_{N}(g')|}{d(g,g')^{\mathcal{Q}+2\alpha s}}dg'\Big)^{1/p}
   \Big(\int_{\X}\frac{|\eta_{N}(g)-\eta_{N}(g')|}{d(g,g')^{\mathcal{Q}+2\alpha s}}dg'\Big)^{1/p'}   \\
   &&\lesssim N^{-2\alpha s/p'}\Big(\int_{\X}\frac{|u(g')|^{p}|\eta_{N}(g)-\eta_{N}(g')|}{d(g,g')^{\mathcal{Q}+2\alpha s}}dg'\Big)^{1/p},
\end{eqnarray*}
which implies
\begin{eqnarray*}\notag
   && \Big(\int_{\X}\Big(\int_{\X}|u(g')||\eta_{N}(g)-\eta_{N}(g')|\int_{0}^{\infty}t^{-(1+s)}K_{\alpha,t}(g'^{-1}g)dtdg'\Big)^{p}dg\Big)^{1/p} \\ \notag
   &&\lesssim N^{-2\alpha s/p'}\Big(\int_{\X}\int_{\X}\frac{|u(g')|^{p}|\eta_{N}(g)-\eta_{N}(g')|}{d(g,g')^{\mathcal{Q}+2\alpha s}}dg'dg\Big)^{1/p}  \\ \label{u-u-8}
   &&\lesssim N^{-2\alpha s}\|u\|_{L^{p}(\X)}\rightarrow 0\ \ \ as\ N\rightarrow\infty.
\end{eqnarray*}
This proves (\ref{u-u-6}).

{\it Step 3: Approximation $u\in \mathcal{\dot{W}}^{2\mathbf{s},p}(\X)$ by $C_{c}^{\infty}(\X)$-functions.} A combination of Steps 1-2 and the triangle inequalities for $\|\cdot\|_{L^{p}(\X)}$ and $\|\cdot\|_{\mathcal{\dot{W}}^{2\mathbf{s},p}(\X)}$ deduces that any $\mathcal{\dot{W}}^{2\mathbf{s},p}(\X)$-function $u$ can be approximated by the $C_{c}^{\infty}(\X)$-functions
$$u_{\epsilon,N}:=(\tau_{\epsilon}\ast u)\eta_{N}$$
under $\|\cdot\|_{L^{p}(\X)}$ and $\|\cdot\|_{\mathcal{\dot{W}}^{2\mathbf{s},p}(\X)}$ respectively. Of course,
if $u$ is non-negative, then the approximation functions $u_{\epsilon,N}$ are also non-negative.

 In particular, by   Step 3 we know that the approximation sequence $\{u_{j}\}_{j=1}^{\infty}$ is of the form
$$u_{\epsilon,N}=(\tau_{\epsilon}\ast u)\eta_{N}.$$
Since $\mathcal{K}$ is compact, we may assume
$\mathcal{K}\subset B(0,N_{0})$ for some $N_{0}\in\mathbb{N}.$
Choose $\epsilon_{0}$ such that
$$0<\epsilon_{0}<\text{dist}\Big\{\partial \mathcal{K},\ \partial(\{g\in\X: u(g)\geq 1\}^{\circ})\Big\}.$$
If $\epsilon\in(0,\epsilon_{0})$ and $N>N_{0}$, then
$$\left\{\begin{aligned}
&\eta_{N}(g)=\eta(\delta_{N^{-1}}(g))=1\ \text{for all}\ g\in \mathcal{K};\\
&u\geq 1\ \ on\ B(g,\epsilon),
\end{aligned}\right.$$
which induces
\begin{eqnarray*}
  u_{\epsilon,N}(g) &=& \eta_{N}(g)\int_{\X}\tau_{\epsilon}(gg'^{-1})u(g')dg'\\
  &=&\int_{B(g,\epsilon)}\tau_{\epsilon}(gg'^{-1})u(g')dg'  \\
   &\geq& \int_{B(g,\epsilon)}\tau_{\epsilon}(gg'^{-1})dg'=1,
\end{eqnarray*}
as desired.

\end{proof}

Next we give the weak gradient and the distributional fractional Sobolev spaces. Before that we need to recall some concepts.
From \cite{Z23} we know that for $0<s<1$, the $s$-gradient $\nabla_{\X}^{s}$ and the $s$-divergence $\text{div}_{\X}^{s}$ are defined as
\begin{eqnarray*}
 \nabla_{\X}^{s}u(g)  &=& \nabla_{\X}(\mathscr{I}_{1-s}u)(g) \\
   &=& \sum_{j=1}^{n_{1}}\Big(\lim_{\varepsilon\rightarrow 0}\int_{
   d(g,g')\geq\varepsilon }u(g')\mathbf{X}_{j}I_{1-s}(g'^{-1}g)dg'\Big)\mathbf{X}_{j}
\end{eqnarray*}
and
\begin{eqnarray*}
  \text{div}_{\X}^{s}\varphi(g) &=& \mathscr{I}_{1-s}(\text{div}_{\X}\varphi)(g) \\
   &=& \sum_{j=1}^{n_{1}} \Big(\lim_{\varepsilon\rightarrow 0}\int_{ d(g,g')\geq\varepsilon }\varphi_{j}(g')\widetilde{\mathbf{X}}_{j}I_{1-s}(g'^{-1}g)dg'
\end{eqnarray*}
for each $u\in C_{c}^{\infty}(\X)$ and $\varphi=\sum_{j=1}^{n_{1}}\varphi_{j}\mathbf{X}_{j}\in C_{c}^{\infty}(\X;H\X),$ respectively.

Following the proof of \cite[Proposition 3]{Z23} and the definition of $\nabla_{\X}^{s}$ and $\text{div}_{\X}^{s}$, we have
\begin{proposition}
Let $(\alpha,\sigma,s)\in(0,1)\times(0,1)\times(0,1)$,
$\mathbf{s}=s\alpha$ and $\widetilde{\mathbf{s}}=s\sigma$. For each
$u\in C_{c}^{\infty}(\X)$ and
 $\varphi=\sum_{j=1}^{n_{1}}\varphi_{j}\mathbf{X}_{j}\in
 C_{c}^{\infty}(\X;H\X)$,  we have

\item{\rm (i)} \begin{eqnarray*}
 \nabla_{\X}^{\mathbf{s}}u(g)  &=& \nabla_{\X}(\mathscr{I}_{1-\mathbf{s}}u)(g) \\
   &=& \sum_{j=1}^{n_{1}}\Big(\lim_{\varepsilon\rightarrow 0}\int_{\{d(g,g')\geq
   \varepsilon\}}u(g')\mathbf{X}_{j}I_{1-\mathbf{s}}(g'^{-1}g)dg'\Big)\mathbf{X}_{j};
\end{eqnarray*}

\begin{eqnarray*}
  \mathrm{ {div}}_{\X}^{\mathbf{s}}\varphi(g) &=& \mathscr{I}_{1-\mathbf{s}}( \mathrm{{div}}_{\X}\varphi)(g) \\
   &=& \sum_{j=1}^{n_{1}} \Big(\lim_{\varepsilon\rightarrow 0}\int_{\{d(g,g')\geq\varepsilon\}}\varphi_{j}(g')\widetilde{\mathbf{X}}_{j}I_{1-\mathbf{s}}(g'^{-1}g)dg'.
\end{eqnarray*}

\item{\rm (ii)}
\begin{eqnarray*}
 \nabla_{\X}^{\widetilde{\mathbf{s}}}u(g)  &=& \nabla_{\X}(\mathscr{I}_{1-\widetilde{\mathbf{s}}}u)(g) \\
   &=& \sum_{j=1}^{n_{1}}\Big(\lim_{\varepsilon\rightarrow 0}\int_{\{d(g,g')\geq\varepsilon\}}u(g')\mathbf{X}_{j}I_{1-\widetilde{\mathbf{s}}}(g'^{-1}g)dg'\Big)\mathbf{X}_{j};
\end{eqnarray*}

\begin{eqnarray*}
  \mathrm{div}_{\X}^{\widetilde{\mathbf{s}}}\varphi(g) &=& \mathscr{I}_{1-\widetilde{\mathbf{s}}}(\mathrm{div}_{\X}\varphi)(g) \\
   &=& \sum_{j=1}^{n_{1}} \Big(\lim_{\varepsilon\rightarrow 0}\int_{\{d(g,g')\geq\varepsilon\}}\varphi_{j}(g')\widetilde{\mathbf{X}}_{j}I_{1-\widetilde{\mathbf{s}}}(g'^{-1}g)dg'.
\end{eqnarray*}
\end{proposition}
\begin{remark}\label{KLA}
\item{(i)} Recall that the Riesz transform on $\X$ is defined by $\mathcal{R}:=\nabla_{\X}\L^{-1/2}$. And we may write
$$\mathcal{R}u=\sum_{j=1}^{n_{1}}(\mathcal{R}_{j}u)\mathbf{X}_{j}\ \ \ \forall\ u\in C_{c}^{\infty}(\X),$$
where $\mathcal{R}_{j}=\mathbf{X}_{j}\L^{-1/2}$. For every $1<p<\infty$, $\mathcal{R}$ is bounded from  $L^{p}(\X)$ to $L^{p}(\X;H\X)$, see \cite{F75}.

\item{(ii)} From the definition of the Riesz transform, $\nabla_{\X}^{\mathbf{s}}$ has the following expression: for any $u\in C_{c}^{\infty}(\X)$,
$$\nabla_{\X}^{\mathbf{s}}u=\nabla_{\X}\mathscr{I}_{1-\mathbf{s}}u=\mathcal{R}(\L^{\mathbf{s}/2}u).$$
Consequently, $\nabla_{\X}^{\mathbf{s}}$ can be extended to the fractional Sobolev space $\mathcal{W}^{\mathbf{s},p}(\X)$, where $1<p<\infty$. For every $u\in\mathcal{W}^{\mathbf{s},p}(\X)$, it is reasonable to define
$$\nabla_{\X}^{\mathbf{s}}u:=\mathcal{R}(\L^{\mathbf{s}/2}u)\in L^{p}(\X,H\X),$$
where we have used the boundedness of the Riesz transform on $L^{p}(\X)$ and $\L^{\mathbf{s}/2}u\in L^{p}(\X)$. The above conclusion holds true for $\nabla_{\X}^{\widetilde{\mathbf{s}}}$ as well.

\item{(iii)}  By integration by parts and Remark \ref{EL1}, for each $u\in \text{Lip}_{c}(\X)$ and $\phi\in \text{Lip}_{c}(\X;H\X)$,
$$\int_{\X}u\mathrm{div}_{\X}^{\Theta}\phi=-\int_{\X}\langle\phi,\nabla_{\X}^{\Theta}u\rangle,$$
where $\Theta=\mathbf{s}$ or $\Theta=\widetilde{\mathbf{s}}$.
\end{remark}

\begin{definition}
Let $0<s<1$ and $1\leq p\leq\infty$. For each $u\in L^{p}(\X)$, $f\in L_{loc}^{1}(\X;H\X)$ is said to be a weak $s$-gradient of $u$, denoted by $\nabla_{\X,w}^{s}u,$ if
\begin{equation*}
  \int_{\X}u\mathrm{div}^{s}_{\X}\phi=-\int_{\X}\langle\phi,f\rangle
\end{equation*}
holds for any $\phi\in C_{c}^{\infty}(\X;H\X).$
\end{definition}

From this we can define the distributional fractional Sobolev space $\mathbb{W}^{\mathbf{s},p}(\X)$.
$$\mathbb{W}^{\mathbf{s},p}(\X):=\Big\{u\in L^{p}(\X)\ \text{with a weak}\
 \mathbf{s}-\text{gradient}\ \nabla_{\X,w}^{\mathbf{s}}u\in L^{p}(\X;H\X)\Big\} $$
endowed with the norm
\begin{equation}\label{ffs}
\|u\|_{\mathbb{W}^{\mathbf{s},p}(\X)}:=\|u\|_{L^{p}(\X)}+\|\nabla_{\X,w}^{\mathbf{s}}u\|_{L^{p}(\X;H\X)}.
\end{equation}
The homogeneous distributional fractional Sobolev space $\mathbb{\dot{W}}^{\mathbf{s},p}(\X)$ is defined in the same way as for $\mathcal{\dot{W}}^{\mathbf{s},p}(\X)$.

If $\nabla_{\X,w}^{\mathbf{s}}$ is replaced by $\nabla_{\X,w}^{\widetilde{\mathbf{s}}}$, the above space is denoted as $\mathbb{W}^{\widetilde{\mathbf{s}},p}(\X)$.

\begin{proposition}\label{w-w}
Let $(\alpha,\sigma,s)\in(0,1)\times(0,1)\times(0,1)$,
$\mathbf{s}=s\alpha$ and $\widetilde{\mathbf{s}}=s\sigma$.

\item{\rm (i)} For $1<p<\Q/(1-\mathbf{s})$, $\mathbb{W}^{\mathbf{s},p}(\X)=\mathcal{W}^{\mathbf{s},p}(\X).$

\item{\rm (ii)} For $1<p<\Q/(1-\widetilde{\mathbf{s}})$, $\mathbb{W}^{\widetilde{\mathbf{s}},p}(\X)=\mathcal{W}^{\widetilde{\mathbf{s}},p}(\X).$

\end{proposition}
\begin{proof}
We only give the proof of (i). Firstly, it is easy to prove that
$\mathbb{W}^{\mathbf{s},p}(\X)$ is a Banach space. Therefore, the
proof of this theorem is divided into the following three steps.

{\it Step 1: $\mathcal{W}^{\mathbf{s},p}(\X)\subset
\overline{C_{c}^{\infty}(\X)}^{\|\cdot\|_{\mathbb{W}^{\mathbf{s},p}(\X)}}$.}
In fact, it follows  from Proposition \ref{dense-1} that for each
$u\in\mathcal{W}^{\mathbf{s},p}(\X)$, we can find
$\{u_{j}\}_{j=1}^{\infty}\subset C_{c}^{\infty}(\X)$ such that
$$u_{j}\rightarrow u\ \text{in}\ L^{p}(\X)\ \ \ \text{and}\ \ \ \L^{\mathbf{s}/2}u_{j}\rightarrow\L^{\mathbf{s}/2}u\ \text{in}\ L^{p}(\X)\ \text{as}\ j\rightarrow\infty.$$
Since $\nabla_{\X}^{\mathbf{s}}$ and $\nabla_{\X,w}^{\mathbf{s}}$ coincide on $C_{c}^{\infty}(\X)$, by Remark \ref{KLA} (ii), we have
$$\nabla_{\X,w}^{\mathbf{s}}u_{j}=\nabla_{\X}^{\mathbf{s}}u_{j}=\mathcal{R}(\L^{\mathbf{s}/2}u_{j}).$$
Then from the boundedness of the Riesz transform on $L^{p}(\X)$, we
deduce that
$$\nabla_{\X,w}^{\mathbf{s}}u_{j}\rightarrow\mathcal{R}(\L^{\mathbf{s}/2}u)=\nabla_{\X}^{\mathbf{s}}u\ \text{in}\ L^{p}(\X; H\X)\ \text{as}\ j\rightarrow\infty.$$
Next, we will prove that $u_{j}\rightarrow u$ in $\mathbb{W}^{\mathbf{s},p}(\X)$ as $j\rightarrow\infty$. In order to do this,
it suffices to indicate that $\nabla_{\X,w}^{\mathbf{s}}u=\nabla_{\X}^{\mathbf{s}}u.$ Indeed, for any $\phi\in C_{c}^{\infty}(\X;H\X)$,
 using Remark \ref{KLA} (iii), we obtain
\begin{eqnarray*}
  \int_{\X}u\mathrm{div}_{\X}^{\mathbf{s}}\phi &=& \lim_{j\rightarrow\infty}\int_{\X}u_{j}\mathrm{div}_{\X}^{\mathbf{s}}\phi=-\lim_{j\rightarrow\infty}
  \int_{\X}\langle\phi,\nabla_{\X}^{\mathbf{s}}u_{j}\rangle  \\
   &=& -\int_{\X}\langle\phi,\nabla_{\X}^{\mathbf{s}}u\rangle.
\end{eqnarray*}

{\it Step 2:
$\overline{\mathcal{W}^{\mathbf{s},p}(\X)}^{\|\cdot\|_{\mathbb{W}^{\mathbf{s},p}(\X)}}=\mathcal{W}^{\mathbf{s},p}(\X)$.}
For every
$u\in\overline{\mathcal{W}^{\mathbf{s},p}(\X)}^{\|\cdot\|_{\mathbb{W}^{\mathbf{s},p}(\X)}}$,
it is easy to see that  we can find
$\{u_{j}\}_{j=1}^{\infty}\subset\mathcal{W}^{\mathbf{s},p}(\X)$ such
that
$$u_{j}\rightarrow u\ \text{in}\ L^{p}(\X)\ \ \ \text{and}\ \ \ \nabla_{\X,w}^{\mathbf{s}}u_{j}
\rightarrow\nabla_{\X,w}^{\mathbf{s}}u\ \text{in}\ L^{p}(\X;H\X)\ \text{as}\ j\rightarrow\infty.$$
Since $\{u_{j}\}_{j=1}^{\infty}\subset\mathcal{W}^{\mathbf{s},p}(\X)$, as already explained in Step 1,
$$\nabla_{\X,w}^{\mathbf{s}}u_{j}=\nabla_{\X}^{\mathbf{s}}u_{j}=\mathcal{R}(\L^{\mathbf{s}/2}u_{j}),$$
thus $\{\mathcal{R}_{k}\L^{\mathbf{s}/2}u_{j}\}_{j=1}^{\infty}$ forms a Cauchy sequence in $L^{p}(\X)$ for each $k=1,\ldots,n_{1}$.

The decomposition (\ref{eq-3.22}) allows us to write the
 identity operator on $L^{p}(\X)$ as
 $\sum_{k=1}^{n_{1}}T_{k}\mathcal{R}_{k}$ due to
 the boundedness of the Riesz transform on $L^{p}(\X)$ and a standard approximation argument. Therefore,
  we  can also see that $\{\L^{\mathbf{s}/2}u_{j}\}_{j=1}^{\infty}$ is a Cauchy sequence in $L^{p}(\X)$. In fact, for every $j,l\in\mathbb{N}$,
\begin{eqnarray*}
  \|\L^{\mathbf{s}/2}u_{j}-\L^{\mathbf{s}/2}u_{l}\|_{L^{p}(\X)} &=& \Big\|\sum_{k=1}^{n_{1}}T_{k}(\mathcal{R}_{k}\L^{\mathbf{s}/2}u_{j}-\mathcal{R}_{k}\L^{\mathbf{s}/2}u_{l})\Big\|_{L^{p}(\X)}  \\
   &\lesssim& \sum_{k=1}^{n_{1}}\|\mathcal{R}_{k}\L^{\mathbf{s}/2}u_{j}-\mathcal{R}_{k}\L^{\mathbf{s}/2}u_{l}\|_{L^{p}(\X)}.
\end{eqnarray*}
Consequently there exists a function $f\in L^{p}(\X)$ such that
$$\L_{p}^{\mathbf{s}/2}u_{j}=\L^{\mathbf{s}/2}u_{j}\rightarrow f\ \text{in}\ L^{p}(\X)\ \text{as}\ j\rightarrow\infty,$$
where $\L_{p}$ is defined as Remark \ref{re-cc} (iv). Then we have $f=\L_{p}^{\mathbf{s}/2}u=\L^{\mathbf{s}/2}u.$ Thus, $u\in\mathcal{W}^{\mathbf{s},p}(\X)$.

{\it Step 3: $\mathbb{W}^{\mathbf{s},p}(\X)\subset\mathcal{W}^{\mathbf{s},p}(\X)$.}
Let $\{\tau_{\epsilon}\}_{\epsilon>0}$ be a family of standard mollifiers in Proposition \ref{dense-1}. Then for each $u\in\mathbb{W}^{\mathbf{s},p}(\X)$ and $\phi\in C_{c}^{\infty}(\X;H\X)$, using (\ref{e-2.6}), (\ref{e-2.7}) and Remark \ref{EL1},
\begin{eqnarray*}
 \int_{\X}\tau_{\epsilon}\ast u\mathrm{div}_{\X}^{\mathbf{s}}\phi  &=&
  \int_{\X}\tau_{\epsilon}\ast u\mathscr{I}_{1-\mathbf{s}}(\mathrm{div}_{\X}\phi)=\int_{\X}\mathrm{div}_{\X}\phi
  \mathscr{I}_{1-\mathbf{s}}(\tau_{\epsilon}\ast u)  \\
   &=&\int_{\X}\mathrm{div}_{\X}\phi(\tau_{\epsilon}\ast \mathscr{I}_{1-\mathbf{s}}u)=\int_{\X}\mathscr{I}_{1-\mathbf{s}}u(\tau_{\epsilon}\ast\mathrm{div}_{\X}\phi) \\
   &=&\int_{\X}\mathscr{I}_{1-\mathbf{s}}u\mathrm{div}_{\X}(\tau_{\epsilon}\ast \phi)=\int_{\X}u\mathrm{div}_{\X}^{\mathbf{s}}(\tau_{\epsilon}\ast \phi)  \\
   &=&-\int_{\X}\langle\tau_{\epsilon}\ast \phi,\nabla_{\X,w}^{\mathbf{s}}u\rangle=-\int_{\X}\langle\phi,\tau_{\epsilon}\ast\nabla_{\X,w}^{\mathbf{s}}u\rangle.
\end{eqnarray*}
Thus $\nabla_{\X,w}^{\mathbf{s}}(\tau_{\epsilon}\ast
u)=\tau_{\epsilon}\ast\nabla_{\X,w}^{\mathbf{s}}u.$ Then by the
approximation property of the mollifier, we conclude  that the set
$\Lambda:=\{\tau_{\epsilon}\ast u;
u\in\mathbb{W}^{\mathbf{s},p}(\X),\epsilon>0\}$ is dense in
$\mathbb{W}^{\mathbf{s},p}(\X)$.

We claim that
\begin{equation}\label{eq-ii}
  \Lambda\subset\mathcal{W}^{\mathbf{s},p}(\X).
\end{equation}
Once (\ref{eq-ii}) is proved, combining with the preceding two steps, we have the following relationship
$$\mathbb{W}^{\mathbf{s},p}(\X)=\overline{\Lambda}^{\|\cdot\|_{\mathbb{W}^{\mathbf{s},p}(\X)}}\subset
\overline{\mathcal{W}^{\mathbf{s},p}(\X)}^{\|\cdot\|_{\mathbb{W}^{\mathbf{s},p}(\X)}}=
\mathcal{W}^{\mathbf{s},p}(\X)\subset \overline{C_{c}^{\infty}(\X)}^{\|\cdot\|_{\mathbb{W}^{\mathbf{s},p}(\X)}}\subset\mathbb{W}^{\mathbf{s},p}(\X),$$
which shows that $\mathbb{W}^{\mathbf{s},p}(\X)=\mathcal{W}^{\mathbf{s},p}(\X).$

 It remains to prove (\ref{eq-ii}). By the integration by parts and Remark \ref{EL1}, we obtain
 \begin{eqnarray*}
   \int_{\X}u_{\epsilon}\mathrm{div}_{\X}^{\mathbf{s}}\phi &=& \int_{\X}u_{\epsilon}\mathscr{I}_{1-\mathbf{s}}(\mathrm{div}_{\X}\phi)=\int_{\X}\mathrm{div}_{\X}\phi\mathscr{I}_{1-\mathbf{s}}u_{\epsilon} \\
    &=& -\int_{\X}\langle\phi,\nabla_{\X}\mathscr{I}_{1-\mathbf{s}}u_{\epsilon}\rangle,
 \end{eqnarray*}
where $$u_{\epsilon}:=\tau_{\epsilon}\ast u\in\Lambda\subset
C^{\infty}(\X)\cap L^{p}(\X)$$ and we have used the fact that
$$\mathscr{I}_{1-\mathbf{s}}u_{\epsilon}=\tau_{\epsilon}\ast\mathscr{I}_{1-\mathbf{s}}u\in
C^{\infty}(\X).$$

Since $\nabla_{\X}\mathscr{I}_{1-\mathbf{s}}u_{\epsilon}=\nabla_{\X,w}^{\mathbf{s}}u_{\epsilon}\in L^{p}(\X;H\X)$, $\mathbf{X}_{k}(\mathscr{I}_{1-\mathbf{s}}u_{\epsilon})\in L^{p}(\X),\ k=1,\ldots,n_{1}.$ Then from (\ref{eq-3.22}),
$$\L^{1/2}(\mathscr{I}_{1-\mathbf{s}}u_{\epsilon})=\sum_{k=1}^{n_{1}}T_{k}\mathbf{X}_{k}(\mathscr{I}_{1-\mathbf{s}}u_{\epsilon})\in L^{p}(\X).$$
Notice that $\L^{\mathbf{s}/2}u_{\epsilon}=\L^{1/2}(\mathscr{I}_{1-\mathbf{s}}u_{\epsilon}).$ In conclusion, $u_{\epsilon}\in\mathcal{W}^{\mathbf{s},p}(\X).$

\end{proof}

\subsection{Fractional Sobolev-Riesz capacities}\label{sec-3.2}
In this section, we only need to investigate  the fractional Sobolev
capacity associated with $\L^{\mathbf{s}}$. The fractional Sobolev
capacity related to $\L^{\widetilde{\mathbf{s}}}$ can be handled
similarly.
\begin{definition}\label{cap-1}
Let $(s,\alpha)\in(0,1)\times(0,1)$,  $p\in[1,\infty)$ and
$\mathbf{s}=s\alpha $. The fractional Sobolev capacity of a compact
set $\mathcal{K}\subset\X$, denoted by
$Cap_{\mathcal{\dot{W}}^{2\mathbf{s},p}}(\mathcal{K})$, is defined
as
$$Cap_{\mathcal{\dot{W}}^{2\mathbf{s},p}}(\mathcal{K}):=\inf\Big\{\|u\|_{\mathcal{\dot{W}}^{2\mathbf{s},p}(\X)}^{p}:\ u\in C_{c}^{\infty}(\X),\ u\geq 1_{\mathcal{K}}\Big\}.$$
\end{definition}
Then we can extend Definition \ref{cap-1} to a general set via the following manner:

$\star$ If $O\subset\X$ is open, then we set
$$Cap_{\mathcal{\dot{W}}^{2\mathbf{s},p}}(O):=\sup\Big\{Cap_{\mathcal{\dot{W}}^{2\mathbf{s},p}}(\mathcal{K}):\ \text{compact}\ \mathcal{K}\subset O\Big\}.$$

$\star$ For an arbitrary set $E\subset\X$, define
$$Cap_{\mathcal{\dot{W}}^{2\mathbf{s},p}}(E):=\inf\Big\{Cap_{\mathcal{\dot{W}}^{2\mathbf{s},p}}(O):\ \text{open}\ O\supset E\Big\}.$$








Subsequently, we define the Riesz capacity, which is constructed
utilizing the Riesz potential operator as presented in (\ref{FFF}).
This capacity   is equivalent to the fractional Sobolev capacity in
Definition \ref{cap-1} and we will then proceed to prove the
properties associated with this capacity.

\begin{definition}\label{cap-2}
Let $(s,\alpha)\in(0,1)\times(0,1)$, $p\in[1,\infty)$ and
$\mathbf{s}=s\alpha $. For any set $E\subset\X$, the Riesz capacity
is defined as
$$Cap_{\mathbf{s},p}(E):=\inf\Big\{\|f\|_{L^{p}(\X)}^{p}:\ 0\leq f\in L^{p}(\X),\ \mathscr{I}_{2\mathbf{s}}f\geq 1\ \text{on}\ E\Big\}.$$

\end{definition}

Some fundamental properties of the Riesz capacity are stated in the
following proposition.  The Euclidean case can be seen  in
\cite[Proposition 2]{cha}.
\begin{proposition}\label{pro-6} Let $(s,\alpha)\in(0,1)\times(0,1)$, $p\in[1,\infty)$ and
$\mathbf{s}=s\alpha $. The following properties are valid.
\item{\rm (i)} $Cap_{\mathbf{s},p}(\emptyset)=0.$
\item{\rm (ii)} If $E_{1}\subseteq E_{2}\subset\X$, then $Cap_{\mathbf{s},p}(E_{1})\leq Cap_{\mathbf{s},p}(E_{2})$.
\item{\rm (iii)} For any sequence $\{E_{j}\}^{\infty}_{j=1}$ of subsets of $\X$
$$Cap_{\mathbf{s},p}\Big(\bigcup^{\infty}_{j=1}E_{j}\Big)\leq \sum^{\infty}_{j=1}Cap_{\mathbf{s},p}(E_{j}).$$

\end{proposition}
\begin{proof}
 The statements (i) \&\ (ii) can be deduced from the definition of $Cap_{\mathbf{s},p}(\cdot)$ immediately. For (iii), let $\epsilon>0$.
 Take $f_{j}\geq 0$ such that $\mathscr{I}_{2\mathbf{s}}f_{j}\geq 1$ on $E_{j}$ and $$\int_{\X}|f_{j}(g)|^{p}dg\leq Cap_{\mathbf{s},p}(E_{j})+\epsilon/2^{j}.$$
  Let $f=\sup_{j\in\mathbb{N}_{+}}f_{j}$. For any $g \in\bigcup_{j=1}^{\infty}E_{j}$, there exists a $j_{0}$ such that $g\in E_{j_{0}}$
   and $\mathscr{I}_{2\mathbf{s}}f_{j_{0}}(g)\geq 1$. Hence $\mathscr{I}_{2\mathbf{s}}f(g)\geq 1$ on $\bigcup_{j=1}^{\infty}E_{j}$. On the other hand,
 $$\|f\|_{L^{p}(\X)}^{p}=\int_{\X}|f(g)|^{p}dg\leq \sum_{j=1}^{\infty}\int_{\X}|f_{j}(g)|^{p}dg=\sum_{j=1}^{\infty}Cap_{\mathbf{s},p}(E_{j})+\epsilon,$$
 which indicates
 $$Cap_{\mathbf{s},p}(\bigcup_{j=1}^{\infty}E_{j})\leq \sum_{j=1}^{\infty}Cap_{\mathbf{s},p}(E_{j}).$$

\end{proof}
Furthermore,  we prove that the capacity $Cap_{\mathbf{s},p}(\cdot)$
is an outer capacity.
\begin{proposition}\label{cap-11} Let $(s,\alpha)\in(0,1)\times(0,1)$, $p\in[1,\infty)$ and
$\mathbf{s}=s\alpha $. For any subset $E\subset \X$,
$$Cap_{\mathbf{s},p}(E)=\inf\Big\{Cap_{\mathbf{s},p}(O):\ O\supset E,\ O\  \text{open}\Big\}.$$
\end{proposition}
\begin{proof}
Without loss of generality, we assume that $Cap_{\mathbf{s},p}(E)<\infty$. By (ii) of Proposition \ref{pro-6},
$$Cap_{\mathbf{s},p}(E)\leq \inf\Big\{Cap_{\mathbf{s},p}(O):\ O\supset E,\ O\  \text{open}\Big\}.$$
For $\epsilon\in(0,1)$, there exists a measurable  nonnegative
function $f$ such that $\mathscr{I}_{2\mathbf{s}}f\geq 1$ on $E$ and
$$\int_{\X}|f(g)|^{p}dg\leq Cap_{\mathbf{s},p}(E)+\epsilon.$$
Since $\mathscr{I}_{2\mathbf{s}}f$ is lower semi-continuous, then the set
$$O_{\epsilon}:=\Big\{g\in\X:
\mathscr{I}_{2\mathbf{s}}f(g)>1-\epsilon\Big\}$$ is  an open set. On
the other hand, $E\subset O_{\epsilon}$, which implies that
$$Cap_{\mathbf{s},p}(O_{\epsilon})\leq \frac{1}{(1-\epsilon)^{p}}\int_{\X}|f(g)|^{p}dg<\frac{1}{(1-\epsilon)^{p}}(Cap_{\mathbf{s},p}(E)+\epsilon).$$
The arbitrariness of $\epsilon$ indicates that
$$Cap_{\mathbf{s},p}(E)\geq \inf\Big\{Cap_{\mathbf{s},p}(O):\ O\supset E,\ O\  \text{open}\Big\}.$$

\end{proof}

An immediate corollary of Proposition \ref{cap-11} is the following
result.
\begin{corollary}\label{cap-2} Let $(s,\alpha)\in(0,1)\times(0,1)$, $p\in[1,\infty)$ and
$\mathbf{s}=s\alpha $. If $\{\mathcal{K}_{j}\}_{j=1}^{\infty}$ is a
decreasing sequence of compact sets, then
$$Cap_{\mathbf{s},p}(\bigcap_{j=1}^{\infty}\mathcal{K}_{j})=\lim_{j\rightarrow\infty}Cap_{\mathbf{s},p}(\mathcal{K}_{j}).$$
\end{corollary}
%
\begin{proposition}\label{cap-3}
Let $(s,\alpha)\in(0,1)\times(0,1)$, $p\in(1,\infty)$ and
$\mathbf{s}=s\alpha $.  If $\{E_{j}\}_{j=1}^{\infty}$ is an
increasing sequence of arbitrary subsets of $\X$, then
$$Cap_{\mathbf{s},p}(\bigcup_{j=1}^{\infty}E_{j})=\lim_{j\rightarrow\infty}Cap_{\mathbf{s},p}(E_{j}).$$
\end{proposition}
\begin{proof}
Since $\{E_{j}\}_{j=1}^{\infty}$ is increasing, then
$$Cap_{\mathbf{s},p}(\bigcup_{j=1}^{\infty}E_{j})\geq\lim_{j\rightarrow\infty}Cap_{\mathbf{s},p}(E_{j}).$$
Conversely, without loss generality, we assume that
$\lim_{j\rightarrow\infty}Cap_{\mathbf{s},p}(E_{j})$ is finite.  For
each $j$, let $f_{E_{j}}$ be the unique function such that
$f_{E_{j}}\geq 1$ on $E_{j}$ and
$\|f_{E_{j}}\|_{L^{p}}^{p}=Cap_{\mathbf{s},p}(E_{j}).$ Then for
$i<j$, it follows  that $\mathscr{I}_{2\mathbf{s}}f_{E_{j}}\geq 1$
on $E_{i}$ and further,
$\mathscr{I}_{2\mathbf{s}}((f_{E_{i}}+f_{E_{j}})/2)\geq 1$ on
$E_{i}$, which means that
$$\int_{\X}((f_{E_{i}}+f_{E_{j}})/2)^{p}dg\geq Cap_{\mathbf{s},p}(E_{i}).$$
By \cite[Corollary 1.3.3]{ada}, the sequence $\{f_{E_{j}}\}_{j=1}^{\infty}$ converges strongly to a function $f$ satisfying
$$\|f\|_{L^{p}(\X)}^{p}=\lim_{j\rightarrow\infty}Cap_{\mathbf{s},p}(E_{j}).$$
Similarly to \cite[Proposition 2.3.12]{ada}, we can prove that $\mathscr{I}_{2\mathbf{s}}f\geq 1$ on $\bigcup_{j=1}^{\infty}E_{j}$, except
 possibly on a countable union of sets with $Cap_{\mathbf{s},p}(\cdot)$ zero.
 Hence,
$$\lim_{j\rightarrow\infty}Cap_{\mathbf{s},p}(E_{j})\geq \int_{\X}|f(g)|^{p}dg\geq Cap_{\mathbf{s},p}(\bigcup_{j=1}^{\infty}E_{j}).$$

\end{proof}
As a corollary of Proposition \ref{cap-3}, we can get
\begin{corollary}\label{cap-4} Let $(s,\alpha)\in(0,1)\times(0,1)$, $p\in(1,\infty)$ and
$\mathbf{s}=s\alpha $. Let $O$ be an open subset of $\X$. Then
$$Cap_{\mathbf{s},p}(O)=\sup\Big\{Cap_{\mathbf{s},p}(\mathcal{K}):\ \text{compact}\ \mathcal{K}\subset O\Big\}.$$
\end{corollary}
%
%

\begin{proposition}\label{relation}
Let $(s,\alpha,\sigma)\in(0,1)\times(0,1)\times(0,1)$,
$p\in(1,\infty)$, $\mathbf{s}=s\alpha $ and
$\widetilde{\mathbf{s}}=s\sigma $.

\item{\rm (i)} For $p<\mathcal{Q}/(2\mathbf{s})$,
$$Cap_{\dot{W}^{2\mathbf{s},p}(E)}\sim Cap_{\mathbf{s},p}(E)\ \text{for all Borel sets}\ E\subset\X.$$

\item{\rm (ii)} For $p<\mathcal{Q}/(2\widetilde{\mathbf{s}})$,
$$Cap_{\mathcal{\dot{W}}^{2\widetilde{\mathbf{s}},p}}(E)\sim Cap_{\widetilde{\mathbf{s}},p}(E)\ \text{for all Borel sets}\ E\subset\X.$$
\end{proposition}
\begin{proof}
We shall provide a proof for  (i) exclusively, since the proofs for
both (i) and (ii) follow similar arguments and thus the latter will
be omitted for brevity. From Proposition \ref{pro-6} (i) \&\ (ii),
Corollary \ref{cap-2} and Proposition \ref{cap-3}, we know that
$Cap_{\mathbf{s},p}(\cdot)$ is a Choquet capacity and thus enjoys
the inner/outer regularity properties, that is, any Borel set
$E\subset\X$ satisfies
$$Cap_{\mathbf{s},p}(E)=\sup_{\mathcal{K}\subset E, \mathcal{K}\ \text{compact}}Cap_{\mathbf{s},p}(\mathcal{K})=\inf_{O\supset E, O\ \text{open}}Cap_{\mathbf{s},p}(O).$$
So it suffices  to prove the validity of  this proposition for
compact sets $\mathcal{K}$ in $\X$.

For any $\epsilon\in(0,\infty)$, by the definition of
$Cap_{\mathcal{\dot{W}}^{2\mathbf{s},p}}(\cdot)$, there exists $u\in
C_{c}^{\infty}(\X)$ such that $u\geq 1_{\mathcal{K}}$ and
$$Cap_{\mathcal{\dot{W}}^{2\mathbf{s},p}}(\mathcal{K})+\epsilon>\|u\|_{\mathcal{\dot{W}}^{2\mathbf{s},p}(\X)}^{p}.$$
Upon taking $f=|\L^{\mathbf{s}}u|$ which belongs to $L^{p}(\X)$, we obtain
$$\mathscr{I}_{2\mathbf{s}}f\geq |\mathscr{I}_{2\mathbf{s}}\L^{\mathbf{s}}u|=u\geq 1_{\mathcal{K}}.$$
This in turn derives
$$Cap_{\mathbf{s},p}(\mathcal{K})\leq \|f\|_{L^{p}(\X)}^{p}=\|\L^{\mathbf{s}}u\|_{L^{p}(\X)}^{p}= \|u\|_{\mathcal{\dot{W}}^{2\mathbf{s},p}(\X)}^{p}\lesssim Cap_{\mathcal{\dot{W}}^{2\mathbf{s},p}}(\mathcal{K})+\epsilon.$$
Letting $\epsilon\rightarrow 0$ yields
\begin{equation*}\label{r-1}
  Cap_{\mathbf{s},p}(\mathcal{K})\lesssim Cap_{\mathcal{\dot{W}}^{2\mathbf{s},p}}(\mathcal{K}).
\end{equation*}

Conversely, for any $\epsilon\in(0,\infty)$, by the definition of
$Cap_{\mathbf{s},p}(\cdot)$, there exists a non-negative function
$f\in L^{p}(\X)$ such that $\mathscr{I}_{2\mathbf{s}}f\geq
1_{\mathcal{K}}$ and
\begin{equation}\label{r-2}
  Cap_{\mathbf{s},p}(\mathcal{K})+\epsilon>\|f\|_{L^{p}(\X)}^{p}.
\end{equation}

Consider now the function
$$u=\mathscr{I}_{2\mathbf{s}}f  \ \ \mathrm{with} \ \  f\in L^{p}(\X).$$
Also by the definition of $Cap_{\mathbf{s},p}(\cdot)$, we can see that $u\geq 1_{\mathcal{K}}$. Moreover,
\begin{equation}\label{r-4}
  \|u\|_{\mathcal{\dot{W}}^{2\mathbf{s},p}(\X)}=\|\L^{\mathbf{s}}u\|_{L^{p}(\X)}=\|\L^{\mathbf{s}}\mathscr{I}_{2\mathbf{s}}f\|_{L^{p}(\X)}= \|f\|_{L^{p}(\X)}<\infty.
\end{equation}
By the boundedness of
$$\mathscr{I}_{2\mathbf{s}}: L^{p}(\X)\rightarrow L^{p\mathcal{Q}/(Q-2\mathbf{s} p)}(\X)\ \text{under}\ 1<p<\mathcal{Q}/2\mathbf{s},$$
we know that $\mathscr{I}_{2\mathbf{s}}f\in L^{p\mathcal{Q}/(Q-2\mathbf{s} p)}(\X)$,
whence
\begin{eqnarray*}
  \|u\|_{L^{p}(\X)}^{p} &=& \int_{\X}|\mathscr{I}_{2\mathbf{s}}f(g)|^{p}dg  \\
   &\leq& \int_{\mathscr{I}_{2\mathbf{s}}f(g)\geq 1}|\mathscr{I}_{2\mathbf{s}}f(g)|^{p}dg+\int_{1/2\leq \mathscr{I}_{2\mathbf{s}}f(g)<1}dg  \\
   &\leq& \int_{\mathscr{I}_{2\mathbf{s}}f(g)\geq 1}|\mathscr{I}_{2\mathbf{s}}f(g)|^{p\mathcal{Q}/(Q-2\mathbf{s} p)}dg+
   \int_{1/2\leq \mathscr{I}_{2\mathbf{s}}f(g)<1}|2\mathscr{I}_{2\mathbf{s}}f(g)|^{p\mathcal{Q}/(Q-2\mathbf{s} p)}dg  \\
   &\lesssim& \|\mathscr{I}_{2\mathbf{s}}f\|_{L^{p\mathcal{Q}/(Q-2\mathbf{s} p)}(\X)}^{p\mathcal{Q}/(Q-2\mathbf{s} p)}  \\
   &\lesssim& \|f\|_{L^{p}(\X)}^{p\mathcal{Q}/(Q-2\mathbf{s} p)}<\infty.
\end{eqnarray*}

So far, we have obtained
$u\in  \mathcal{\dot{W}}^{2\mathbf{s},p}(\X)\ \text{and}\ u\geq 1_{\mathcal{K}}.$
Since $C_{c}^{\infty}(\X)$ is dense in $\mathcal{W}^{2\mathbf{s},p}(\X)$, we find
$$ \left\{\begin{aligned}
&\{u_{j}\}_{j=1}^{\infty}\subset C_{c}^{\infty}(\X)\ \text{such that}\ \lim_{j\rightarrow\infty}\|u_{j}-u\|_{\mathcal{\dot{W}}^{2\mathbf{s},p}(\X)}=0;\\
&u_{j}\geq 1_{\mathcal{K}}\ \text{for all}\ j\in\mathbb{N}.
\end{aligned}\right.$$
Combining these with (\ref{r-2}) and (\ref{r-4}) yields
$$Cap_{\mathcal{\dot{W}}^{2\mathbf{s},p}(\X)}(\mathcal{K})\leq \lim_{j\rightarrow\infty}\|u_{j}\|_{\mathcal{\dot{W}}^{2\mathbf{s},p}(\X)}^{p}=\|u\|_{\mathcal{\dot{W}}^{2\mathbf{s},p}(\X)}^{p}
=\|f\|_{L^{p}(\X)}^{p}\lesssim
Cap_{\mathbf{s},p}(\mathcal{K})+\epsilon.$$ Finally, letting
$\epsilon\rightarrow 0$ gives the desired inequality:
$$Cap_{\mathcal{\dot{W}}^{2\mathbf{s},p}(\X)}(\mathcal{K})\lesssim Cap_{\mathbf{s},p}(\mathcal{K}).$$
\end{proof}

For any $u\in L^{1}_{loc}(\X)$, define its Hardy-Littlewood maximal function $\mathcal{M}u$ as follows:
$$\mathcal{M}u(g):=\sup_{g\in B}\frac{1}{|B|}\int_{B}|u(g')|dg'\ \ \forall\ g\in\X,$$
where $B=B(g,r)$ and the supremum is taken over all balls $B$ of
$\X$  containing  $g$. It was established  by Folland and Stein
\cite{F82} that $\mathcal{M}$ is bounded on $L^{p}(\X)$ when
$p\in(1,\infty]$, and bounded from $L^{1}(\X)$ to
$L^{1,\infty}(\X)$. In this context, $L^{1,\infty}(\X)$ denotes the
weak-$L^{1}$ space of all measurable functions $f$ on $\X$ that
satisfy
$$\|f\|_{L^{1,\infty}(\X)}:=\sup_{\lambda>0}\lambda|\{g\in \X:|f(g)|>\lambda\}|<\infty.$$

Below we obtain two strong-type capacitary inequalities.
\begin{proposition}\label{strong}
Let $(s,\alpha)\in(0,1)\times(0,1)$,
$p\in(1,\mathcal{Q}/(2\mathbf{s}))$ and $\mathbf{s}=s\alpha $. Then
the following two statements hold.

\item{\rm (i)} For any $u\in \mathcal{\dot{W}}^{2\mathbf{s},p}(\X)$,
$$\int_{0}^{\infty}Cap_{\mathcal{\dot{W}}^{2\mathbf{s},p}}(\{g\in\X: |u(g)|\geq \lambda\})d\lambda^{p}\lesssim \|u\|_{\mathcal{\dot{W}}^{2\mathbf{s},p}(\X)}^{p}.$$

\item{\rm (ii)} For any $u\in \mathcal{\dot{W}}^{2\mathbf{s},p}(\X)$,
$$\int_{0}^{\infty}Cap_{\mathcal{\dot{W}}^{2\mathbf{s},p}}(\{g\in\X: |\mathcal{M}u(g)|\geq \lambda\})d\lambda^{p}\lesssim \|u\|_{\mathcal{\dot{W}}^{2\mathbf{s},p}(\X)}^{p}.$$
\end{proposition}

\begin{proof}
We first prove (i). For any $u\in \mathcal{\dot{W}}^{2\mathbf{s},p}(\X)$, each nonnegative measure $\mu$ given on the Borel $\sigma$-algebra of the space $\X$ generates the function $\Theta_{p,\mathbf{s}}\mu$ defined on $\X$ by
$$(\Theta_{p,\mathbf{s}}\mu)(g)=\int_{\X}I_{2\mathbf{s}}(g'^{-1}g)\Big(\int_{\X}I_{2\mathbf{s}}(g'^{-1}z)d\mu(z)\Big)^{1/(p-1)}dg'$$
or, equivalently
$$\Theta_{p,\mathbf{s}}\mu=\mathscr{I}_{2\mathbf{s}}(\mathscr{I}_{2\mathbf{s}} \mu)^{p'-1},\ \ p+p'=pp'.$$
By Lemma \ref{upper}, we can see that the Riesz kernel
$I_{2\mathbf{s}}(\cdot)$ satisfies
$$\int_{0}^{1}I_{2\mathbf{s}}(r)r^{\mathcal{Q}-1}dr<\infty\ \ \text{and}\ \ \int_{1}^{\infty}I_{2\mathbf{s}}^{p'}(r)r^{\mathcal{Q}-1}dr<\infty $$
for $p<\mathcal{Q}/2\mathbf{s}.$ Thus by \cite{AM}, we know that
$\Theta_{p,\mathbf{s}}\mu$ satisfies the rough maximum principle,
i.e., there exists a constant $C$ depending only on $\mathcal{Q}$, $p$ and
$\mathbf{s}$ such that
\begin{equation}\label{maxi}
  (\Theta_{p,\mathbf{s}}\mu)(g)\leq C\sup\{(\Theta_{p,\mathbf{s}}\mu)(g): g\in supp \mu\}.
\end{equation}
Then we utilize  (\ref{maxi})  to obtain (i)  using  the similar
argument of \cite[page 368, Theorem]{M1985}.

For (ii), by Proposition \ref{relation}, we can see that for
$1<p<\mathcal{Q}/2\mathbf{s}$, $u\in
\mathcal{\dot{W}}^{2\mathbf{s},p}(\X)$ if and only if
$u=\mathscr{I}_{2\mathbf{s}}\varphi$ and
$\|u\|_{\mathcal{\dot{W}}^{2\mathbf{s},p}(\X)}=\|\varphi\|_{L^{p}(\X)}$
for some $\varphi\in L^{p}(\X)$. Then for fixed
$u\in\mathcal{\dot{W}}^{2\mathbf{s},p}(\X)$  and $\varphi\in
L^{p}(\X)$ with $u(g)=\mathscr{I}_{2\mathbf{s}}\varphi(g)$,
according to Johnson \cite{joh}, we have
$$\mathcal{M}(\mathscr{I}_{2\mathbf{s}}\varphi)\leq \mathscr{I}_{2\mathbf{s}}(\mathcal{M}\varphi)$$
and
$$\{g\in\X: |\mathcal{M}u(g)|\geq\lambda\}\subseteq \{g\in\X: |\mathscr{I}_{2\mathbf{s}}(\mathcal{M}\varphi)(g)|\geq\lambda\}.$$
Note that $\mathcal{M}\varphi\in L^{p}(\X)$ and $\|\mathcal{M}\varphi\|_{L^{p}(\X)}\lesssim \|\varphi\|_{L^{p}(\X)}$.
Thus (i) implies (ii).

\end{proof}

Similarly to the proof of Proposition \ref{strong}, we can also  get

\begin{proposition}\label{strong-1}
Let $(s,\sigma)\in(0,1)\times(0,1)$,
$p\in(1,\mathcal{Q}/2\widetilde{\mathbf{s}})$ and
$\tilde{\mathbf{s}}=s\sigma $. Then the following two statements
hold.

\item{\rm (i)} For any $u\in \mathcal{\dot{W}}^{2\widetilde{\mathbf{s}},p}(\X)$,
$$\int_{0}^{\infty}Cap_{\mathcal{\dot{W}}^{2\widetilde{\mathbf{s}},p}}(\{g\in\X: |u(g)|\geq \lambda\})d\lambda^{p}\lesssim \|u\|_{\mathcal{\dot{W}}^{2\widetilde{\mathbf{s}},p}(\X)}^{p}.$$

\item{\rm (ii)} For any $u\in \mathcal{\dot{W}}^{2\widetilde{\mathbf{s}},p}(\X)$,
$$\int_{0}^{\infty}Cap_{\mathcal{\dot{W}}^{2\widetilde{\mathbf{s}},p}}(\{g\in\X: |\mathcal{M}u(g)|\geq \lambda\})d\lambda^{p}\lesssim \|u\|_{\mathcal{\dot{W}}^{2\widetilde{\mathbf{s}},p}(\X)}^{p}.$$
\end{proposition}

\section{$L^{q}(\X\times\mathbb{R}_{+},\mu)$-extension of $\mathcal{\dot{W}}^{2\Theta,p}(\X)$}\label{sec-4}
According to  Propositions \ref{pro-frac} \&\ \ref{pro-1}, it can be
observed  that the kernels of $H_{\alpha,t^{2\alpha}}$ and
$P_{\sigma,t}$ have the following relation:
$$K_{\sigma,t^{2\sigma}}(\cdot)\thicksim P_{\sigma,t}(\cdot),$$
Therefore,  the proofs of the main results for
$H_{\alpha,t^{2\alpha}}$ and $P_{\sigma,t}$  are analogous  in this
section, thus  we shall present  the relevant results for
$P_{\sigma,t}$ only.
\begin{lemma}\label{lem-4.3-1} Let $(s,\sigma)\in(0,1)\times(0,1)$
and $\widetilde{\mathbf{s}}=s\sigma$.  Let
$f\in\mathcal{\dot{W}}^{2\widetilde{\mathbf{s}},p}(\X)$  and $\mu$
be a nonnegative measure on $\X\times\mathbb{R}_{+}$. For any
$\lambda\in\mathbb{R}_{+}$, set
$$E_{\lambda}(f):=\Big\{(g',t)\in\X\times\mathbb{R}_{+}:\ |P_{\sigma,t}f(g')|>\lambda\Big\}$$
and
$$O_{\lambda}(f):=\Big\{g\in\X:\ \sup_{t\in\mathbb{R}_{+}}\sup_{d(g,g')<t}|P_{\sigma,t}f(g')|>\lambda\Big\}.$$
Then the following four statements are true.
\item{\rm (i)} For any $(g_{0},r)\in\X\times\mathbb{R}_{+}$,
$$\mu(E_{\lambda}(f)\cap T(B(g_{0},r)))\leq \mu(T(O_{\lambda}(f)\cap B(g_{0},r))).$$
\item{\rm (ii)} For any $(g_{0},r)\in\X\times\mathbb{R}_{+}$,
$$Cap_{\mathcal{\dot{W}}^{2\widetilde{\mathbf{s}},p}}(O_{\lambda}(f)\cap B(g_{0},r))\geq c_{p}(\mu; \mu(T(O_{\lambda}(f)\cap B(g_{0},r)))).$$
\item{\rm (iii)} There exists a constant $C>0$ such that for any locally integrable function $f: \X\rightarrow\mathbb{R}$ and $g\in\X$,
$$\sup_{t\in\mathbb{R}_{+}}\sup_{d(g,g')<t}|P_{\sigma,t}f(g')|\leq C\mathcal{M}f(g).$$
\item{\rm (iv)} For any bounded open set $O\subset \text{int}(\{g\in\X:|f(g)|\geq 1\})$
$$\inf_{(g,t)\in T(O)}P_{\sigma,t}(|f|)(g)\geq C>0.$$
\end{lemma}
\begin{proof}
(i) By Remark \ref{333} (ii), we know that
$$g\rightarrow \sup_{t\in\mathbb{R}_{+}}\sup_{d(g,g')<t}|P_{\sigma,t}f(g')|$$
is low semicontinuous. Thus $O_{\lambda}(f)$ is an open set of $\X$. Via observing
\begin{equation*}
  \begin{cases}
  E_{\lambda}(f)\subset T(O_{\lambda}(f)), \\
  T(U_{1}\cup U_{2})=T(U_{1})\cap T(U_{2}) \ \ \forall\ \text{open sets}\ U_{1},\ U_{2}\subset \X,
  \end{cases}
\end{equation*}
we consequently obtain
$$\mu(E_{\lambda}(f)\cap T(B(g_{0},r)))\leq \mu(T(O_{\lambda}(f))\cap T(B(g_{0},r)))=\mu(T(O_{\lambda}(f)\cap B(g_{0},r))).$$

(ii) It can be proved by use of the definition of $c_{p}(\mu;\cdot)$ and
the fact that $O_{\lambda}(f)$ is open.

(iii) From Proposition \ref{pro-1}, we know that
$$\sup_{t\in\mathbb{R}_{+}}\sup_{d(g,g')<t}|P_{\sigma,t}f(g')|\lesssim
\sup_{t\in\mathbb{R}_{+}}\sup_{d(g,g')<t}\int_{\X}\frac{t^{2\sigma}}{(t+d(g,g'))^{\Q+2\sigma}}f(g)dg\lesssim
\mathcal{M}f(g),$$ where we have used the method of ring
decomposition in the last inequality.

(iv) For any $(g,t)\in T(O)$, we use
$$O\subset \{g\in \X:|f(g)|\geq 1\}$$
to derive $B(g,t)\subset O$ and $|f|\geq 1$ on $B(g,t)$. By Proposition \ref{pro-1}, for any $g,g'\in \X$ and $d(g,g')<t$,
$$P_{\sigma,t}(g'^{-1}g)\gtrsim t^{-\Q}.$$
Then
$$P_{\sigma,t}(|f|)(g)\geq \int_{B(g,t)}P_{\sigma,t}(g'^{-1}g)|f(g')|dg'\geq C>0.$$

\end{proof}

\subsection{Proof of Theorem \ref{thm1}}
The following result covers Theorem \ref{thm1}.
Under the condition that $q<p$, we have the following result.
\begin{theorem}\label{extension-1}
Let $(s,\sigma)\in(0,1)\times(0,1)$,
$p\in(1,\mathcal{Q}/2\widetilde{\mathbf{s}})$,
$\widetilde{\mathbf{s}}=s\sigma$ and $q\in(0,p)$. Let $\mu$ be a
nonnegative measure on $\X\times\mathbb{R}_{+}$.  The following
statements are equivalent:
\item{\rm (i)}
$$\Big(\int_{\X\times\mathbb{R}_{+}}\Big|P_{\sigma,t}u(g)\Big|^{q}d\mu(g,t)\Big)^{1/q}\lesssim \|u\|_{\mathcal{\dot{W}}^{2\widetilde{\mathbf{s}},p}(\X)}\ \ \ \ \forall\ u\in \mathcal{\dot{W}}^{2\widetilde{\mathbf{s}},p}(\X).$$

\item{\rm (ii)} The measure $\mu$ satisfies
$$I_{p,q}(\mu):=\int^{\infty}_{0}\Big(\frac{t^{p/q}}{c_{p}(\mu;t)}\Big)^{\mathcal{Q}/(p-q)}\frac{dt}{t}<\infty,$$
where   $c_{p}(\mu;t)$ is defined  in (\ref{1.4}).
\end{theorem}
\begin{proof}
(i)$\Rightarrow$(ii). If (i) holds, then
$$C_{p,q}(\mu):=\sup_{u\in C_{c}^{\infty}(\X)\ \&\ \|u\|_{\mathcal{\dot{W}}^{2\widetilde{\mathbf{s}},p}(\X)}>0}\frac{1}{\|u\|_{\mathcal{\dot{W}}^{2\widetilde{\mathbf{s}},p}(\X)}}
\Big(\int_{\X\times\mathbb{R}_{+}}\Big|P_{\sigma,t}u(g)\Big|^{q}d\mu(g,t)\Big)^{1/q}<\infty.$$
For each $u\in C_{c}^{\infty}(\X)$ with
$\|u\|_{\mathcal{\dot{W}}^{2\widetilde{\mathbf{s}},p}(\X)}>0$, it
holds
$$\Big(\int_{\X\times\mathbb{R}_{+}}\Big|P_{\sigma,t}u(g)\Big|^{q}d\mu(g,t)\Big)^{1/q}\leq C_{p,q}(\mu)\|u\|_{\mathcal{\dot{W}}^{2\widetilde{\mathbf{s}},p}(\X)},$$
which indicates that
$$\sup_{\lambda>0}\lambda\Big(\mu(E_{\lambda}(u))\Big)^{1/q}\lesssim C_{p,q}(\mu)\|u\|_{\mathcal{\dot{W}}^{2\widetilde{\mathbf{s}},p}(\X)}.$$
This, together with (iv) of Lemma \ref{lem-4.3-1}, implies that for fixed $u\in C_{c}^{\infty}(\X)$ and any bounded open set $U\subseteq \{g\in\X: u(g)\geq 1\}^{o}$,
$$\mu(T(U))\lesssim C_{p,q}^{q}(\mu)\|u\|_{\mathcal{\dot{W}}^{2\widetilde{\mathbf{s}},p}(\X)}^{q}.$$
The definition of $c_{p}(\mu;t)$ implies that $c_{p}(\mu;t)>0$.

For any $j\in\mathbb{N}$, by the definition of $c_{p}(\mu;t)$, there exists a bounded open set $U_{j}$ of $\X$ such that
$$\mu(T(U_{j}))\geq 2^{j}\ \ \&\ \ \ Cap_{\mathcal{\dot{W}}^{2\widetilde{\mathbf{s}},p}}(U_{j})\leq 2c_{p}(\mu;2^{j}).$$
By Proposition \ref{relation}, we know that for any $E\subset\X$,
$$Cap_{\mathcal{\dot{W}}^{2\widetilde{\mathbf{s}},p}}(E)\sim\inf\Big\{\|f\|_{L^{p}(\X)}^{p}: 0\leq f\in C_{c}^{\infty}(\X),\ E\subseteq \{g\in\X: \mathscr{I}_{2\widetilde{\mathbf{s}}}f(g)\geq 1\}^{o}\Big\}.$$
Thus there exists $f_{j}\in C_{c}^{\infty}(\X)$ such that
$$ \left\{\begin{aligned}
&\mathscr{I}_{2\widetilde{\mathbf{s}}}f_{j}(g)\geq 1,\ g\in U_{j};\\
&\|f_{j}\|_{L^{p}(\X)}^{p}\leq 2Cap_{\mathcal{\dot{W}}^{2\widetilde{\mathbf{s}},p}}(U_{j})\leq 4c_{p}(\mu;2^{j}).
\end{aligned}\right.$$
For the integers $i,k$ with $i<k$, define
$$f_{i,k}:=\sup_{i\leq j\leq k}\Big(\frac{2^{j}}{c_{p}(\mu;2^{j})}\Big)^{1/(p-q)}f_{j}.$$
Then $f_{i,k}\in L^{p}(\X)$ with
$$\|f_{i,k}\|_{L^{p}(\X)}^{p}\lesssim \sum_{j=i}^{k}\Big(\frac{2^{j}}{c_{p}(\mu;2^{j})}\Big)^{p/(p-q)}c_{p}(\mu;2^{j}).$$
It is easy to see that for $i\leq j\leq k$, if $g\in U_{j}$, then
$$\mathscr{I}_{2\widetilde{\mathbf{s}}}f_{i,k}(g)\geq\Big(\frac{2^{j}}{c_{p}(\mu;2^{j})}\Big)^{1/(p-q)}.$$
Set
$$u_{i,k}(g,t):=P_{\sigma,t}(\mathscr{I}_{2\widetilde{\mathbf{s}}}f_{i,k})(g).$$
We can deduce from (iv) of Lemma \ref{lem-4.3-1} that there exists a constant $M$ such that for $(g,t)\in T(U_{j})$,
$$|u_{i,k}(g,t)|\geq \Big(\frac{2^{j}}{c_{p}(\mu;2^{j})}\Big)^{1/(p-q)}M.$$
Thus, with $\lambda=\Big(\frac{2^{j}}{c_{p}(\mu;2^{j})}\Big)^{1/(p-q)}\frac{M}{2}$,
$$2^{j}\leq \mu(T(U_{j}))\leq\mu(E_{\lambda}(\mathscr{I}_{2\widetilde{\mathbf{s}}}f_{i,k}(g))).$$

This indicates that
\begin{eqnarray*}
&&(C_{p,q}(\mu)\|f_{i,k}\|_{L^{p}(\X)})^{q}\gtrsim \int_{\X\times\mathbb{R}_{+}}\Big|u_{i,k}(g,t)\Big|^{q}d\mu(g,t)\\
   &&= \int^{\infty}_{0}\Big(\inf \{\epsilon>0:\mu(\{(g,t)\in\X\times\mathbb{R}_{+}:\Big|P_{\sigma,t}(\mathscr{I}_{2\widetilde{\mathbf{s}}}f_{i,k})(g)\Big|>\epsilon\})\leq\lambda\}\Big)^{q}d\lambda  \\
  &&\geq \sum_{j=i}^{k}2^{j}\Big(\inf \{\epsilon>0:\mu(\{(g,t)\in\X\times\mathbb{R}_{+}:|P_{\sigma,t}(\mathscr{I}_{2\widetilde{\mathbf{s}}}f_{i,k})(g)|>\epsilon\})\leq 2^{j}\}\Big)^{q}\\
  &&\gtrsim \sum_{j=i}^{k}2^{j}\Big(\frac{2^{j}}{c_{p}(\mu;2^{j})}\Big)^{q/(p-q)}\\
  &&\gtrsim \Big(\sum_{j=i}^{k}\frac{2^{jp/(p-q)}}{c_{p}(\mu;2^{j})^{q/(p-q)}}\Big)^{(p-q)/p}\|f_{i,k}\|_{L^{p}(\X)}^{q}.
\end{eqnarray*}
Thus
$$\sum_{j=i}^{k}\frac{2^{jp/(p-q)}}{c_{p}(\mu;2^{j})^{q/(p-q)}}\lesssim (C_{p,q}(\mu))^{pq/(p-q)}.$$

(ii)$\Rightarrow$(i). Denote by  $B_{r}:=B(g_{0},r)$, where
$(g_{0},r)\in\X\times\mathbb{R}_{+}$. For any $j\in\mathbb{Z}$ and
$u\in C_{c}^{\infty}(\X)$, let
\begin{equation*}
  \begin{cases}
  E_{j}:=\big\{(g,t)\in\X\times\mathbb{R}_{+}:\ |P_{\sigma,t}u(g)|>2^{j}\big\};\\
  O_{j}:=\big\{g\in\X:\ \sup_{t\in\mathbb{R}_{+}}\sup_{d(g,g')<t}|P_{\sigma,t}u(g')|>2^{j}\big\}.
  \end{cases}
\end{equation*}
By Lemma \ref{lem-4.3-1} (i), we obtain
\begin{eqnarray*}
  \int_{T(B_{r})}\Big|P_{\sigma,t}u(g)\Big|^{q}d\mu(g,t) &=&\int^{\infty}_{0}\mu(\{(g,t)\in T(B_{r}): |P_{\sigma,t}u(g)|>\lambda\})d\lambda^{q}  \\
   &=& \sum_{j\in\mathbb{Z}}\int^{2^{j+1}}_{2^{j}}\mu(\{(g,t)\in T(B_{r}):|P_{\sigma,t}u(g)|>\lambda\})d\lambda^{q}  \\
   &\leq& \sum_{j\in\mathbb{Z}}2^{(j+1)q}\mu(E_{j}\cap T(B_{r}))  \\
   &\leq& \sum_{j\in\mathbb{Z}}2^{(j+1)q}\mu(T(O_{j}\cap B_{r}))  \\
   &=& \sum_{j\in\mathbb{Z}}2^{(j+1)q}\sum_{i=j}^{\infty}\Big(\mu(T(O_{i}\cap B_{r}))-\mu(T(O_{i+1}\cap B_{r}))\Big) \\
   &=& \frac{2^{q}}{1-2^{-q}}\sum_{i\in\mathbb{Z}}2^{iq}\Big(\mu(T(O_{i}\cap B_{r}))-\mu(T(O_{i+1}\cap B_{r}))\Big).
\end{eqnarray*}
For the sake of notational simplicity and to facilitate clarity in
the subsequent discussion, let's define
\begin{equation*}
 \left\{ \begin{aligned}
  \mu_{i,r}&:=\mu(T(O_{i}\cap B_{r}))\ \ \ \forall\ i\in\mathbb{Z};\\
  S_{p,q,r}&:=\sum_{i\in\mathbb{Z}}\frac{(\mu_{i,r}-\mu_{i+1,r})^{p/(p-q)}}{(Cap_{\mathcal{\dot{W}}^{2\widetilde{\mathbf{s}},p}}(O_{i}\bigcap B_{r}))^{q/(p-q)}}.
  \end{aligned}\right.
\end{equation*}

Note that Lemma \ref{lem-4.3-1} (iii) implies
$$O_{i}\subset \{g\in\X: C\mathcal{M}u(g)>2^{i}\}\ \ \ \forall\ i\in\mathbb{Z}.$$
For $0<q<p$, by H\"{o}lder's inequality and Proposition
\ref{strong} (ii), it follows that
\begin{eqnarray*}
  &&\sum_{i\in\mathbb{Z}}2^{iq}(\mu_{i,r}-\mu_{i+1,r})\\
  &&\leq \sum_{i\in\mathbb{Z}}2^{iq}(Cap_{\mathcal{\dot{W}}^{2\widetilde{\mathbf{s}},p}}(O_{i}))^{q/p}
  \frac{\mu_{i,r}-\mu_{i+1,r}}{(Cap_{\mathcal{\dot{W}}^{2\widetilde{\mathbf{s}},p}}(O_{i}\cap B_{r}))^{q/p}}  \\
   &&\leq \Big(S_{p,q,r}\Big)^{(p-q)/p}\Big(\sum_{i\in\mathbb{Z}}2^{ip}Cap_{\mathcal{\dot{W}}^{2\widetilde{\mathbf{s}},p}}(O_{i})\Big)^{q/p}  \\
   &&\leq \Big(S_{p,q,r}\Big)^{(p-q)/p}\Big(\frac{1}{1-2^{-p}}\sum_{i\in\mathbb{Z}}\int^{2^{i}}_{2^{i-1}}Cap_{\mathcal{\dot{W}}^{2\widetilde{\mathbf{s}},p}}(\{g\in \X: C\mathcal{M}u(g)>2^{i}\})d\lambda^{p}\Big)^{q/p}   \\
   &&\leq 2^{q}\Big(S_{p,q,r}\Big)^{(p-q)/p}\Big(\int^{\infty}_{0}Cap_{\mathcal{\dot{W}}^{2\widetilde{\mathbf{s}},p}}(\{g\in\X: C\mathcal{M}u(g)>\lambda\})d\lambda^{p}\Big)^{q/p}  \\
 &&\lesssim \Big(S_{p,q,r}\Big)^{(p-q)/p}\|u\|^{q}_{\mathcal{\dot{W}}^{2\widetilde{\mathbf{s}},p}(\X)}.
\end{eqnarray*}
Using Lemma \ref{lem-4.3-1} (ii), the elementary inequality
$$(a-b)^{\kappa}\leq a^{\kappa}-b^{\kappa}\ \ \forall\ \kappa\in(1,\infty)\ \&\ 0<b<a<\infty,$$
and the fact that $t\mapsto c_{p}(\mu;t)$ is increasing on
$\mathbb{R}_{+}$, we have
\begin{eqnarray*}
  S_{p,q,r} &\leq&\sum_{i\in\mathbb{Z}}\frac{(\mu_{i,r})^{p/(p-q)}-(\mu_{i+1,r})^{p/(p-q)}}{(c_{p}(\mu;\mu_{i,r}))^{q/(p-q)}}  \\
   &=& \sum_{i\in\mathbb{Z}}\int_{\mu_{i+1,r}}^{\mu_{i,r}}\frac{1}{(c_{p}(\mu;\mu_{i,r}))^{q/(p-q)}}dt^{p/(p-q)} \\
   &\leq& \int^{\infty}_{0}\frac{1}{(c_{p}(\mu;t))^{q/(p-q)}}dt^{p/(p-q)}  \\
   &=&  I_{p,q}(\mu),
\end{eqnarray*}
whence
$$\int_{T(B_{r})}\Big|P_{\sigma,t}u(g)\Big|^{q}d\mu(g,t)\lesssim (I_{p,q}(\mu))^{(p-q)/p}\|u\|^{q}_{\mathcal{\dot{W}}^{2\widetilde{\mathbf{s}},p}(\X)}.$$
Letting $r\rightarrow\infty$ yields the desired inequality (i).
\end{proof}

For the case $q\geq p$, the equivalent characterization of the
extension result is given as follows.

\begin{theorem}\label{extension-2}
Let $(s,\sigma)\in(0,1)\times(0,1)$,
$p\in(1,\mathcal{Q}/2\widetilde{\mathbf{s}})$ with
$\widetilde{\mathbf{s}}=s\sigma$ and $q\in[p,\infty)$. Let $\mu$ be
a nonnegative measure on $\X\times\mathbb{R}_{+}$. Then the
following four statements are equivalent:
\item{\rm (i)}
$$\Big(\int_{\X\times\mathbb{R}_{+}}\Big|P_{\sigma,t}u(g)\Big|^{q}d\mu(g,t)\Big)^{1/q}\lesssim \|u\|_{\mathcal{\dot{W}}^{2\widetilde{\mathbf{s}},p}(\X)}\ \ \forall\ u\in \mathcal{\dot{W}}^{2\widetilde{\mathbf{s}},p}(\X).$$
\item{\rm (ii)}
$$\sup_{\lambda>0}\lambda(\mu(\{(g,t)\in\X\times\mathbb{R}_{+}:|P_{\sigma,t}u(g)|>\lambda\}))^{1/q}\lesssim \|u\|_{\mathcal{\dot{W}}^{2\widetilde{\mathbf{s}},p}(\X)}\ \ \forall\ u\in \mathcal{\dot{W}}^{2\widetilde{\mathbf{s}},p}(\X).$$

\item{\rm (iii)} The measure $\mu$ satisfies
$$\sup_{t>0}\frac{t^{1/q}}{(c_{p}(\mu;t))^{1/p}}<\infty.$$
\item{\rm (iv)} The measure $\mu$ satisfies
$$(\mu(T(O)))^{p/q}\lesssim Cap_{\mathcal{\dot{W}}^{2\widetilde{\mathbf{s}},p}}(O),$$
which holds for any bounded open set $O\subseteq \X.$
\end{theorem}
\begin{proof}
It is obvious that (i) implies the weak-type estimate (ii).

Assuming that (ii) holds, we next prove (iii). Given any
$\epsilon\in\mathbb{R}_{+}$, by the definition of $c_{p}(\mu;t)$,
there exists a bounded open set $O\subset\X$ such that
$$\mu(T(O))\geq t\ \ \&\ \ Cap_{\mathcal{\dot{W}}^{2\widetilde{\mathbf{s}},p}}(O)\leq c_{p}(\mu;t)+\epsilon.$$
Furthermore, by the definition of
$Cap_{\mathcal{\dot{W}}^{2\widetilde{\mathbf{s}},p}}(O)$, we know
that for all $u\in
\mathcal{\dot{W}}^{2\widetilde{\mathbf{s}},p}(\X)$,
$$\|u\|_{\mathcal{\dot{W}}^{2\widetilde{\mathbf{s}},p}(\X)}^{p}\leq Cap_{\mathcal{\dot{W}}^{2\widetilde{\mathbf{s}},p}}(O)+\epsilon\leq c_{p}(\mu;t)+2\epsilon.$$
So, Lemma \ref{lem-4.3-1} (iv) yields
$$\Big|P_{\sigma,t}u(g)\Big|\geq C\ \ \forall\ (g,t)\in T(O),$$
thereby giving
\begin{eqnarray*}
  t &\leq& \mu(T(O))  \\
   &\leq& \mu\Big(\{(g,t)\in \X\times\mathbb{R}_{+}:|P_{\sigma,t}u(g)|>C-\epsilon\}\Big) \\
   &\lesssim& \Big(\frac{\|u\|_{\mathcal{\dot{W}}^{2\widetilde{\mathbf{s}},p}(\X)}}{C-\epsilon}\Big)^{q} \\
   &\lesssim&
   \Big(\frac{(c_{p}(\mu;t)+2\epsilon)^{1/p}}{C-\epsilon}\Big)^{q}
\end{eqnarray*}  via (ii).
Letting $\epsilon\rightarrow 0$ implies
$$t^{1/q}\lesssim (c_{p}(\mu;t))^{1/p},$$
and hence (iii) holds.

Next, we need to show that (iii) implies (iv). For any bounded open
set $O\subset\X$ and $t_{0}:=\mu(T(O))$, we utilize the definition
of $c_{p}(\mu;t_{0})$ to derive
$$c_{p}(\mu;t_{0})\leq Cap_{\mathcal{\dot{W}}^{2\widetilde{\mathbf{s}},p}}(O),$$
thereby using (iii) to deduce
$$\frac{(\mu(T(O)))^{1/q}}{(Cap_{\mathcal{\dot{W}}^{2\widetilde{\mathbf{s}},p}}(O))^{1/p}}\leq\frac{t_{0}^{1/q}}{(c_{p}(\mu;t_{0}))^{1/p}}<\infty.$$
Thus, (iv) holds.

Finally, we prove that (iv) implies (i). Let $u\in \mathcal{\dot{W}}^{2\widetilde{\mathbf{s}},p}(\X)$. Then for any $\lambda\in\mathbb{R}_{+}$, let
\begin{equation*}
  \begin{cases}
  E_{\lambda}:=\{(g,t)\in\X\times\mathbb{R}_{+}:|P_{\sigma,t}u(g)|>\lambda\};\\
  O_{\lambda}:=\{g\in\X: \sup_{t\in\mathbb{R}_{+}}\sup_{d(g,g')<t}|P_{\sigma,t}u(g')|>\lambda\}.
  \end{cases}
\end{equation*}
For any $j\in\mathbb{Z}$, we simply write $E_{2^{j}}$ as $E_{j}$ and $O_{2^{j}}$ as $O_{j}$.

Fix a ball $B_{r}=B(g_{0},r)$, where $g_{0}\in\X$ and $r\in\mathbb{R}_{+}$. By Lemma \ref{lem-4.3-1} (i) and $p<q$, we have
\begin{eqnarray*}
  \Big(\int_{T(B_{r})}|P_{\sigma,t}u(g)|^{q}d\mu(g,t)\Big)^{p/q} &=& \Big(\int^{\infty}_{0}\mu(E_{\lambda}\cap T(B_{r}))d\lambda^{q}\Big)^{p/q}  \\
   &\leq& \Big(\sum_{j\in\mathbb{Z}}2^{(j+1)q}\mu(E_{j}\cap T(B_{r}))\Big)^{p/q} \\
   &\leq& \Big(\sum_{j\in\mathbb{Z}}2^{(j+1)q}\mu(T(O_{j}\cap B_{r}))\Big)^{p/q} \\
   &\leq&  \sum_{j\in\mathbb{Z}}2^{(j+1)p}(\mu(T(O_{j}\cap B_{r})))^{p/q}
\end{eqnarray*}
due to the following inequality
\begin{equation}\label{4.4}
\Big(\sum_{j}|a_{j}|\Big)^{\kappa}\leq \sum_{j}|a_{j}|^{\kappa}\ \ \forall\ \{a_{j}\}_{j}\subset\mathbb{C}\ \&\ \kappa\in(0,1].
\end{equation}

According to the proof of Lemma \ref{lem-4.3-1} (i), we know that
every $O_{j}$ is an open set. Therefore, every $O_{j}\cap B_{r}$ is
a bounded open set. Upon applying (iv) and the monotonicity of
$Cap_{\mathcal{\dot{W}}^{2\widetilde{\mathbf{s}},p}}$, we derive
\begin{eqnarray*}
  \sum_{j\in\mathbb{Z}}2^{(j+1)p}(\mu(T(O_{j}\cap B_{r})))^{p/q} &\lesssim& \sum_{j\in\mathbb{Z}}2^{(j+1)p}Cap_{\mathcal{\dot{W}}^{2\widetilde{\mathbf{s}},p}}(O_{j}\cap B_{r})  \\
   &\lesssim& \sum_{j\in\mathbb{Z}}\int^{2^{j}}_{2^{j-1}}Cap_{\mathcal{\dot{W}}^{2\widetilde{\mathbf{s}},p}}(O_{j})d\lambda^{p}  \\
   &\lesssim&  \int^{\infty}_{0}Cap_{\mathcal{\dot{W}}^{2\widetilde{\mathbf{s}},p}}(O_{\lambda})d\lambda^{p}.
\end{eqnarray*}
By Lemma \ref{lem-4.3-1} (iii), we obtain
\begin{equation}\label{eq5.2}O_{\lambda}\subset \{g\in \X: C\mathcal{M}u(g)>\lambda\}\
\ \forall\ \lambda\in\mathbb{R}_{+}.\end{equation} From
(\ref{eq5.2}) and Proposition \ref{strong-1} (ii), it follows that
$$\int^{\infty}_{0}Cap_{\mathcal{\dot{W}}^{2\widetilde{\mathbf{s}},p}}(O_{\lambda})d\lambda^{p}\leq \int^{\infty}_{0}Cap_{\mathcal{\dot{W}}^{2\widetilde{\mathbf{s}},p}}(\{g\in \X: C\mathcal{M}u(g)>\lambda\})d\lambda^{p}\lesssim \|u\|_{\mathcal{\dot{W}}^{2\widetilde{\mathbf{s}},p}(\X)}^{p}.$$
Therefore,
$$\Big(\int_{T(B_{r})}|P_{\sigma,t}u(g)|^{q}d\mu(g,t)\Big)^{p/q}\lesssim \|u\|_{\mathcal{\dot{W}}^{2\widetilde{\mathbf{s}},p}(\X)}^{p}.$$
Letting $r\rightarrow\infty$ yields (i).

\end{proof}

\subsection{Proof of Theorem \ref{thm2}}\label{sec-4.2}
The following result covers Theorem \ref{thm2}. As the  limiting
case $t\rightarrow 0$ of Theorems \ref{extension-1} \&\
\ref{extension-2}, we have the following trace theorem for the
fractional Sobolev space
$\mathcal{\dot{W}}^{2\widetilde{\mathbf{s}},p}(\X)$.

Under the condition that $q<p$, we have the following result.
\begin{theorem}\label{extension-3}
Let $(s,\sigma)\in(0,1)\times(0,1)$,
$p\in(1,\mathcal{Q}/2\widetilde{\mathbf{s}})$ with
$\widetilde{\mathbf{s}}=s\sigma$ and $q\in(0,p)$. Given a
nonnegative measure $\nu$ on $\X$, let
$$h_{p}(t):=\inf\{Cap_{\mathcal{\dot{W}}^{2\widetilde{\mathbf{s}},p}}(E): E\subset \X, \nu(E)\geq t\}\ \ \forall t\in \mathbb{R}_{+}.$$

The following statements are equivalent:
\item{\rm (i)}
$$\Big(\int_{\X}\big|u(g)\big|^{q}d\nu(g)\Big)^{1/q}\lesssim \|u\|_{\mathcal{\dot{W}}^{2\widetilde{\mathbf{s}},p}(\X)}\ \ \ \ \forall\ u\in \mathcal{\dot{W}}^{2\widetilde{\mathbf{s}},p}(\X).$$

\item{\rm (ii)} The measure $\nu$ satisfies
$$\|h_{p}\|:=\int^{\infty}_{0}\Big(\frac{d\xi^{{p}/{(p-q)}}}{(h_{p}(\xi))^{{q}/{(p-q)}}}\Big)^{(p-q)/pq}<\infty.$$

\end{theorem}
\begin{proof}
(i)$\Rightarrow$(ii). If (i) holds, then
$$C_{p,q}(\nu):=\sup_{u\in C_{c}^{\infty}(\X)\ \&\ \|u\|_{\mathcal{\dot{W}}^{2\widetilde{\mathbf{s}},p}(\X)}>0}\frac{1}
{\|u\|_{\mathcal{\dot{W}}^{2\widetilde{\mathbf{s}},p}(\X)}}
\Big(\int_{\X}\big|u(g)\big|^{q}d\nu(g)\Big)^{1/q}<\infty.$$ For
each $u\in C_{c}^{\infty}(\X)$ with
$\|u\|_{\mathcal{\dot{W}}^{2\widetilde{\mathbf{s}},p}(\X)}>0$, it
follows that
$$\Big(\int_{\X}|u(g)|^{q}d\nu(g)\Big)^{1/q}\leq C_{p,q}(\nu)\|u\|_{\mathcal{\dot{W}}^{2\widetilde{\mathbf{s}},p}(\X)},$$
which indicates that for every $E\subset\X$,
$$\nu(E)\lesssim C_{p,q}^{q}(\nu)\|u\|_{\mathcal{\dot{W}}^{2\widetilde{\mathbf{s}},p}(\X)}^{q}.$$
The definition of $h_{p}(t)$ implies that $h_{p}(t)>0$.

For any $j\in\mathbb{N}$, by the definition of $h_{p}(t)$, there exists $E_{j}\subset\X$ such that
$$\nu(E_{j})\geq 2^{j}\ \ \&\ \ \ Cap_{\mathcal{\dot{W}}^{2\widetilde{\mathbf{s}},p}}(E_{j})\leq 2h_{p}(2^{j}).$$
By Proposition \ref{relation}, we know that for any $E\subset\X$,
$$Cap_{\mathcal{\dot{W}}^{2\widetilde{\mathbf{s}},p}}(E)\sim\inf\{\|f\|_{L^{p}(\X)}^{p}: 0\leq f\in C_{c}^{\infty}(\X),\ E\subseteq \{g\in\X: \mathscr{I}_{2\widetilde{\mathbf{s}}}f(g)\geq 1\}^{o}\}.$$
Thus there exists $f_{j}\in C_{c}^{\infty}(\X)$ such that
$$ \left\{\begin{aligned}
&\mathscr{I}_{2\widetilde{\mathbf{s}}}f_{j}(g)\geq 1,\ g\in E_{j};\\
&\|f_{j}\|_{L^{p}(\X)}^{p}\leq 2Cap_{\mathcal{\dot{W}}^{2\widetilde{\mathbf{s}},p}}(E_{j})\leq 4h_{p}(2^{j}).
\end{aligned}\right.$$
For the integers $i,k$ with $i<k$, define
$$f_{i,k}=\sup_{i\leq j\leq k}\Big(\frac{2^{j}}{h_{p}(2^{j})}\Big)^{1/(p-q)}f_{j}.$$
Then $f_{i,k}\in L^{p}(\X)$ with
$$\|f_{i,k}\|_{L^{p}(\X)}^{p}\lesssim \sum_{j=i}^{k}\Big(\frac{2^{j}}{h_{p}(2^{j})}\Big)^{p/(p-q)}h_{p}(2^{j}).$$
It is easy to see that for $i\leq j\leq k$, if $g\in E_{j}$, then
$$\mathscr{I}_{2\widetilde{\mathbf{s}}}f_{i,k}(g)\geq\Big(\frac{2^{j}}{h_{p}(2^{j})}\Big)^{1/(p-q)}.$$

Thus
$$2^{j}\leq \nu(E_{j})\leq\nu(\{g\in\X: \mathscr{I}_{2\widetilde{\mathbf{s}}}f_{i,k}>\epsilon\})
$$ with
$\lambda=\Big(\frac{2^{j}}{c_{p}(\nu;2^{j})}\Big)^{1/(p-q)}\frac{1}{2}$.
Equivalently, if $\lambda>0$ such that $\nu(\{g\in\X:
\mathscr{I}_{2\widetilde{\mathbf{s}}}f_{i,k}>\epsilon\})\leq\lambda$,
then
$$\lambda>\frac{1}{2} \Big(\frac{2^{j}}{h_{p}(2^{j})}\Big)^{1/(p-q)}.$$

This indicates that
\begin{eqnarray*}
&&(C_{p,q}(\nu)\|f_{i,k}\|_{L^{p}(\X)})^{q}\gtrsim \int_{\X}\Big|\mathscr{I}_{2\widetilde{\mathbf{s}}}f_{i,k}(g)\Big|^{q}d\nu(g)\\
   &&= \int^{\infty}_{0}\Big(\inf \{\epsilon>0:\nu(\{g\in\X:\Big|\mathscr{I}_{2\widetilde{\mathbf{s}}}f_{i,k}(g)\Big|>\epsilon\})\leq\lambda\}\Big)^{q}d\lambda  \\
  &&\geq \sum_{j=i}^{k}2^{j}\Big(\inf \{\epsilon>0:\nu(\{g\in\X:|\mathscr{I}_{2\widetilde{\mathbf{s}}}f_{i,k}(g)|>\epsilon\})\leq 2^{j}\}\Big)^{q}\\
  &&\gtrsim \sum_{j=i}^{k}2^{j}\Big(\frac{2^{j}}{h_{p}(2^{j})}\Big)^{q/(p-q)}\\
  &&\gtrsim \Big(\sum_{j=i}^{k}\frac{2^{jp/(p-q)}}{h_{p}(2^{j})^{q/(p-q)}}\Big)^{(p-q)/p}\|f_{i,k}\|_{L^{p}(\X)}^{q}.
\end{eqnarray*}
Thus
$$\sum_{j=i}^{k}\frac{2^{jp/(p-q)}}{h_{p}(2^{j})^{q/(p-q)}}\lesssim (C_{p,q}(\nu))^{pq/(p-q)}.$$

(ii)$\Rightarrow$(i).  For any $t\in\mathbb{R}_{+}$ and $u\in C_{c}^{\infty}(\X)$, let
\begin{equation*}
 E_{t}:=\{g\in\X: |u(g)|>t\}.
\end{equation*}
Then
\begin{eqnarray}\notag
  \int_{X}\big|u(g)\big|^{q}d\nu(g) &=& \sum_{k\in\mathbb{Z}}\int^{2^{k+1}}_{2^{k}}\nu(E_{t})dt^{q}  \\ \notag
   &\lesssim& \sum_{k\in\mathbb{Z}}2^{kq}\nu(E_{2k})  \\ \label{5.7}
  &\lesssim& \sum_{i\in\mathbb{Z}}2^{iq}\Big(\nu(E_{2^{i}})-\nu(E_{2^{i+1}})\Big).
 \end{eqnarray}

Since $q<p$ and $\frac{1}{p/q}+\frac{1}{p/(p-q)}=1$, it follows from H\"{o}lder's inequality that
\begin{eqnarray}\label{5.8}
  &&\nonumber\sum_{i\in\mathbb{Z}}2^{iq}\Big(\nu(E_{2^{i}})-\nu(E_{2^{i+1}})\Big)\\
 &&
  \leq
  \Big(\sum_{i\in\mathbb{Z}}2^{ip}Cap_{\mathcal{\dot{W}}^{2\widetilde{\mathbf{s}},p}}(E_{2^{i}})\Big)^{q/p}\Big(\sum_{i\in\mathbb{Z}}
  \Big(\frac{\nu(E_{2^{i}})-\nu(E_{2^{i+1}})}{(Cap_{\mathcal{\dot{W}}^{2\widetilde{\mathbf{s}},p}}(E_{2^{i}}))^{q/p}}\Big)^{p/(p-q)}\Big)^{(p-q)/p}.
  \end{eqnarray}
By the monotonicity of $Cap_{\mathcal{\dot{W}}^{2\widetilde{\mathbf{s}},p}}(\cdot)$ and Proposition \ref{strong-1} (i), we have
\begin{eqnarray}\notag
 \sum_{i\in\mathbb{Z}}2^{ip}Cap_{\mathcal{\dot{W}}^{2\widetilde{\mathbf{s}},p}}(E_{2^{i}})  &\lesssim& \sum_{i\in\mathbb{Z}}\int_{2^{i-1}}^{2^{i}}Cap_{\mathcal{\dot{W}}^{2\widetilde{\mathbf{s}},p}}(E_{2^{i}})dt^{p}  \\ \notag
   &\lesssim&\int_{0}^{\infty}Cap_{\mathcal{\dot{W}}^{2\widetilde{\mathbf{s}},p}}(E_{t})dt^{p}  \\ \label{5.9}
   &\lesssim& \|u\|_{\mathcal{\dot{W}}^{2\widetilde{\mathbf{s}},p}(\X)}^{p}.
\end{eqnarray}
Using the definition of $h_{p}(\cdot)$, we know
$$h_{p}(\nu(E_{2^{i}}))\leq Cap_{\mathcal{\dot{W}}^{2\widetilde{\mathbf{s}},p}}(E_{2^{i}}).$$
From $p/(p-q)>1$ and the fact that $h_{p}(\cdot)$ is an increasing
function as well as (\ref{5.9}), it follows that
\begin{eqnarray}\notag
 \sum_{i\in\mathbb{Z}}
  \Big(\frac{\nu(E_{2^{i}})-\nu(E_{2^{i+1}})}{(Cap_{\mathcal{\dot{W}}^{2\widetilde{\mathbf{s}},p}}(E_{2^{i}}))^{q/p}}\Big)^{p/(p-q)}  &\leq&
  \sum_{i\in\mathbb{Z}}\frac{(\nu(E_{2^{i}})-\nu(E_{2^{i+1}}))^{p/(p-q)}}{(h_{p}(\nu(E_{2^{i}})))^{q/(p-q)}}  \\ \notag
   &\leq& \sum_{i\in\mathbb{Z}}\frac{(\nu(E_{2^{i}}))^{p/(p-q)}-(\nu(E_{2^{i+1}}))^{p/(p-q)}}{(h_{p}(\nu(E_{2^{i}})))^{q/(p-q)}}  \\ \notag
   &=& 2^{p/(p-q)}\sum_{i\in\mathbb{Z}}\int^{\nu(E_{2^{i}})}_{\nu(E_{2^{i+1}})}\frac{d\xi^{p/(p-q)}}{(h_{p}(\nu(E_{2^{i}})))^{q/(p-q)}} \\ \label{5.10}
   &\leq& 2^{p/(p-q)}\|h_{p}\|^{pq/(p-q)}.
\end{eqnarray}
Inserting (\ref{5.10}) and (\ref{5.9}) into (\ref{5.8}) yields
$$\sum_{i\in\mathbb{Z}}2^{iq}(\nu(E_{2^{i}})-\nu(E_{2^{i+1}}))\lesssim \|h_{p}\|^{q}\|u\|_{\mathcal{\dot{W}}^{2\widetilde{\mathbf{s}},p}(\X)}^{q},$$
which, along with (\ref{5.7}), further gives
$$\Big(\int_{\X}|u(g)|^{q}d\nu(g)\Big)^{1/q}\lesssim \|h_{p}\|\|u\|_{\mathcal{\dot{W}}^{2\widetilde{\mathbf{s}},p}(\X)}.$$

\end{proof}

For $q\geq p$, we get the following result.

\begin{theorem}\label{extension-4}
Let $(s,\sigma)\in(0,1)\times(0,1)$,
$p\in(1,\mathcal{Q}/2\widetilde{\mathbf{s}})$ with
$\widetilde{\mathbf{s}}=s\sigma$ and $q\in[p,\infty)$. Let $\nu$ be
a nonnegative measure on $\X$. Then the following four statements
are equivalent:
\item{\rm (i)}
$$\Big(\int_{\X}\big|u(g)\big|^{q}d\nu(g)\Big)^{1/q}\lesssim \|u\|_{\mathcal{\dot{W}}^{2\widetilde{\mathbf{s}},p}(\X)}\ \ \forall\ u\in \mathcal{\dot{W}}^{2\widetilde{\mathbf{s}},p}(\X).$$
\item{\rm (ii)} The measure $\nu$ satisfies
$$\sup_{E\subset\X}\frac{(\nu(E))^{1/q}}{(Cap_{\mathcal{\dot{W}}^{2\widetilde{\mathbf{s}},p}}(E))^{1/p}}<\infty.$$
\end{theorem}
\begin{proof}
 (i)$\Rightarrow$(ii). Given any $E\subset\X$, by the definition of $Cap_{\mathcal{\dot{W}}^{2\widetilde{\mathbf{s}},p}}(E)$ and (i), we know that for all $u\in C_{c}^{\infty}(\X)$
$$\nu(E)\leq \nu(\{g\in\X: u(g)=1\})\leq\int_{\X}|u(g)|^{q}d\nu(g)\lesssim \|u\|_{\mathcal{\dot{W}}^{2\widetilde{\mathbf{s}},p}(\X)}^{q},$$
which implies (ii).

 (ii)$\Rightarrow$(i). Let $u\in \mathcal{\dot{W}}^{2\widetilde{\mathbf{s}},p}(\X)$. Then for any $t\in\mathbb{R}_{+}$, let
$$E_{t}:=\{g\in\X:|u(g)|>t\}.$$
By (ii), we have
$$\int_{\X}|u(g)|^{q}d\nu(g)\lesssim \sum_{k\in\mathbb{Z}}2^{kq}\nu(E_{2^{k}})\lesssim \sum_{k\in\mathbb{Z}}2^{kq}(Cap_{\mathcal{\dot{W}}^{2\widetilde{\mathbf{s}},p}}(E_{2^{k}}))^{q/p}.$$

From (\ref{4.4}), $p\leq q$ and (\ref{5.9}), we can see that
$$\sum_{k\in\mathbb{Z}}2^{kq}(Cap_{\mathcal{\dot{W}}^{2\widetilde{\mathbf{s}},p}}(E_{2^{k}}))^{q/p}\lesssim \Big(\sum_{k\in\mathbb{Z}}2^{kq}Cap_{\mathcal{\dot{W}}^{2\widetilde{\mathbf{s}},p}}(E_{2^{k}})\Big)^{q/p}\lesssim \|u\|_{\mathcal{\dot{W}}^{2\widetilde{\mathbf{s}},p}(\X)}^{q},$$
whence
$$\Big(\int_{\X}|u(g)|^{q}d\nu(g)\Big)^{1/q}\lesssim \|u\|_{\mathcal{\dot{W}}^{2\widetilde{\mathbf{s}},p}(\X)}.$$
Accordingly, (i) holds.
\end{proof}
\begin{remark}
By Proposition \ref{w-w}, we know that when
$p\in(1,\min\{\mathcal{Q}/2\widetilde{\mathbf{s}},\Q/(1-2\widetilde{\mathbf{s}})\})$,
$\mathbb{W}^{2\widetilde{\mathbf{s}},p}(\X)=\mathcal{W}^{2\widetilde{\mathbf{s}},p}(\X)$.
Therefore,
$\|\cdot\|_{\mathcal{\dot{W}}^{2\widetilde{\mathbf{s}},p}(\X)}$ in
Theorems \ref{extension-1}, \ref{extension-2}, \ref{extension-3} \&\
\ref{extension-4} can be replaced by
$\|\cdot\|_{\mathbb{\dot{W}}^{2\widetilde{\mathbf{s}},p}(\X)}$.

\end{remark}

\section{Besov type spaces and Besov capacities }\label{sec-5}

In the subsequent discussion, we primarily focus on Besov-type
spaces that are constructed using fractional heat semigroups and
Caffarelli-Silvestre type extensions. It is important to note that
the study of Besov-type spaces in the context of general metric
spaces has been addressed in the literature, specifically in
\cite{ABC2020,ABCRST2020,ABCRST2021,CG2022,GL2015}. In contrast to
the scenarios involving general metric spaces, we are able to derive
more nuanced results for Besov-type spaces and   Besov capacities in
the framework of Stratified Lie groups.

\begin{definition}\label{Besov1}
Let $\alpha\in(0,1)$. The Besov type space
$B_{p,q,\mathcal{H}}^{{\alpha},\beta}(\X)$ is defined according to
the values of the parameters $(p,q,\beta)$ as follows:

\item{\rm (i)} If $(p,q,\beta)\in[1,\infty)\times[1,\infty)\times(0,\infty),$ then $B_{p,q,\mathcal{H}}^{{\alpha},\beta}(\X)$
              is the class of all $L^{p}(\X)$-functions $u$ such that
              \begin{align}\label{Besov}
                      N_{p,q,\mathcal{H}}^{{\alpha},\beta}(u):=\bigg(\int_{0}^{\infty}
                       \Big(\int_{\X}H_{\alpha,t}(|u-u(g)|^{p})(g)dg\Big)^{q/p}\frac{dt}{t^{{\beta q}/{2}+1}}
                    \bigg)^{1/q}<\infty.
\end{align}

\item{\rm (ii)} If $(p,q,\beta)\in[1,\infty)\times\{\infty\}\times(0,\infty),$ then $B_{p,q,\mathcal{H}}^{{\alpha},\beta}(\X)$
              is the class of all $L^{p}(\X)$-functions $u$ such that
              \begin{align}\label{NP}
                 N_{p,\infty,\mathcal{H}}^{{\alpha},\beta}(u):=\sup_{t>0}t^{-\beta/2}\Big(\int_{\X}H_{\alpha,t}(|u-u(g)|^{p})(g)dg\Big)^{1/p}<\infty.
              \end{align}

The space $B_{p,q,\mathcal{H}}^{{\alpha},\beta}(\X)$ is endowed with
the following norm
\begin{align*}\label{BSN}
\|u\|_{B_{p,q,\mathcal{H}}^{{\alpha},\beta}(\X)}:=
\|u\|_{L^{p}(\X)}+N_{p,q,\mathcal{H}}^{{\alpha},\beta}(u).
\end{align*}
\item{\rm (iii)} If $H_{\alpha,t}$ is replaced by $P_{\sigma,t}$ with $ \sigma\in
(0,1)$, we denote by $B_{p,q,\mathcal{P}}^{ {\sigma},\beta}(\X)$ and
$N_{p,q,\mathcal{P}}^{ {\sigma},\beta}$ the corresponding Besov
spaces and semi-norms, respectively.
\end{definition}

\begin{remark}\label{rem2}
\item{(i)} By Lemma \ref{fracheat-1} (i) \&\ (iii), we get the following conclusion.

 (1) If $(p,q,\beta)\in[1,\infty)\times[1,\infty)\times(0,\infty),$ then
$$(\ref{Besov})\Longleftrightarrow \widetilde{N}_{p,q,\mathcal{H}}^{ {\alpha},\beta}(u):=\bigg(\int_{0}^{1}
\Big(\int_{\X}H_{\alpha,t}(|u-u(g)|^{p})(g)dg\Big)^{q/p}\frac{dt}{t^{{\beta
q}/{2}+1}} \bigg)^{1/q}<\infty.$$

(2) When $(p,q,\beta)\in[1,\infty)\times\{\infty\}\times(0,\infty),$ then
  $$(\ref{NP})\Longleftrightarrow \limsup_{t\rightarrow 0}t^{-\beta/2}\Big(\int_{\X}H_{\alpha,t}(|u-u(g)|^{p})(g)dg\Big)^{1/p}<\infty.$$

\item{(ii)} The Besov type space $B_{p,q,\mathcal{H}}^{{\alpha},\beta}(\X)$
can be seen as a generalization of
 the classical Aronszajn-Gagliardo-Slobedetzky spaces $L^{\beta}_{p}(\mathbb{R}^{n})$  defined via the following seminorm:
$$\Bigg(\int_{\mathbb{R}^{n}}\int_{\mathbb{R}^{n}}\frac{|u(x)-u(y)|^{p}}{|x-y|^{n+\beta p}}dxdy\Bigg)^{1/p},$$
which is easily recognizable from (\ref{Besov}) since  one has $e^{t\Delta}(x-y)=(4\pi t)^{-n/2}e^{-|x-y|^{2}/4t}$ when $\X=\mathbb{R}^{n}$, $p=q$ and $\alpha=1$.

\item{(iii)} The conclusion above also holds for   Besov spaces
$B_{p,q,\mathcal{P}}^{ {\sigma},\beta}(\X)$.

\end{remark}

\subsection{Some properties of Besov type spaces }\label{sec-5-1}
In this section, we first focus on the dense subspaces and  the
min-max property of Besov spaces. Then  we investigate the
relationship between the fractional Sobolev spaces defined as
Definition \ref{def-sobolev} and the Besov spaces defined as
Definition \ref{Besov1}. Finally, we also study the equivalence of
the Besov spaces defined by fractional heat semigroups and
Caffarelli-Silvestre extensions, respectively.
\begin{proposition}\label{desity}
Let $1\leq p\leq q<\infty$  and $(\alpha,\sigma)\in(0,1)\times(0,1).$ The following statements hold:
  \item{\rm (i)} For $2\alpha>p$ and $\beta\in(0,1/\alpha)$, $C_{c}^{\infty}(\X)$ is dense in $B_{p,q,\mathcal{H}}^{{\alpha},\beta}(\X).$ That is, for any $u\in B_{p,q,\mathcal{H}}^{{\alpha},\beta}(\X),$
      there exist a sequence $\{u_{j}\}_{j\in\mathbb{N}}\subset C_{c}^{\infty}(\X)$
      such that $$\lim_{j\rightarrow\infty}\|u_{j}-u\|_{B_{p,q,\mathcal{H}}^{{\alpha},\beta}(\X)}=0.$$
  Moreover, if $u\geq 0$, then all the approximating sequence $\{u_{j}\}_{j\in\mathbb{N}}$ can be chosen to be non-negative. In particular, if $\mathcal{K}\subset\X$ is a compact set and $u\geq 1_{\mathcal{K}}$, then the approximating sequence $\{u_{j}\}_{j\in\mathbb{N}}$ in $C_{c}^{\infty}(\X)$ can be chosen to satisfy $u_{j}\geq 1_{\mathcal{K}}$ for all $j\in\mathbb{N}.$
  \item{\rm (ii)} For the Besov type space $B_{p,q,\mathcal{P}}^{ {\sigma},\beta}(\X)$, the above results also holds, where $2\sigma>p$ and $\beta\in(0,2)$.

\end{proposition}
\begin{proof}
We only give the proof of (i). The proof of (ii) is analogous. In order to prove Proposition \ref{desity} (i), we need a constructive approach
and show the density in the following two steps.

\emph{Step 1:\ Approximation of $u\in
B_{p,q,\mathcal{H}}^{{\alpha},\beta}(\X)$ by
$C^{\infty}(\X)$-functions.} Let $\{\tau_{\epsilon}\}_{\epsilon>0}$
be a family of standard mollifiers in Proposition \ref{dense-1}. For
any $\epsilon\in(0,\infty),$ define
$u_{\epsilon}:=\tau_{\epsilon}\ast u,$ which is a function in
$C^{\infty}(\X).$ Clearly, we have
$\|u_{\epsilon}-u\|_{L^{p}(\X)}\rightarrow 0$ as
$\epsilon\rightarrow0,$ and consequently we only need to prove
\begin{equation}\label{conver}
N_{p,q,\mathcal{H}}^{{\alpha},\beta}(u-u_{\epsilon})\rightarrow0\
\textrm{as}\ \epsilon\rightarrow0.
\end{equation}
To do so, we write
\begin{align*}
&N_{p,q,\mathcal{H}}^{{\alpha},\beta}(u-u_{\epsilon})^{q}\\
&=\int_{0}^{\infty}\frac{1}{t^{{\beta q}/{2}+1}}
\Big(\int_{\X}H_{\alpha,t}(|u-u_{\epsilon}-u(g)+u_{\epsilon}(g)|^{p})(g)dg\Big)^{q/p}dt\\
&=\int_{0}^{1}\frac{1}{t^{{\beta q}/{2}+1}}
\Big(\int_{\X}\int_{\X}K_{\alpha,t}(g'^{-1}g)
|u(g')-u_{\epsilon}(g')-u(g)+u_{\epsilon}(g)|^{p}dg'dg\Big)^{q/p}dt\\
&\quad+\int_{1}^{\infty}\frac{1}{t^{{\beta q}/{2}+1}}
\Big(\int_{\X}\int_{\X}K_{\alpha,t}(g'^{-1}g)
|u(g')-u_{\epsilon}(g')-u(g)+u_{\epsilon}(g)|^{p}dg'dg\Big)^{q/p}dt\\
&=:I+II.
\end{align*}
Using Lemma \ref{fracheat-1} (i), we have
\begin{align*}
II&\lesssim \int_{1}^{\infty}\frac{1}{t^{{\beta q}/{2}+1}}
\Big(\int_{\X}\int_{\X}K_{\alpha,t}(g'^{-1}g)
|u(g')-u_{\epsilon}(g')|^{p}dgdg'\Big)^{q/p}dt\\
&\quad +\int_{1}^{\infty}\frac{1}{t^{{\beta q}/{2}+1}}
\Big(\int_{\X}\int_{\X}K_{\alpha,t}(g'^{-1}g)
|u(g)-u_{\epsilon}(g)|^{p}dg'dg\Big)^{q/p}dt\\
&\lesssim \int_{1}^{\infty}\frac{1}{t^{{\beta q}/{2}+1}}
\Big(\int_{\X}
|u(g')-u_{\epsilon}(g')|^{p}dg'\Big)^{q/p}dt\\
&\quad +\int_{1}^{\infty}\frac{1}{t^{{\beta q}/{2}+1}}
\Big(\int_{\X}
|u(g)-u_{\epsilon}(g)|^{p}dg\Big)^{q/p}dt\rightarrow0\ \textrm{as}\ \epsilon\rightarrow0.
\end{align*}

In what follows, we estimate the first term $I$.
For any $g,g'\in\X$ and $t\in(0,1),$ we write
\begin{align*}
&u(g')-u_{\epsilon}(g')-u(g)+u_{\epsilon}(g)\\
&=u(g')-u(g)-\int_{\X}u(g_{2})\tau(\bar{g})d\bar{g}
+\int_{\X}u(g_{1})\tau(\bar{g})d\bar{g}\\
&=u(g')-u(g)-\int_{\X}\big(u(g_{2})-u(g_{1})\big)\tau(\bar{g})d\bar{g},
\end{align*}
where $g_{1}=(\delta_{\epsilon}\bar{g})^{-1}g$ and $g_{2}=(\delta_{\epsilon}\bar{g})^{-1}g'$.
Consequently,
\begin{align*}
&K_{\alpha,t}(g'^{-1}g)|u(g')-u_{\epsilon}(g')-u(g)+u_{\epsilon}(g)|^{p}\\
&=\Bigg|K_{\alpha,t}(g'^{-1}g)^{{1}/{p}}(u(g')-u(g))-\int_{\X}
K_{\alpha,t}(g_{2}^{-1}g_{1})^{{1}/{p}}\big(u(g_{2})-
u(g_{1})\big)\tau(\bar{g})d\bar{g}\Bigg|^{p}.
\end{align*}
Now, we set
$$\hbar_{u}(g,g',t)=t^{-{\beta}/{2}-{1}/{q}}K_{\alpha,t}(g'^{-1}g)^{{1}/{p}}(u(g')-u(g)).$$
Notice that $\textrm{supp}\ \tau\subset B(e,1)$ and
$\|\tau\|_{L^{1}(\X)}=1.$ Therefore, it follws from H\"older's
inequality that
\begin{align*}
&t^{-{\beta p}/{2}-{p}/{q}}K_{\alpha,t}(g'^{-1}g)|u(g')-u_{\epsilon}(g')-u(g)+u_{\epsilon}(g)|^{p}\\
&\lesssim\big|\hbar_{u}(g,g',t)-\int_{\X}
\hbar_{u}(g_{1},g_{2},t)\tau(\bar{g})d\bar{g}\big|^{p}\\
&=\bigg|\int_{B(e,1)}\big(\hbar_{u}(g,g',t)-\hbar_{u}(g_{1},g_{2},t)
\big)\tau(\bar{g})d\bar{g}\bigg|^{p}\\
&\lesssim\int_{B(e,1)}\big|\hbar_{u}(g,g',t)-\hbar_{u}(g_{1},g_{2},t)
\big|^{p}\tau(\bar{g})^{p}d\bar{g}.
\end{align*}
Further,  using Minkowski's inequality, we have
\begin{align*}
I&=\int_{0}^{1}\frac{1}{t^{{\beta q}/{2}+1}}
\Big(\int_{\X}\int_{\X}K_{\alpha,t}(g'^{-1}g)
|u(g')-u_{\epsilon}(g')-u(g)+u_{\epsilon}(g)|^{p}dg'dg\Big)^{q/p}dt\\
&\lesssim \bigg(\int_{B(e,1)}\tau(\bar{g})^{p}\Big(\int_{0}^{1}
\|\hbar_{u}(\cdot,\cdot,t)-\hbar_{u}((\delta_{\epsilon}\bar{g})^{-1}\cdot,(\delta_{\epsilon}\bar{g})^{-1}\cdot,t)\|
^{q}_{L^{p}(\X\times\X)}dt\Big)^{p/q}d\bar{g}\bigg)^{q/p}.
\end{align*}
From the definition (\ref{Besov}) of $N_{p,q,\mathcal{H}}^{
{\alpha},\beta}(\cdot),$ it is easy to verify that
$$\big(N_{p,q,\mathcal{H}}^{{\alpha},\beta}(u)\big)^{q}
=\big\|\|\hbar_{u}(\cdot,\cdot,\cdot)\|^{q}_{L^{p}(\X\times\X)}\big\|_{L^{1}(0,\infty)},$$
which implies that
$$u\in B_{p,q,\mathcal{H}}^{{\alpha},\beta}(\X)
\Longleftrightarrow \|\hbar_{u}(\cdot,\cdot,t)\|^{q}_{L^{p}(\X\times\X)}
\in L^{1}(0,\infty).$$
Therefore, for almost any $t\in(0,1),$ we have
$$\|\hbar_{u}(\cdot,\cdot,t)-\hbar_{u}((\delta_{\epsilon}\bar{g})^{-1}\cdot,(\delta_{\epsilon}\bar{g})^{-1}\cdot,t)\|
^{q}_{L^{p}(\X\times\X)}\rightarrow0\
\textrm{as}\ \epsilon\rightarrow0.$$
On the other hand, we know that
$$\|\hbar_{u}(\cdot,\cdot,t)-\hbar_{u}((\delta_{\epsilon}\bar{g})^{-1}\cdot,(\delta_{\epsilon}\bar{g})^{-1}\cdot,t)\|
^{q}_{L^{p}(\X\times\X)}\lesssim
\|\hbar_{u}(\cdot,\cdot,t)\|^{q}_{L^{p}(\X\times\X)}\in
L^{1}(0,1).$$ Therefore,  the Lebesgue dominated convergence theorem
ensures that the first term $I$ converges to zero as
$\epsilon\rightarrow0.$ This completes the proof of (\ref{conver}).

\emph{Step 2:\ Approximation of $u\in C^{\infty}(\X)\cap
B_{p,q,\mathcal{H}}^{{\alpha},\beta}(\X)$ by
$C^{\infty}_{c}(\X)$-functions.} Let $\eta$ be a radial decreasing
function defined as Proposition \ref{dense-1}. Clearly, we have
$u_{N}:=\eta_{N}u\in C^{\infty}_{c}(\X)$ and
$\|u_{N}-u\|_{L^{p}(\X)}\rightarrow 0$ as $N\rightarrow\infty.$
Thus we are required to validate
\begin{align}\label{conver1}
N_{p,q,\mathcal{H}}^{{\alpha},\beta}(u-u_{N})\rightarrow0\
\textrm{as}\ N\rightarrow\infty.
\end{align}
By the definition of $N_{p,q,\mathcal{H}}^{{\alpha},\beta}(\cdot)$,
we write
\begin{align*}
&N_{p,q,\mathcal{H}}^{{\alpha},\beta}(u-u_{N})^{q}\\
&=\int_{0}^{\infty}
\Big(\int_{\X}\int_{\X}\frac{K_{\alpha,t}(g'^{-1}g)}{t^{{\beta p}/{2}+{p}/{q}}}
|u(g')-u_{N}(g')-u(g)+u_{N}(g)|^{p}dg'dg\Big)^{q/p}dt\\
&=:\int_{0}^{\infty}\Big(\int_{\X}\int_{\X}
\ell_{N}(g,g',t)dg'dg\Big)^{q/p}dt.
\end{align*}
For any $(g,g',t)\in\X\times\X\times(0,\infty),$ we conclude that
$\ell_{N}(g,g',t)\rightarrow0$ as $N\rightarrow\infty.$ Similarly to
\emph{Step 1}, we will discuss (\ref{conver1}) in two parts: $t>1$
and $t\leq1$. On the one hand, it follows from Lemma
\ref{fracheat-1} (i) that
\begin{align*}
&\int_{1}^{\infty}\Big(\int_{\X}\int_{\X}
\ell_{N}(g,g',t)dg'dg\Big)^{q/p}dt\\
&\lesssim\int_{1}^{\infty}\Big(\int_{\X}\int_{\X}
\frac{K_{\alpha,t}(g'^{-1}g)}{t^{{\beta p}/{2}+{p}/{q}}}\big((1-\eta_{N}(g'))^{p}|u(g')|^{p}
+(1-\eta_{N}(g))^{p}|u(g)|^{p}\big)dg'dg\Big)^{q/p}dt\\
&\lesssim\int_{1}^{\infty}\Big(\int_{\X}\int_{\X}
\frac{K_{\alpha,t}(g'^{-1}g)}{t^{{\beta p}/{2}+{p}/{q}}}(|u(g')|^{p}+|u(g)|^{p})dg'dg\Big)^{q/p}dt<\infty.
\end{align*}
On the other hand, let
$B(e,2):=\{\bar{g}\in\X:d(e,\bar{g})\leq2\}$. Then for every
$N\gg1,$ we obtain
\begin{align*}
&\int_{0}^{1}\Big(\int_{\X}\int_{\X}
\ell_{N}(g,g',t)dg'dg\Big)^{q/p}dt\\
&\lesssim\int_{0}^{1}\Big(\int_{\X}\int_{\X}
\frac{K_{\alpha,t}(g'^{-1}g)}{t^{{\beta p}/{2}+{p}/{q}}}\big((1+\eta_{N}(g')^{p})|u(g')-u(g)|^{p}\\
&\quad+|\eta_{N}(g)-\eta_{N}(g')|^{p}|u(g)|^{p}\big)dg'dg\Big)^{q/p}dt\\
&\lesssim\int_{0}^{1}\Big(\int_{\X}\int_{\X}
|\hbar_{u}(g,g',t)|^{p}\\
&\quad+(\mathbf{1}_{B(e,2)}(\delta_{N}
g)+\mathbf{1}_{B(e,2)}(\delta_{N} g'))
N^{-p}\frac{K_{\alpha,t}(g'^{-1}g)}{t^{{\beta p}/{2}+{p}/{q}}}
d(g,g')^{p}|u(g)|^{p}dg'dg\Big)^{q/p}dt\\
&\lesssim\int_{0}^{1}\Big(\int_{\X}\int_{\X}
|\hbar_{u}(g,g',t)|^{p}dg'dg\Big)^{q/p}dt+
\int_{0}^{1}\Big(\int_{\X}
\int_{\X}\frac{K_{\alpha,t}(g'^{-1}g)}{t^{{\beta p}/{2}+{p}/{q}}}
d(g,g')^{p}|u(g)|^{p}dg'dg\Big)^{q/p}dt,
\end{align*}
where we have used the fact that
$$|\eta_{N}(g)-\eta_{N}(g')|^{p}\leq
N^{-p}\|\nabla_{\X} \eta\|^{p}_{\infty}d(g,g')^{p}.$$ Since $u\in
B_{p,q,\mathcal{H}}^{{\alpha},\beta}(\X),$ we have
$\|\hbar_{u}(\cdot,\cdot,t)\|^{q}_{L^{p}(\X\times\X)} \in
L^{1}(0,\infty).$ Therefore, the first integral above is finite. In
addition, it follows from Proposition \ref{pro-frac} that
\begin{align*}
&\int_{0}^{1}\Big(\int_{\X}
\int_{\X}\frac{K_{\alpha,t}(g'^{-1}g)}{t^{{\beta p}/{2}+{p}/{q}}}
d(g,g')^{p}|u(g)|^{p}dg'dg\Big)^{q/p}dt\\
&\lesssim\int_{0}^{1}\frac{1}{t^{{\beta q}/{2}+1}}\Big(\int_{\X}
\int_{\X}K_{\alpha,t}(\bar{g})
d(e,\bar{g})^{p}|u(g)|^{p}d\bar{g}dg\Big)^{q/p}dt\\
&\lesssim\int_{0}^{1}\frac{t^{{q}/{(2\alpha)}}}{t^{{\beta q}/{2}+1}}\Big(\int_{\X}
\int_{\X}\frac{t}{(t^{1/(2\alpha)}+d(e,\bar{g}))^{\Q+2\alpha}}
\frac{d(e,\bar{g})^{p}}{t^{{p}/{(2\alpha)}}}|u(g)|^{p}d\bar{g}dg\Big)^{q/p}dt\\
&\lesssim\|u\|^{q}_{L^{p}(\X)}\int_{0}^{1}t^{\frac{q}{2}(1/\alpha-\beta)-1}dt
<\infty,
\end{align*}
where we have used the fact that $1/\alpha>\beta>0$ and $2\alpha>p.$

This implies that
$$\int_{0}^{1}\Big(\int_{\X}\int_{\X}
\ell_{N}(g,g',t)dg'dg\Big)^{q/p}dt<\infty.$$
Using  the Lebesgue dominated convergence theorem, we obtained the desired result (\ref{conver1}).

Owing to the arguments in \emph{Steps 1-2}, for any $u\in
B_{p,q,\mathcal{H}}^{{\alpha},\beta}(\X)$ letting
$u_{\epsilon,N}:=\eta_{N}(\tau_{\epsilon}\ast u),$ we get the
desired limit
$$\lim_{N\rightarrow\infty}\lim_{\epsilon\rightarrow0}\|u_{\epsilon,N}-u\|_{B_{p,q,\mathcal{H}}^{{\alpha},\beta}(\X)}=0.$$
The proof of the rest of (i) is similar to that of Proposition
\ref{dense-1}, and so is omitted.

\end{proof}

The following results provide the min-max property of Besov spaces.
\begin{lemma}\label{max}
Let $u_{1},u_{2}\in B_{p,q,\mathcal{H}}^{{\alpha},\beta}(\X).$ Then
the functions $G:=\max\{u_{1},u_{2}\}$ and $H:=\min\{u_{1},u_{2}\}$
belong to $B_{p,p,\mathcal{H}}^{{\alpha},\beta}(\X)$ with
$$\|G\|^{p}_{B_{p,p,\mathcal{H}}^{{\alpha},\beta}(\X)}+\|H\|^{p}_{B_{p,p,\mathcal{H}}^{{\alpha},\beta}(\X)}\leq \|u_{1}\|^{p}_{B_{p,p,\mathcal{H}}^{{\alpha},\beta}(\X)}+\|u_{2}\|^{p}_{B_{p,p,\mathcal{H}}^{{\alpha},\beta}(\X)}.$$
The above result also holds for
$B_{p,q,\mathcal{P}}^{{\sigma},\beta}(\X)$ as well.
\end{lemma}
\begin{proof}
Let $D_{1}:=\{g\in\X:u_{1}(g)\geq u_{2}(g)\}$ and
$D_{2}:=\{g\in\X:u_{1}(g)< u_{2}(g)\}.$ Then
\begin{align*}
&\|G\|^{p}_{L^{p}(\X)}+\|H\|^{p}_{L^{p}(\X)}\\
&=\bigg(\int_{D_{1}}|u_{1}(g)|^{p}dg+\int_{D_{2}}|u_{2}(g)|^{p}dg\bigg)
+\bigg(\int_{D_{1}}|u_{2}(g)|^{p}dg+\int_{D_{2}}|u_{1}(g)|^{p}dg\bigg)\\
&=\|u_{1}\|^{p}_{L^{p}(\X)}+\|u_{2}\|^{p}_{L^{p}(\X)}.
\end{align*}
Thus  we only need to prove
\begin{equation}\label{semi}
(N_{p,p,\mathcal{H}}^{{\alpha},\beta}
(G))^{p}+(N_{p,p,\mathcal{H}}^{{\alpha},\beta} (H))^{p}\leq
(N_{p,p,\mathcal{H}}^{{\alpha},\beta}
(u_{1}))^{p}+(N_{p,p,\mathcal{H}}^{{\alpha},\beta} (u_{2}))^{p}.
\end{equation}
By calculation, we have
\begin{align*}
(N_{p,p,\mathcal{H}}^{{\alpha},\beta}
(G))^{p}&=\int_{0}^{\infty}\dfrac{1}{t^{{\beta p}/{2}+1}}\int_{\X}
H_{\alpha,t}(|G-G(g')|^{p})(g')dg'dt\\
&=\int_{0}^{\infty}\dfrac{1}{t^{{\beta p}/{2}+1}}\int_{\X}\int_{\X}K_{\alpha,t}(g'^{-1}g)|G(g)-G(g')|^{p}dgdg'dt\\
&=\int_{0}^{\infty}\dfrac{1}{t^{{\beta p}/{2}+1}}\int_{D_{1}}\int_{D_{1}}K_{\alpha,t}(g'^{-1}g)|u_{2}(g)-u_{2}(g')|^{p}dgdg'dt\\
&\qquad+\int_{0}^{\infty}\dfrac{1}{t^{{\beta p}/{2}+1}}\int_{D_{1}}\int_{D_{2}}K_{\alpha,t}(g'^{-1}g)|u_{1}(g)-u_{2}(g')|^{p}dgdg'dt\\
&\qquad+\int_{0}^{\infty}\dfrac{1}{t^{{\beta p}/{2}+1}}\int_{D_{2}}\int_{D_{1}}K_{\alpha,t}(g'^{-1}g)|u_{2}(g)-u_{1}(g')|^{p}dgdg'dt\\
&\qquad+\int_{0}^{\infty}\dfrac{1}{t^{{\beta p}/{2}+1}}\int_{D_{2}}\int_{D_{2}}K_{\alpha,t}(g'^{-1}g)|u_{1}(g)-u_{1}(g')|^{p}dgdg'dt
\end{align*}
and
\begin{align*}
(N_{p,p,\mathcal{H}}^{{\alpha},\beta} (H))^{p}
&=\int_{0}^{\infty}\dfrac{1}{t^{{\beta p}/{2}+1}}\int_{D_{1}}\int_{D_{1}}K_{\alpha,t}(g'^{-1}g)|u_{1}(g)-u_{1}(g')|^{p}dgdg'dt\\
&\qquad+\int_{0}^{\infty}\dfrac{1}{t^{{\beta p}/{2}+1}}\int_{D_{1}}\int_{D_{2}}K_{\alpha,t}(g'^{-1}g)|u_{2}(g)-u_{1}(g')|^{p}dgdg'dt\\
&\qquad+\int_{0}^{\infty}\dfrac{1}{t^{{\beta p}/{2}+1}}\int_{D_{2}}\int_{D_{1}}K_{\alpha,t}(g'^{-1}g)|u_{1}(g)-u_{2}(g')|^{p}dgdg'dt\\
&\qquad+\int_{0}^{\infty}\dfrac{1}{t^{{\beta p}/{2}+1}}\int_{D_{2}}\int_{D_{2}}K_{\alpha,t}(g'^{-1}g)|u_{2}(g)-u_{2}(g')|^{p}dgdg'dt.
\end{align*}
Then
\begin{align*}
&(N_{p,p,\mathcal{H}}^{{\alpha},\beta} (G))^{p}+(N_{p,p,\mathcal{H}}^{{\alpha},\beta} (H))^{p}\\
&=\bigg(\int_{0}^{\infty}\dfrac{1}{t^{{\beta p}/{2}+1}}\int_{D_{1}}\int_{D_{1}}+\int_{0}^{\infty}\dfrac{1}{t^{{\alpha p}/{2}}}\int_{D_{2}}\int_{D_{2}}\bigg)K_{\alpha,t}(g'^{-1}g)|u_{1}(g)-u_{1}(g')|^{p}dgdg'dt\\
&\qquad+\int_{0}^{\infty}\dfrac{1}{t^{{\beta p}/{2}+1}}\int_{D_{2}}\int_{D_{1}}K_{\alpha,t}(g'^{-1}g)(|u_{1}(g)-u_{2}(g)'|^{p}+|u_{2}(g)-u_{1}(g')|^{p})dgdg'dt\\
&\qquad+\bigg(\int_{0}^{\infty}\dfrac{1}{t^{{\beta p}/{2}+1}}\int_{D_{1}}\int_{D_{1}}+\int_{0}^{\infty}\dfrac{1}{t^{{\alpha p}/{2}}}\int_{D_{2}}\int_{D_{2}}\bigg)K_{\alpha,t}(g'^{-1}g)|u_{2}(g)-u_{2}(g')|^{p}dgdg'dt\\
&\qquad+\int_{0}^{\infty}\dfrac{1}{t^{{\beta p}/{2}+1}}\int_{D_{1}}\int_{D_{2}}K_{\alpha,t}(g'^{-1}g)(|u_{2}(g)-u_{1}(g')|^{p}+|u_{1}(g)-u_{2}(g')|^{p})dgdg'dt\\
&=:I_{1}+I_{2}+I_{3}+I_{4}.
\end{align*}
Next we deal with $I_{2}$ and $I_{4}$   using \cite[Lemma
3.3]{SW2010}, which says that for
any convex and lower-semi-continuous function
$F:\mathbb{R}^{2}\rightarrow(-\infty,+\infty),$ we have
$$F(x_{0},y_{1})+F(y_{0},x_{1})\leq F(x_{0},x_{1})+F(y_{0},y_{1})~\mbox{as}~(x_{0}-y_{0})(x_{1}-y_{1})\leq0.$$
If we take $F(r,s):=|r-s|^{p},$ then for any $g\in D_{1}$ and $g'\in
D_{2},$ we have
\begin{align*}
|u_{2}(g)-u_{1}(g')|^{p}+|u_{1}(g)-u_{2}(g')|^{p}
&\leq F(u_{2}(g),u_{2}(g'))+F(u_{1}(g),u_{1}(g'))\\
&=|u_{2}(g)-u_{2}(g')|^{p}+|u_{1}(g)-u_{1}(g')|^{p}.
\end{align*}
Likewise, for any $g\in D_{2}$ and $g'\in D_{1},$ we obtain
$$|u_{1}(g)-u_{2}(g')|^{p}+|u_{2}(g)-u_{1}(g')|^{p}\leq|u_{1}(g)-u_{1}(g')|^{p}+|u_{2}(g)-u_{2}(g')|^{p}.$$
Combining the last two formulas  gives
\begin{align*}
I_{2}+I_{4}
&\leq\bigg(\int_{0}^{\infty}\dfrac{1}{t^{{\beta p}/{2}+1}}\int_{D_{1}}\int_{D_{2}}+\int_{0}^{\infty}\dfrac{1}{t^{{\beta p}/{2}+1}}\int_{D_{2}}\int_{D_{1}}\bigg)\\
&\times K_{\alpha,t}(g'^{-1}g)(|u_{1}(g)-u_{1}(g')|^{p}+|u_{2}(g)-u_{2}(g')|^{p})dgdg'dt.
\end{align*}
This,   together  with the expressions of $I_{1}$ and $I_{3}$,
implies (\ref{semi}) immediately.
\end{proof}

Below we give some results about the boundedness properties of $\mathcal{L}^{\mathbf{s}}$ and $\mathcal{L}^{\widetilde{\mathbf{s}}}.$
\begin{proposition}\label{boundedness}
Let $(s,\alpha,\sigma)\in(0,1)\times(0,1)\times(0,1) $,
$\mathbf{s}=s\alpha$ and $\widetilde{\mathbf{s}}=s\sigma$.

\item{\rm (i)} For $p\in(1,\infty)$ and $\beta>2s,$ we have
$$\mathcal{L}^{\mathbf{s}}:B_{p,p,\mathcal{H}}^{{\alpha},\beta}(\X)\hookrightarrow L^{p}(\X).$$
For $p=1,$ the above results are also valid for $\beta\geq 2s.$

\item{\rm (ii)} For $p\in(1,\infty)$ and $\beta>2s,$ we have
$$\mathcal{L}^{\widetilde{\mathbf{s}}}:B_{p,p,\mathcal{P}}^{{\sigma},\beta}(\X)\hookrightarrow L^{p}(\X).$$
For $p=1,$ the above results are also valid for $\beta\geq 2s.$

\end{proposition}
\begin{proof}
We only give the proof of (i). Since the proof of (ii) is similar to
(i), we omit the details. For $p\in[1,\infty)$ and $u\in
B_{p,p,\mathcal{H}}^{{\alpha},\beta}(\X),$ we deduce from (\ref{-G})
that
\begin{align*}
\|\mathcal{L}^{\mathbf{s}}u\|_{L^{p}(\X)}&\leq\frac{s}{\Gamma(1-s)}\int_{0}^{\infty}
t^{-(1+s)}\|H_{\alpha,t}u-u\|_{L^{p}(\X)}dt= I+II,
\end{align*}
where
\begin{align*}
  \left\{\begin{aligned}
  I&:=\frac{s}{\Gamma(1-s)}\int_{0}^{1}t^{-(1+s)}\|H_{\alpha,t}u-u\|_{L^{p}(\X)}dt;\\
  II&:=\frac{s}{\Gamma(1-s)}\int_{1}^{\infty}t^{-(1+s)}\|H_{\alpha,t}u-u\|_{L^{p}(\X)}dt.
  \end{aligned}\right.
\end{align*}
Using Lemma \ref{fracheat-1} (iii), we obtain
$$II\lesssim \|u\|_{L^{p}(\X)}.$$
Next, we will consider the term $I.$
It follows from H\"older's inequality and Lemma \ref{fracheat-1} (i) that
\begin{align}\label{ineq20}
\|H_{\alpha,t}u-u\|_{L^{p}(\X)}
&= \bigg(\int_{\X}|H_{\alpha,t}u(g)-u(g)|^{p}dg\bigg)^{1/p}\nonumber\\
&=\bigg(\int_{\X}\bigg|\int_{\X}
K_{\alpha,t}(g'^{-1}g)(u(g')-u(g))dg'\bigg|^{p}dg\bigg)^{1/p}\nonumber\\
&\leq
\bigg(\int_{\X}H_{\alpha,t}(|u-u(g)|^{p})(g)dg\bigg)^{1/p}.
\end{align}
If $p=1$ and $\beta\geq2s,$ we can obtain from (\ref{ineq20}) that
\begin{align*}
&\int_{0}^{1}t^{-(1+s)}\|H_{\alpha,t}u-u\|_{L^{1}(\X)}dt\\
&\leq \int_{0}^{1}\frac{t^{{\beta}/{2}-s}}{t^{{\beta}/{2}+1}}\Bigg(
\int_{\X}H_{\alpha,t}(|u-u(g)|)(g)dg\Bigg)dt\\
&\leq \int_{0}^{1}\frac{1}{t^{{\beta}/{2}+1}}
\Bigg(\int_{\X}H_{\alpha,t}(|u-u(g)|)(g)dg\Bigg)dt\\
&=\widetilde{N}_{1,1,\mathcal{H}}^{{\alpha},\beta}(u)<\infty,
\end{align*}
where we have used  Remark \ref{rem2} (i) in the last step.

If $p\in(1,\infty)$ and $\beta>2s,$ we use H\"older's inequality and
(\ref{ineq20}) to give
\begin{align*}
&\int_{0}^{1}t^{-(1+s)}\|H_{\alpha,t}u-u\|_{L^{p}(\X)}dt\\
&=\int_{0}^{1}t^{-s+{\beta}/{2}-{\beta}/{2}}\|H_{\alpha,t}u-u\|_{L^{p}(\X)}\frac{dt}{t}\\
&\leq\bigg(\int_{0}^{1}t^{-(s-{\beta}/{2})p'}\frac{dt}{t}\bigg)^{1/p'}
\bigg(\int_{0}^{1}t^{-{\beta p}/{2}}\|H_{\alpha,t}u-u\|^{p}_{L^{p}(\X)}\frac{dt}{t}\bigg)^{1/p}\\
&\leq\bigg(\int_{0}^{1}t^{-(s-{\beta}/{2})p'-1}dt\bigg)^{1/p'}\bigg(\int_{0}^{1}\frac{1}{t^{{\beta p}/{2}+1}}
\int_{\X}H_{\alpha,t}(|u-u(g)|^p{})(g)dg\bigg)^{1/p}\\
&\lesssim
\widetilde{N}_{p,p,\mathcal{H}}^{{\alpha},\beta}(u)<\infty.
\end{align*}
According to  Remark \ref{rem2} (i), we obtain the desired
conclusion.
\end{proof}

\begin{proposition}\label{boundedness1}
Let $(s,\alpha,\sigma)\in(0,1)\times(0,1)\times(0,1),$
$\mathbf{s}=s\alpha$ and $\widetilde{\mathbf{s}}=s\sigma$.

\item{\rm (i)}
For $(p,q,\beta)\in[1,\infty)\times\{\infty\}\times(2s,\infty)$ we have
$$\mathcal{L}^{\mathbf{s}}:B_{p,\infty,\mathcal{H}}^{{\alpha},\beta}(\X)\hookrightarrow L^{p}(\X).$$

\item{\rm (ii)}
For $(p,q,\beta)\in[1,\infty)\times\{\infty\}\times(2s,\infty)$ we have
$$\mathcal{L}^{\widetilde{\mathbf{s}}}:B_{p,\infty,\mathcal{P}}^{{\sigma},\beta}(\X)\hookrightarrow L^{p}(\X).$$
\end{proposition}
\begin{proof} We only give the proof of (i). The proof of (ii) is similar.
For any $u\in B_{p,\infty,\mathcal{H}}^{{\alpha},\beta}(\X),$ it
follows from (\ref{ineq20}) that
$$\|H_{\alpha,t}u-u\|_{L^{p}(\X)}\leq t^{\beta/2}N_{p,\infty,\mathcal{H}}^{{\alpha},\beta}(u),$$
which, together with (\ref{fraheat-1}) and Lemma \ref{fracheat-1}
(iii), yields
\begin{align*}
\|\mathcal{L}^{\mathbf{s}}u\|_{L^{p}(\X)}
&\leq\frac{s}{\Gamma(1-s)}\int_{0}^{\infty}
t^{-1-s}\|H_{\alpha,t}u-u\|_{L^{p}(\X)}dt\\
&\lesssim \int_{0}^{1}t^{-1-s}\|H_{\alpha,t}u-u\|_{L^{p}(\X)}dt
+\int_{1}^{\infty}t^{-1-s}\|H_{\alpha,t}u-u\|_{L^{p}(\X)}dt\\
&\lesssim
N_{p,\infty,\mathcal{H}}^{{\alpha},\beta}(u)\int_{0}^{1}t^{-1-s+\beta/2}dt
+\|u\|_{L^{p}(\X)}\int_{1}^{\infty}t^{-1-s}dt\\
&\lesssim
N_{p,\infty,\mathcal{H}}^{{\alpha},\beta}(u)+\|u\|_{L^{p}(\X)}<\infty.
\end{align*}
This completes the proof.
\end{proof}

Combining Propositions \ref{boundedness} \&\
\ref{boundedness1}, we can directly obtain
the following corollary.
\begin{corollary}\label{BW}
Let $(s,\alpha,\sigma)\in(0,1)\times(0,1)\times(0,1),$
$\mathbf{s}=s\alpha$ and $\widetilde{\mathbf{s}}=s\sigma$. The
following statements   hold:

\item{\rm (i)}  If $p=q\in(1,\infty)$ and $\beta>2s$ then
$$B_{p,p,\mathcal{H}}^{{\alpha},\beta}(\X)\hookrightarrow \mathcal{W}^{2\mathbf{s},p}(\X)\ \ \ \&\ \ \
 B_{p,p,\mathcal{P}}^{{\sigma},\beta}(\X)\hookrightarrow \mathcal{W}^{2\widetilde{\mathbf{s}},p}(\X).$$
                  In particular, when $p=q=1$ and $\beta\geq2s,$ we have
                  $$B_{1, 1,\mathcal{H}}^{{\alpha},\beta}(\X)\hookrightarrow \mathcal{W}^{2\mathbf{s},1}(\X)\ \ \ \&\ \ \
                  B_{1,1,\mathcal{P}}^{{\sigma},\beta}(\X)\hookrightarrow \mathcal{W}^{2\widetilde{\mathbf{s}},1}(\X).$$

\item{\rm  (ii)} If $p\in[1,\infty),\, q=\infty$ and $\beta>2s$, then
  $$B_{p,\infty,\mathcal{H}}^{{\alpha},\beta}(\X)\hookrightarrow \mathcal{W}^{2\mathbf{s},p}(\X)\ \ \ \&\ \
  \ B_{p,\infty,\mathcal{P}}^{{\sigma},\beta}(\X)\hookrightarrow \mathcal{W}^{2\widetilde{\mathbf{s}},p}(\X).$$

\end{corollary}

Below, we provide the connection between
$B_{p,q,\mathcal{H}}^{{\alpha},\beta}(\X)$, $B_{p,q,\mathcal{P}}^{
{\sigma},\beta}(\X)$  and Besov spaces $B_{p,q}^{\beta}(\X)$ defined
via the difference. To be precise, for any $p,q\in[1,\infty)$ and
$\beta>0,$ the space $B_{p,q}^{\beta}(\X)$ is defined to be the
collection of all functions $u\in L^{p}(\X)$ such that
\begin{align*}\label{Dbesov}
N_{p,q}^{\beta}(u):=\Big(\int_{0}^{\infty}\Big(\int_{\X}\int_{B(g,r)}\frac{|u(g)-u(g')|^{p}}{r^{2\beta p+\Q}}dg'dg\Big)^{q/p}\frac{dr}{r}\Big)^{1/q}<\infty,
\end{align*}
When $q=\infty$, then $B_{p,\infty}^{\beta}(\X)$ is defined to be the collection of all
$u\in L^{p}(\X)$ such that
$$N_{p,\infty}^{\beta}(u):=\sup_{r>0}\Big(\int_{\X}\int_{B(g,r)}\frac{|u(g)-u(g')|^{p}}{r^{2\beta p+\Q}}dg'dg\Big)^{1/p}<\infty.$$
The space $B_{p,q}^{\beta}(\X)$ can be endowed
with the norm
$$\|u\|_{B_{p,q}^{\beta}(\X)}
:=\|u\|_{L^{p}(\X)}+N_{p,q}^{\beta}(u).$$

\begin{proposition}\label{com-2}
Let $p,q\in[1,\infty)$ and $\beta\in(0,1/p)$.

\item{\rm (i)} For $\alpha\in(0,1),$
$$B_{p,q,\mathcal{H}}^{\alpha{ },2\beta}(\X)=B_{p,q}^{\alpha\beta}(\X).$$

\item{\rm (ii)} For $\sigma\in(0,1),$
$$B_{p,q,\mathcal{P}}^{{\sigma},2\beta}(\X)=B_{p,q}^{\sigma\beta}(\X).$$
\end{proposition}
\begin{proof}We only give the proof of (i). Since the proof of (ii) is similar to (i), we omit the details.
Similarly to the argument of Remark \ref{rem2}, we conclude  that
\begin{equation}\label{Ne}
N_{p,q}^{\beta}(u)<\infty\Leftrightarrow \Big(\int_{0}^{1}\Big(\int_{\X}\int_{B(g,r)}\frac{|u(g)-u(g')|^{p}}{r^{2\beta p+\Q}}dg'dg\Big)^{q/p}\frac{dr}{r}\Big)^{1/q}<\infty.
\end{equation}
Fix a decreasing geometric sequence $\{r_{k}\}_{k\in\mathbb{Z}_{+}}$ and observe that by (\ref{Ne})
$$(N_{p,q}^{\beta}(u))^{q}\sim \sum_{k=0}^{\infty}\Big(\int_{\X}\int_{B(g,r_{k})}\frac{|u(g)-u(g')|^{p}}{r_{k}^{2\beta p+\Q}}dg'dg\Big)^{q/p}.$$

Without loss of generality, we  assume that $r_{k}=2^{-k}$.  We
first prove that $B_{p,q,\mathcal{H}}^{\alpha{
},2\beta}(\X)\subseteq B_{p,q}^{\alpha\beta}(\X)$. Let us write
$\Phi(s)=(1+s)^{-(\Q+2\alpha)}$. By Proposition \ref{pro-frac}, we
know that $$K_{\alpha,t}(g'^{-1}g)\geq
\frac{\Phi(1)}{t^{\Q/2\alpha}}$$ on $B(g,t^{1/2\alpha})$. Therefore,
\begin{eqnarray*}
   &&\int_{0}^{1}\frac{1}{t^{\beta q}}\Big(\int_{\X}\int_{\X}|u(g)-u(g')|^{p}K_{\alpha,t}(g'^{-1}g)dg'dg\Big)^{q/p}\frac{dt}{t}  \\
   &&\gtrsim \int_{0}^{1}\frac{1}{t^{\beta q }}\Big(\int_{\X}\int_{B(g,t^{1/2\alpha})}|u(g)-u(g')|^{p}K_{\alpha,t}(g'^{-1}g)dg'dg\Big)^{q/p}\frac{dt}{t}  \\
   && \gtrsim \int_{0}^{1}\Big(\int_{\X}\int_{B(g,t^{1/2\alpha})}\frac{|u(g)-u(g')|^{p}}{t^{\beta p}|B(g,t^{1/2\alpha})|}dg'dg\Big)^{q/p}\frac{dt}{t}\\
   &&\sim \sum_{k=1}^{\infty}\int_{r_{k+1}^{2\alpha}}^{r_{k}^{2\alpha}}\Big(\int_{\X}\int_{B(g,t^{1/2\alpha})}\frac{|u(g)-u(g')|^{p}}{t^{\beta p+\Q/2\alpha}}dg'dg\Big)^{q/p}\frac{dt}{t}\\
   &&\gtrsim \sum_{k=1}^{\infty}\Big(\int_{\X}\int_{B(g,r_{k})}\frac{|u(g)-u(g')|^{p}}{r_{k}^{2\alpha\beta
   p+\Q}}dg'dg\Big)^{q/p},
\end{eqnarray*}
which deduces  $B_{p,q,\mathcal{H}}^{\alpha{ },2\beta}(\X)\subseteq
B_{p,q}^{\alpha\beta}(\X)$.

Conversely, we write
$$\int_{0}^{1}\frac{1}{t^{\beta q}}\Big(\int_{\X}\int_{\X}|u(g)-u(g')|^{p}K_{\alpha,t}(g'^{-1}g)dg'dg\Big)^{q/p}\frac{dt}{t}\lesssim I_{1}+I_{2},$$
where
$$\left\{\begin{aligned}
I_{1}&:=\sum_{k=0}^{\infty}\int_{r_{k+1}^{2\alpha}}^{r_{k}^{2\alpha}} \frac{1}{t^{\beta q}}\Big(\int_{\X}\int_{d(g,g')>r_{k}}|u(g)-u(g')|^{p}K_{\alpha,t}(g'^{-1}g)dg'dg\Big)^{q/p}\frac{dt}{t};\\
I_{2}&:=\sum_{k=0}^{\infty}\int_{r_{k+1}^{2\alpha}}^{r_{k}^{2\alpha}}\frac{1}{t^{\beta q}}\Big(\int_{\X}\int_{d(g,g')\leq r_{k}}|u(g)-u(g')|^{p}K_{\alpha,t}(g'^{-1}g)dg'dg\Big)^{q/p}\frac{dt}{t}.
\end{aligned}\right.$$

By Proposition \ref{pro-frac}, we obtain
$$\int_{d(g,g')>r_{k}}K_{\alpha,t}(g'^{-1}g)dg'\leq \frac{t}{r_{k}^{2\alpha}}.$$

Note that $|u(g)-u(g')|^{p}\lesssim|u(g)|^{p}+|u(g')|^{p}.$ We deduce from Fubini's theorem and the
preceding inequality that
\begin{align*}
&\int_{\X}\int_{d(g,g')>r_{k}}|u(g)-u(g')|^{p}K_{\alpha,t}(g'^{-1}g)dg'dg\\
&\lesssim\int_{\X}\int_{d(g,g')>r_{k}}|u(g)|^{p}K_{\alpha,t}(g'^{-1}g)dgdg' \\
&\lesssim \|u\|_{L^{p}(\X)}^{p}\frac{t}{r_{k}^{2\alpha}}.
\end{align*}
This implies that
\begin{align*}
I_{1}&\lesssim \sum_{k=0}^{\infty}\int_{r_{k+1}^{2\alpha}}^{r_{k}^{2\alpha}} \frac{1}{t^{\beta q}}\Big(\|u\|_{L^{p}(\X)}^{p}\frac{t}{r_{k}^{2\alpha}}\Big)^{q/p}\frac{dt}{t}\\
&\lesssim \|u\|_{L^{p}(\X)}^{q}\sum_{k=0}^{\infty}r_{k}^{2\alpha q/p-2\alpha\beta q}\\
&\lesssim \|u\|_{L^{p}(\X)}^{q},
\end{align*}
where we have used the fact that $\beta\in(0,1/p)$ and $\{r_{k}\}$ is a decreasing geometric sequence.

Below we deal with $I_{2}$. Observe that
\begin{eqnarray*}
   &&\int_{\X}\int_{d(g,g')\leq r_{k}}|u(g)-u(g')|^{p}K_{\alpha,t}(g'^{-1}g)dg'dg  \\
   &&= \sum_{m=k}^{\infty} \int_{\X}\int_{r_{m+1}\leq d(g,g')\leq r_{m}}|u(g)-u(g')|^{p}K_{\alpha,t}(g'^{-1}g)dg'dg.
\end{eqnarray*}
Using Proposition \ref{pro-frac} and the fact that $t\in[r_{k+1}^{2\alpha},r_{k}^{2\alpha}]$, we have
\begin{eqnarray*}
   && \sum_{m=k}^{\infty} \int_{\X}\int_{r_{m+1}\leq d(g,g')\leq r_{m}}|u(g)-u(g')|^{p}K_{\alpha,t}(g'^{-1}g)dg'dg \\
   &&\lesssim \sum_{m=k}^{\infty}\int_{\X}\int_{r_{m+1}\leq d(g,g')\leq r_{m}}|u(g)-u(g')|^{p}\frac{1}{(r_{k}+r_{m})^{\Q}}\Phi(\frac{r_{m}}{r_{k}})dg'dg  \\
&& \lesssim  \sum_{m=k}^{\infty}\Big(\frac{r_{m}}{r_{k}+r_{m}}\Big)^{\Q}\Phi(\frac{r_{m}}{r_{k}})\int_{\X}\int_{ d(g,g')\leq r_{m}}\frac{|u(g)-u(g')|^{p}}{r_{m}^{\Q}}dg'dg \\
&& \lesssim \sum_{m=k}^{\infty}\Phi(\frac{r_{m}}{r_{k}})\int_{\X}\int_{ d(g,g')\leq r_{m}}\frac{|u(g)-u(g')|^{p}}{r_{m}^{\Q}}dg'dg.
\end{eqnarray*}
Therefore,
\begin{eqnarray*}
  I_{2} &\lesssim& \sum_{k=0}^{\infty}r_{k}^{-2\alpha\beta q}\Big(\sum_{m=k}^{\infty}\Phi(\frac{r_{m}}{r_{k}})\int_{\X}\int_{ d(g,g')\leq r_{m}}\frac{|u(g)-u(g')|^{p}}{r_{m}^{\Q}}dg'dg\Big)^{q/p}  \\
   &\lesssim&  \sum_{k=0}^{\infty}r_{k}^{-2\alpha\beta q}\Big(\sum_{m=k}^{\infty}\int_{\X}\int_{ d(g,g')\leq r_{m}}\frac{|u(g)-u(g')|^{p}}{r_{m}^{\Q}}dg'dg\Big)^{q/p}.
\end{eqnarray*}
Recall the classical discrete Hardy inequality: Let
$r>0$, $z>1$, $x_{k}>0$. Then
$$\sum_{k=0}^{\infty}z^{k}\Big(\sum_{m=k}^{\infty}x_{m}\Big)^{r}\lesssim \sum_{k=0}^{\infty}z^{k}x_{k}^{r}.$$
Denote by  $z:=2^{2\alpha\beta q}$ and
$$x_{m}:=\int_{\X}\int_{ d(g,g')\leq r_{m}}\frac{|u(g)-u(g')|^{p}}{r_{m}^{\Q}}dg'dg.$$ By the above
 discrete Hardy inequality we obtain
$$I_{2}\lesssim \sum_{k=0}^{\infty}\frac{1}{r_{k}^{2\alpha\beta q}}\Big(\int_{\X}\int_{ d(g,g')\leq r_{k}}\frac{|u(g)-u(g')|^{p}}{r_{k}^{\Q}}dg'dg\Big)^{q/p}.$$

Combining the estimates obtained above, we obtain
$$\int_{0}^{1}\frac{1}{t^{\beta q}}\Big(\int_{\X}\int_{\X}|u(g)-u(g')|^{p}K_{\alpha,t}(g'^{-1}g)dg'dg\Big)^{q/p}\frac{dt}{t}\lesssim \|u\|_{L^{p}(\X)}^{q}+(N_{p,q}^{\beta \alpha}(u))^{q}.$$

Therefore,  we obtain the desired result. This completes the  proof.

\end{proof}
In the following discussion,   we consider the case of $q=\infty$
and we have the following results.

\begin{proposition}\label{com-besov-2}
Let $p\in[1,\infty)$ and $\beta\in(0,1/p)$.

\item{\rm (i)} For $\alpha\in(0,1),$
$$B_{p,\infty,\mathcal{H}}^{\alpha{ },2\beta}(\X)=B_{p,\infty}^{\alpha\beta}(\X).$$

\item{\rm (ii)} For $\sigma\in(0,1),$
$$B_{p,\infty,\mathcal{P}}^{\sigma{ },2\beta}(\X)=B_{p,\infty}^{\sigma\beta}(\X).$$
\end{proposition}
\begin{proof}
We only prove (i). The proof of (ii) is similar. Let us write
$\Phi(s)=(1+s)^{-\Q-2\alpha}$. By Proposition \ref{pro-frac}, we
know that $K_{\alpha,t}(g'^{-1}g)\geq \Phi(1)t^{-\Q/2\alpha}$ on
$B(g,t^{1/2\alpha})$. Then we have
\begin{eqnarray*}
   &&\frac{1}{t^{\beta p}}\int_{\X}\int_{\X}|u(g)-u(g')|^{p}K_{\alpha,t}(g'^{-1}g)dgdg'  \\
   &&\geq \frac{1}{t^{\beta p}}\int_{\X}\int_{B(g,t^{1/2\alpha})}|u(g)-u(g')|^{p}K_{\alpha,t}(g'^{-1}g)dg'dg  \\
   && \geq \Phi(1)\int_{\X}\int_{B(g,t^{1/2\alpha})}\frac{|u(g)-u(g')|^{p}}{t^{\beta p+\Q/2\alpha}}dg'dg,
\end{eqnarray*}
and taking the supremum over $t>0$ yields
$B_{p,\infty,\mathcal{H}}^{\alpha{ },2\beta}(\X)\subseteq
B_{p,\infty}^{\alpha\beta}(\X)$. For the upper bound, fix $r>0$ and
set
$$\left\{\begin{aligned}
A(t,r)&:=\int_{\X}\int_{\X\setminus B(g',r)}|u(g)-u(g')|^{p}K_{\alpha,t}(g'^{-1}g)dgdg';\\
B(t,r)&:=\int_{\X}\int_{B(g',r)}|u(g)-u(g')|^{p}K_{\alpha,t}(g'^{-1}g)dgdg'.
\end{aligned}\right.$$
Using Proposition \ref{pro-frac}, we know that
\begin{eqnarray*}
  \int_{\X\setminus B(g',r)}K_{\alpha,t}(g'^{-1}g)dg &\lesssim&  \int_{\X\setminus B(g',r)}\frac{1}{t^{-Q/2\alpha}}\Phi(\frac{d(g,g')}{t^{1/2\alpha}})dg  \\
 &\lesssim& \sum_{k=0}^{\infty}\int_{B(g',r_{k+1})\setminus B(g',r_{k})}\frac{1}{t^{-Q/2\alpha}}\Phi(\frac{r_{k}}{t^{1/2\alpha}})dg  \\
   &\lesssim& \sum_{k=0}^{\infty}\frac{r_{k}^{\Q}}{t^{-Q/2\alpha}}\Phi(\frac{r_{k}}{t^{1/2\alpha}})   \\
  &\lesssim&  \int_{\frac{1}{2}r/t^{1/2\alpha}}\sigma^{\Q}\Phi(\sigma)\frac{d\sigma}{\sigma},
\end{eqnarray*}
where $r_{k}=2^{k}r$. Accordingly, using Fubini's theorem, we obtain
\begin{eqnarray} \notag
 A(t,r) &\lesssim& \int_{\X}\int_{\X\setminus B(g',r)}|u(g')|^{p}K_{\alpha,t}(g'^{-1}g)dgdg'  \\ \notag
   &\lesssim& \|u\|_{L^{p}(\X)}^{p}\int_{\frac{1}{2}rt^{-1/2\alpha}}\sigma^{\Q}\Phi(\sigma)\frac{d\sigma}{\sigma}  \\ \label{KS-t-2}
   &\lesssim&  t^{p\beta}r^{-2\alpha \beta p}\|u\|_{L^{p}(\X)}^{p}\int_{\frac{1}{2}rt^{-1/2\alpha}}\sigma^{\Q+2\beta \alpha p}\Phi(\sigma)\frac{d\sigma}{\sigma}.
\end{eqnarray}

On the other hand, for $B(t,r)$, writing $r_{k}=2^{-k}r$, using Proposition \ref{pro-frac} again, we have
\begin{eqnarray}\notag
   B(t,r) &\lesssim& t^{-\Q/2\alpha}\sum_{k=0}^{\infty}\Phi(\frac{r_{k}}{t^{1/2\alpha}})r_{k}^{2\beta \alpha p+\Q}\int_{\X}\int_{B(g,r_{k})}\frac{|u(g)-u(g')|^{p}}{r_{k}^{2\beta \alpha p+\Q}}dg'dg  \\ \label{KS-t-3}
   &\lesssim&  t^{\beta p}\|u\|_{B_{p,\infty}^{\alpha\beta}(\X)}^{p}\int_{0}^{\infty}\sigma^{\Q+2\beta \alpha p}\Phi(\sigma)\frac{d\sigma}{\sigma}.
\end{eqnarray}

The integrals in both (\ref{KS-t-2}) and (\ref{KS-t-3}) are bounded
due to  $\beta<1/p$ by assumption. Therefore,
$$\frac{1}{t^{\beta p}}\int_{\X}\int_{\X}|u(g)-u(g')|^{p}K_{\alpha,t}(g'^{-1}g)dgdg'\leq C_{p,\beta}\Big(\frac{1}{r^{2\beta \alpha p}}\|u\|_{L^{p}(\X)}^{p}+\|u\|_{B_{p,\infty}^{\alpha\beta}(\X)}^{p}\Big).$$
Taking the supremum over $t>0$, we obtain
$B_{p,\infty}^{\alpha\beta}(\X)\subseteq
B_{p,\infty,\mathcal{H}}^{\alpha{ },2\beta}(\X)$ and letting
$r\rightarrow\infty$ gives the equivalence of the seminorms.
\end{proof}

\begin{remark}\label{ccc}
According to  Theorems  \ref{com-2} \&\ \ref{com-besov-2}, we have
the following conclusion: If
$(p,q,\beta)\in[1,\infty)\times[1,\infty]\times(0,\min\{\alpha/p,\sigma/p\}),$
then for any $(\alpha,\sigma)\in(0,1)\times(0,1),$
$$B_{p,q, \mathcal{P}}^{ {\sigma},2\beta/\sigma}(\X)=B_{p,q,\mathcal{H}}^{ {\alpha},2\beta/\alpha}(\X)
=B_{p,q}^{\beta}(\X).$$

\end{remark}

\subsection{Besov capacities}\label{sec-6.1}
In this section, we first define a Besov capacity, denoted by
$\textrm{Cap}_{B_{p,p,\mathcal{H}}^{\alpha,\beta}}(\cdot)$, which is
naturally derived  from the Besov space
$B_{p,p,\mathcal{H}}^{\alpha,\beta}(\X) $. Similarly, we can
introduce  another  Besov capacity
$\textrm{Cap}_{B_{p,p,\mathcal{P}}^{\sigma,\beta}}(\cdot)$ related
to  $B_{p,p,\mathcal{P}}^{\sigma,\beta}(\X) $. Since their results
and the corresponding proofs are similar, thus we only need to focus
on the case of
$\textrm{Cap}_{B_{p,p,\mathcal{H}}^{\alpha,\beta}}(\cdot)$.
\begin{definition}\label{defn2.7}
Let $p\geq1$, $\beta\geq0$ and $ \alpha \in(0,1) $. The Besov
capacity of an arbitrary set $E\subset \X$ is defined as
$$\textrm{Cap}_{B_{p,p,\mathcal{H}}^{\alpha,\beta}}(E):=\inf\big\{\|u\|^{p}_{B_{p,p,\mathcal{H}}^{\alpha,\beta}(\X)}
:u\in\mathcal{A}(E)\big\},$$ where
$$\mathcal{A}(E):=\{u\in
B_{p,p,\mathcal{H}}^{\alpha,\beta}(\X):
E\subset\{g\in\X:u(g)\geq1\}^{o}\}.$$
\end{definition}

\begin{remark}\label{lem5}
\item{(i)} If $u\in\mathcal{A}(E),$ then $\min\{1,u\}\in\mathcal{A}(E)$ and
$$\|\min\{1,u\}\|_{B_{p,p,\mathcal{H}}^{\alpha,\beta} (\X)}\leq\|u\|_{B_{p,p,\mathcal{H}}^{\alpha,\beta} (\X)}.$$
Thus we may restrict ourselves to those admissible functions $u$ for
which $0\leq u\leq 1.$

\item{(ii)} Note that for any function $u$ on $\X,$ we have
$$\big||u(g)|-|u(g')|\big|\leq|u(g)-u(g')|~~\ \ \ \mbox{for all}~~g,g'\in\X,$$
which implies
$$\big\||u|\big\|_{B_{p,p,\mathcal{H}}^{\alpha,\beta}(\X)}\leq\|u\|_{B_{p,p,\mathcal{H}}^{\alpha,\beta}  (\X)}.$$
Therefore, for any set $E\subset\X,$ we can also  alternatively write
\begin{equation}\label{alter}
\textrm{Cap}_{B_{p,p,\mathcal{H}}^{\alpha,\beta}}(E):=\inf\big\{\|u\|^{p}_{B_{p,p,\mathcal{H}}^{\alpha,\beta}(\X)}:0\leq
u\in B_{p,p,\mathcal{H}}^{\alpha,\beta}(\X)~\mbox{and}~
E\subset\{g\in\X:u(g)\geq1\}^{o}\big\}.
\end{equation}
\end{remark}

Similarly to the fractional Sobolev capacity, we establish the
following measure-theoretic properties of the Besov capacity
$\textrm{Cap}_{B_{p,p,\mathcal{H}}^{\alpha,\beta}}(\cdot)$.
\begin{lemma}\label{CP}
Let $p\geq1$, $\beta\geq0$. Then the following assertions hold:

  \item [(i)] $\textrm{Cap}_{B_{p,p,\mathcal{H}}^{\alpha,\beta}}(\emptyset)=0.$
  \item [(ii)] For any sets $E_{1},E_{2}\subset\X$ with $E_{1}\subset E_{2},$
  $$\textrm{Cap}_{B_{p,p,\mathcal{H}}^{\alpha,\beta}}(E_{1})\leq\textrm{Cap}_{B_{p,p,\mathcal{H}}^{\alpha,\beta}}(E_{2}).$$
  \item [(iii)] For any sequence of sets $\{E_{j}\}_{j\in\mathbb{N}}$ in $\X,$
  $$\textrm{Cap}_{B_{p,p,\mathcal{H}}^{\alpha,\beta}}(\cup_{j=1}^{\infty}E_{j})\le
  \sum_{j=1}^{\infty}\textrm{Cap}_{B_{p,p,\mathcal{H}}^{\alpha,\beta}}(E_{j}).$$
  \item [(iv)] If $\{\mathcal{K}_{j}\}_{j\in\mathbb{N}}$ is a decreasing sequence of compact sets in $\X,$ then
  $$\textrm{Cap}_{B_{p,p,\mathcal{H}}^{\alpha,\beta}}(\cap_{j=1}^{\infty}\mathcal{K}_{j})
  =\lim_{j\rightarrow\infty}\textrm{Cap}_{B_{p,p,\mathcal{H}}^{\alpha,\beta}}(\mathcal{K}_{j}).$$
  \item [(v)] If $\{E_{j}\}_{j\in\mathbb{N}}$ is an increasing sequence of sets in $\X,$ then
  $$\textrm{Cap}_{B_{p,p,\mathcal{H}}^{\alpha,\beta}}(\cup_{j=1}^{\infty}E_{j})
  =\lim_{j\rightarrow\infty}\textrm{Cap}_{B_{p,p,\mathcal{H}}^{\alpha,\beta}}(E_{j}).$$

\end{lemma}
\begin{proof}

We only give the proofs of (iv) and (v), and the items (i), (ii)\ \&\ (iii) can be proved
by the methods similar to those of Proposition \ref{pro-6}. For
(iv), note that (ii) implies
$$\textrm{Cap}_{B_{p,p,\mathcal{H}}^{\alpha,\beta}}(\cap_{j=1}^{\infty}\mathcal{K}_{j})\leq\lim_{j\rightarrow\infty}
\textrm{Cap}_{B_{p,p,\mathcal{H}}^{\alpha,\beta}}(\mathcal{K}_{j}).$$
Thus, it suffices to show the converse of this inequality. For any
$\epsilon\in(0,1),$ we apply (\ref{alter}) to find a nonnegative
function $  u\in B_{p,p,\mathcal{H}}^{\alpha,\beta}(\X)$ such that
$$\cap_{j=1}^{\infty}\mathcal{K}_{j}\subset\{g\in\X:u(g)\geq1\}^{o}~\mbox{and}~
\|u\|^{p}_{B_{p,p,\mathcal{H}}^{\alpha,\beta}(\X)}\leq\textrm{Cap}_{B_{p,p,\mathcal{H}}^{\alpha,\beta}}
(\cap_{j=1}^{\infty}\mathcal{K}_{j})+\epsilon.$$ Since
$\{\mathcal{K}_{j}\}_{j=1}^{\infty}$ is decreasing, there exists
$j_{0}\in\mathbb{N}$ such that
$\mathcal{K}_{j_{0}}\subset\{g\in\X:u(g)\geq1\}^{o},$ whence leading
to
$$\lim_{j\rightarrow\infty}\textrm{Cap}_{B_{p,p,\mathcal{H}}^{\alpha,\beta}}(\mathcal{K}_{j})
\leq\textrm{Cap}_{B_{p,p,\mathcal{H}}^{\alpha,\beta}}(\mathcal{K}_{j_{0}})\leq\|u\|^{p}_{B_{p,p,\mathcal{H}}^{\alpha,\beta}(\X)}
\leq\textrm{Cap}_{B_{p,p,\mathcal{H}}^{\alpha,\beta}}(\cap_{j=1}^{\infty}\mathcal{K}_{j})+\epsilon.$$
Letting $\epsilon\rightarrow0$ yields the desired inequality. Thus,
(iv) holds.

For (v),   due to (ii) and the fact that
$\{E_{j}\}_{j\in\mathbb{N}}$ is increasing, it suffices  to validate
\begin{align}\label{E3}
\textrm{Cap}_{B_{p,p,\mathcal{H}}^{\alpha,\beta}}(\cup_{j=1}^{\infty}E_{j})\leq
\lim_{j\rightarrow\infty}\textrm{Cap}_{B_{p,p,\mathcal{H}}^{\alpha,\beta}}(E_{j}).
\end{align}
Fix $\epsilon\in(0,1).$ For any $j\in\mathbb{N},$ by (\ref{alter}),
there exists $0\leq u_{j}\in B_{p,p,\mathcal{H}}^{\alpha,\beta}(\X)$
such that
\begin{align}\label{E4}
E_{j}\subset\{g\in\X:u_{j}(g)\geq1\}^{o}
~and~\|u_{j}\|^{p}_{B_{p,p,\mathcal{H}}^{\alpha,\beta}(\X)}\leq\textrm{Cap}_{B_{p,p,\mathcal{H}}^{\alpha,\beta}}(E_{j})+2^{-j}\epsilon.
\end{align}
Let $E_{0}=\emptyset,$ which has zero
$B_{p,p,\mathcal{H}}^{\alpha,\beta}$-capacity. Let $h_{0}:=0.$ For
any $m\in\mathbb{N},$ define
$$h_{m}:=\max\{h_{m-1},u_{m}\}=\max_{1\leq j\leq m}u_{j}.$$
For any $m\in\mathbb{N},$ Lemma \ref{max} implies  that both $h_{m}$
and $\min\{h_{m-1},u_{m}\}$ are in
$B_{p,p,\mathcal{H}}^{\alpha,\beta}(\X)$ with
$$\|\max\{h_{m-1},u_{m}\}\|^{p}_{B_{p,p,\mathcal{H}}^{\alpha,\beta}(\X)}+
\|\min\{h_{m-1},u_{m}\}\|^{p}_{B_{p,p,\mathcal{H}}^{\alpha,\beta}(\X)}\leq
\|h_{m-1}\|^{p}_{B_{p,p,\mathcal{H}}^{\alpha,\beta}(\X)}+\|u_{m}\|^{p}_{B_{p,p,\mathcal{H}}^{\alpha,\beta}(\X)}.$$
Observe that
$$E_{m-1}\subset\{g\in\X:\min\{h_{m-1}(x),u_{m}(x)\}\geq1\}^{o}.$$
From the above arguments, (\ref{alter}) and (\ref{E4}), we deduce
\begin{align*}
&\|h_{m}\|^{p}_{B_{p,p,\mathcal{H}}^{\alpha,\beta}(\X)}+\textrm{Cap}_{B_{p,p,\mathcal{H}}^{\alpha,\beta}}(E_{m-1})\\
&\leq\|\max\{h_{m-1},u_{m}\}\|^{p}_{B_{p,p,\mathcal{H}}^{\alpha,\beta}(\X)}
+\|\min\{h_{m-1},u_{m}\}\|^{p}_{B_{p,p,\mathcal{H}}^{\alpha,\beta}(\X)}\\
&\leq\|h_{m-1}\|^{p}_{B_{p,p,\mathcal{H}}^{\alpha,\beta}(\X)}+\textrm{Cap}_{B_{p,p,\mathcal{H}}^{\alpha,\beta}}(E_{m})+2^{-m}\epsilon
\end{align*}
and hence
$$\|h_{m}\|^{p}_{B_{p,p,\mathcal{H}}^{\alpha,\beta}(\X)}-\|h_{m-1}\|^{p}_{B_{p,p,\mathcal{H}}^{\alpha,\beta}(\X)}
\leq\textrm{Cap}_{B_{p,p,\mathcal{H}}^{\alpha,\beta}}(E_{m})-\textrm{Cap}_{B_{p,p,\mathcal{H}}^{\alpha,\beta}}(E_{m-1})+2^{-m}\epsilon.$$
This last inequality further implies
$$\|h_{m}\|^{p}_{B_{p,p,\mathcal{H}}^{\alpha,\beta}(\X)}\leq\textrm{Cap}_{B_{p,p,\mathcal{H}}^{\alpha,\beta}}(E_{m})+\epsilon.$$
Let
$$u:=\lim_{m\rightarrow\infty}h_{m}=\sup_{m\in\mathbb{N}}u_{m}.$$
Clearly,
$$\cup_{m\in\mathbb{N}}E_{m}\subset\{g\in\X:u(g)\geq1\}^{o}.$$
Meanwhile, we use   Fatou's lemma to obtain
\begin{equation*}\label{LC1}
\|u\|^{p}_{L^{p}(\X)}
=\int_{\X}|u(g)|^{p}dg=\int_{\X}|\lim_{m\rightarrow\infty}h_{m}(g)|^{p}dg
\leq\liminf_{m\rightarrow\infty}\|h_{m}\|^{p}_{L^{p}(\X)}
\end{equation*}
and
\begin{align*}
(N{_{p,p,\mathcal{H}}^{\alpha,\beta}}(u))^{p}&=\int_{0}^{\infty}\dfrac{1}{t^{{\beta p}/{2}+1}}\int_{\X}H_{\alpha,t}(|u-u(g')|^{p})(g')dg'dt\\
&=\int_{0}^{\infty}\dfrac{1}{t^{{\beta p}/{2}+1}}\int_{\X}\int_{\X}K_{\alpha,t}(g'^{-1}g)\lim_{m\rightarrow\infty}|h_{m}(g)-h_{m}(g')|^{p}dgdg'dt\\
&\leq \liminf_{m\rightarrow\infty}\int_{0}^{\infty}\dfrac{1}{t^{{\beta p}/{2}+1}}\int_{\X}\int_{\X}K_{\alpha,t}(g'^{-1}g)|h_{m}(g)-h_{m}(g')|^{p}dgdg'dt\\
&=\liminf_{m\rightarrow\infty}(N{_{p,p,\mathcal{H}}^{\alpha,\beta}}(h_{m}))^{p}.
\end{align*}
Thus
\begin{align*}
\textrm{Cap}_{B_{p,p,\mathcal{H}}^{\alpha,\beta}}(\cup_{m\in\mathbb{N}}E_{m})
&\leq\|u\|^{p}_{B_{p,p,\mathcal{H}}^{\alpha,\beta}(\X)}\\
&\leq\liminf_{m\rightarrow\infty}\|h_{m}\|^{p}_{L^{p}(\X)}
+\liminf_{m\rightarrow\infty}(N{_{p,p,\mathcal{H}}^{\alpha,\beta}}(h_{m}))^{p}\\
&\leq\liminf_{m\rightarrow\infty}\|h_{m}\|^{p}_{B_{p,p,\mathcal{H}}^{\alpha,\beta}(\X)}\\
&\leq\lim_{m\rightarrow\infty}\textrm{Cap}_{B_{p,p,\mathcal{H}}^{\alpha,\beta}}(E_{m})+\epsilon.
\end{align*}
Letting $\epsilon\rightarrow0,$ we then obtain (\ref{E3}) and
hence this completes  the proof of (v).

\end{proof}

From Lemma \ref{CP}, we know that
$\textrm{Cap}_{B_{p,p,\mathcal{H}}^{\alpha,\beta}}(\cdot)$ is not
only an outer measure (obeying (i), (ii) and (iii)), but also a
Choquet capacity (satisfying (i), (ii), (iv) and (v)). An important
feature is that the capacity of a Borel set $E$ can be estimated
``from the inside" by a compact set, and ``from the outside" by an
open set:
\begin{align}\label{capco}
\textrm{Cap}_{B_{p,p,\mathcal{H}}^{\alpha,\beta}}(E)&=\sup\{\textrm{Cap}_{B_{p,p,\mathcal{H}}^{\alpha,\beta}}(\mathcal{K}):\mathcal{K}\subset E, \mathcal{K}~\mbox{compact}\}\nonumber\\
&=\inf\{\textrm{Cap}_{B_{p,p,\mathcal{H}}^{\alpha,\beta}}(O):E\subset
O, O~\mbox{open}\}.
\end{align}
 Moreover, we introduce another Besov capacity
$\textrm{Cap}_{B_{p,p,\mathcal{H}}^{\alpha,\beta}}^{*}(\mathcal{K})$
as follows.
\begin{definition}\label{defn2.9}
Let $p\geq1$ and $\beta\geq0.$ For any compact set
$\mathcal{K}\subset\X,$ define
\begin{equation}\label{Cap1}
\textrm{Cap}_{B_{p,p,\mathcal{H}}^{\alpha,\beta}}^{*}(\mathcal{K}):=\inf\{\|u\|^{p}_{B_{p,p,\mathcal{H}}^{\alpha,\beta}(\X)}:u\in
C_{c}^{\infty}(\X)~\textrm{and}~u\geq1~\textrm{on}~\mathcal{K}\}.
\end{equation}
\end{definition}

This definition is extended to any open set $O\subset\X$ via
\begin{equation}\label{Cap2}
\textrm{Cap}_{B_{p,p,\mathcal{H}}^{\alpha,\beta}}^{*}(O):=\sup\{\textrm{Cap}_{B_{p,p,\mathcal{H}}^{\alpha,\beta}}^{*}(\mathcal{K}):
~\textrm{compact}~\mathcal{K}\subset O\},
\end{equation}
and then for an arbitrary set $E\subset\X$ through
putting
\begin{equation}\label{Cap3}
\textrm{Cap}_{B_{p,p,\mathcal{H}}^{\alpha,\beta}}^{*}(E):=\inf\{\textrm{Cap}_{B_{p,p,\mathcal{H}}^{\alpha,\beta}}^{*}(O):~\textrm{open}~O\supset
E\}.
\end{equation}

\begin{lemma}\label{EquiC}

For any $2\alpha>p\geq1,\beta\in(0,1/\alpha)$ and $E\subset\X,$
$$\textrm{Cap}_{B_{p,p,\mathcal{H}}^{\alpha,\beta}}^{*}(E)=\textrm{Cap}_{B_{p,p,\mathcal{H}}^{\alpha,\beta}}(E).$$

\end{lemma}
\begin{proof}
  We can prove this lemma in the following three cases.

\emph{Case 1}: $E=\mathcal{K}$ is compact. We first show
\begin{equation}\label{cap1}
\textrm{Cap}_{B_{p,p,\mathcal{H}}^{\alpha,\beta}}(\mathcal{K})\leq
\textrm{Cap}_{B_{p,p,\mathcal{H}}^{\alpha,\beta}}^{*}(\mathcal{K}).
\end{equation}
For any $\epsilon>0,$ it follows from (\ref{Cap1}) that there exists $u\in C_{c}^{\infty}(\X)$ such that $u\geq1$ on $\mathcal{K}$ and
$$\|u\|^{p}_{B_{p,p,\mathcal{H}}^{{\alpha},\beta}(\X)}\leq\textrm{Cap}_{B_{p,p,\mathcal{H}}^{\alpha,\beta}}^{*}(\mathcal{K})+\epsilon.$$
Let $u_{\lambda}:=(1+\lambda)u\in C_{c}^{\infty}(\X),$ where
$\lambda>0.$ By Proposition \ref{desity}, we have $u_{\lambda}\in
B_{p,p,\mathcal{H}}^{{\alpha},\beta}(\X).$ Moreover,
$$\mathcal{K}\subset\{g\in\X:u_{\lambda}(g)>1\}=\{g\in\X:u_{\lambda}(g)\geq1\}^{o}.$$
Therefore, we get
$$\textrm{Cap}_{B_{p,p,\mathcal{H}}^{\alpha,\beta}}(\mathcal{K})\leq \|u_{\lambda}\|^{p}_{B_{p,p,\mathcal{H}}^{{\alpha},\beta}(\X)}
\leq(1+\lambda)^{p}(\textrm{Cap}_{B_{p,p,\mathcal{H}}^{\alpha,\beta}}^{*}(\mathcal{K})+\epsilon).$$
Letting $\epsilon,\lambda\rightarrow0$ yields (\ref{cap1}).

Next we prove the converse inequality of (\ref{cap1}). For any
$\epsilon>0,$ by (\ref{alter}), there exists $0\leq u\in
B_{p,p,\mathcal{H}}^{{\alpha},\beta}(\X)$ such that
$\mathcal{K}\subset\{g\in\X:u(g)\geq1\}^{o}$ and
$$\|u\|^{p}_{B_{p,p,\mathcal{H}}^{{\alpha},\beta}(\X)}\leq\textrm{Cap}_{B_{p,p,\mathcal{H}}^{\alpha,\beta}}(\mathcal{K})+\epsilon.$$
Using Proposition \ref{desity}, $u$ can be approximated by
$\{u_{j}\}_{j\in\mathbb{N}}$ in the space
$B_{p,p,\mathcal{H}}^{{\alpha},\beta}(\X),$ with an additional
property that $u_{j}\geq1$ on $\mathcal{K}$ for any
$j\in\mathbb{N}.$ Moreover, we can even assume that
$\{u_{j}\}_{j\in\mathbb{N}}\geq0.$ Thus,
$$\textrm{Cap}_{B_{p,p,\mathcal{H}}^{\alpha,\beta}}^{*}(\mathcal{K})\leq \lim_{j\rightarrow\infty}\|u_{j}\|^{p}_{B_{p,p,\mathcal{H}}^{{\alpha},\beta}(\X)}
=\|u\|^{p}_{B_{p,p,\mathcal{H}}^{{\alpha},\beta}(\X)}\leq\textrm{Cap}_{B_{p,p,\mathcal{H}}^{\alpha,\beta}}(\mathcal{K})+\epsilon,$$
which implies the converse inequality of (\ref{cap1}) by letting
$\epsilon\rightarrow0.$

\emph{Case 2}: $E=O$ is open. We deduce from (\ref{Cap2}), \emph{Case 1} and (\ref{capco}) that
\begin{align*}
\textrm{Cap}_{B_{p,p,\mathcal{H}}^{\alpha,\beta}}^{*}(O)&=
\sup\{\textrm{Cap}_{B_{p,p,\mathcal{H}}^{\alpha,\beta}}^{*}(\mathcal{K}):~\textrm{compact}~\mathcal{K}\subset O\}\\
&=\sup\{\textrm{Cap}_{B_{p,p,\mathcal{H}}^{\alpha,\beta}}(\mathcal{K}):~\textrm{compact}~\mathcal{K}\subset O\}\nonumber\\
&=\textrm{Cap}_{B_{p,p,\mathcal{H}}^{\alpha,\beta}}(O).
\end{align*}
\emph{Case 3}: $E$ is an arbitrary set in $\X.$ We first show
\begin{equation}\label{cap2}
\textrm{Cap}_{B_{p,p,\mathcal{H}}^{\alpha,\beta}}(E)\leq
\textrm{Cap}_{B_{p,p,\mathcal{H}}^{\alpha,\beta}}^{*}(E).
\end{equation}
Fix $\epsilon\in(0,1).$ From (\ref{Cap3}) and \emph{Case 2}, there exists an open set $O\supset E$ such that
$$\textrm{Cap}_{B_{p,p,\mathcal{H}}^{\alpha,\beta}}^{*}(E)>\textrm{Cap}_{B_{p,p,\mathcal{H}}^{\alpha,\beta}}^{*}(O)-\epsilon
=\textrm{Cap}_{B_{p,p,\mathcal{H}}^{\alpha,\beta}}(O)-\epsilon.$$
Further, it follows from (\ref{alter}) that there exists $0\leq u\in
B_{p,p,\mathcal{H}}^{{\alpha},\beta}(\X)$ such that
$O\subset\{g\in\X:u(g)\geq1\}^{o}$ and
$\textrm{Cap}_{B_{p,p,\mathcal{H}}^{\alpha,\beta}}(O)>\|u\|^{p}_{B_{p,p,\mathcal{H}}^{{\alpha},\beta}(\X)}-\epsilon,$
which in turn gives
$$\textrm{Cap}_{B_{p,p,\mathcal{H}}^{\alpha,\beta}}^{*}(E)>\|u\|^{p}_{B_{p,p,\mathcal{H}}^{{\alpha},\beta}(\X)}-2\epsilon.$$
Therefore,  we get
$$E\subset O\subset\{g\in\X:u(g)\geq1\}^{o}\Rightarrow \textrm{Cap}_{B_{p,p,\mathcal{H}}^{\alpha,\beta}}(E)\leq\|u\|^{p}_{B_{p,p,\mathcal{H}}^{{\alpha},\beta}(\X)}.$$
The arbitrariness of $\varepsilon$ yields (\ref{cap2}).

Next we prove the converse  of the inequality (\ref{cap2}). Fix
$\epsilon\in(0,1),$ by (\ref{alter}), there exists $0\leq u\in
B_{p,p,\mathcal{H}}^{{\alpha},\beta}(\X)$ such that
$E\subset\{g\in\X:u(g)\geq1\}^{o}$ and
$\|u\|^{p}_{B_{p,p,\mathcal{H}}^{{\alpha},\beta}(\X)}
\leq\textrm{Cap}_{B_{p,p,\mathcal{H}}^{\alpha,\beta}}(E)+\epsilon.$
Denote by $O:=\{g\in\X:u(g)\geq1\}^{o}.$ By \emph{Case 2}, we have
$$\|u\|^{p}_{B_{p,p,\mathcal{H}}^{{\alpha},\beta}(\X)}
\geq\textrm{Cap}_{B_{p,p,\mathcal{H}}^{\alpha,\beta}}(O)=\textrm{Cap}_{B_{p,p,\mathcal{H}}^{\alpha,\beta}}^{*}(O)
\geq \textrm{Cap}_{B_{p,p,\mathcal{H}}^{\alpha,\beta}}^{*}(E).$$
Hence the converse of (\ref{cap2}) can be deduced by letting $\epsilon\rightarrow0$.
\end{proof}

\subsection{Proof of Theorem \ref{thm3}}\label{sec-6.2}

\begin{theorem}\label{CI}
Let $p\geq1$, $\beta\geq0$ and $\alpha\in(0,1)$. Then
$$\int_{0}^{\infty}\textrm{Cap}_{B_{p,p,\mathcal{H}}^{\alpha,\beta}}(\{g\in\X:|u(g)|>t\})dt^{p}\lesssim
\|u\|^{p}_{B_{p,p,\mathcal{H}}^{{\alpha},\beta}(\X)}$$ holds for all
lower semi-continuous functions $u\in
B_{p,p,\mathcal{H}}^{{\alpha},\beta}(\X).$

\end{theorem}
\begin{proof}
Assume that $u\geq0.$ Choose an increasing smooth function
$\phi:\mathbb{R}\rightarrow\mathbb{R}$ with $\phi(t)=0$ as $t\leq
2^{-1};$ $\phi(t)=2$ as $t\geq1;$ $0\leq \phi'(t)<5$ for all
$t\in\mathbb{R};$ $F_{j}(u):=2^{j}\phi(2^{-j}u)$ for all $j\in
\mathbb{Z}.$ Clearly,
\begin{align}\label{CI1}
\{g\in\X:u(g)>2^{j}\}&\subset\{g\in\X:2^{-j}F_{j}(u(g))=2\}\nonumber\\
&\subset\{g\in\X:2^{-j}F_{j}(u(g))>1\}.
\end{align}
Moreover, based on the proof of \cite[Lemma 2.1]{W1999}, we have
$$\sum_{j\in\mathbb{Z}}\|F_{j}(u)\|^{p}_{L^{p}(\X)}\lesssim\|u\|^{p}_{L^{p}(\X)}$$
and
$$\sum_{j\in\mathbb{Z}}|F_{j}(u(g))-F_{j}(u(g'))|^{p}\lesssim|u(g)-u(g')|^{p},$$
which, combining with Fubini's theorem, implies
\begin{align*}
\sum_{j\in\mathbb{Z}}(N_{p,p,\mathcal{H}}^{\alpha,\beta}(F_{j}(u)))^{p}&=\sum_{j\in\mathbb{Z}}
\int_{0}^{\infty}\int_{\X}H_{\alpha,t}(|F_{j}(u)-F_{j}(u(g'))|^{p})(g')dg'\dfrac{dt}{t^{\beta p/2+1}}\\
&=\sum_{j\in\mathbb{Z}}\int_{0}^{\infty}\int_{\X}\int_{\X}K_{\alpha,t}(g'^{-1}g)|F_{j}(u(g))-F_{j}(u(g'))|^{p}dgdg'\dfrac{dt}{t^{\beta p/2+1}}\\
&\lesssim\int_{0}^{\infty}\int_{\X}\int_{\X}K_{\alpha,t}(g'^{-1}g)|u(g)-u(g')|^{p}dgdg'\dfrac{dt}{t^{\beta p/2+1}}\\
&\lesssim(N_{p,p,\mathcal{H}}^{\alpha,\beta}(u))^{p}.
\end{align*}
Therefore,
\begin{align}\label{CI2}
\sum_{j\in\mathbb{Z}}\|F_{j}(u)\|^{p}_{B_{p,p,\mathcal{H}}^{{\alpha},\beta}(\X)}
&=\sum_{j\in\mathbb{Z}}\|F_{j}(u)\|^{p}_{L^{p}(\X)}
+\sum_{j\in\mathbb{Z}}(N_{p,p,\mathcal{H}}^{\alpha,\beta}(u)(F_{j}(u)))^{p}\nonumber\\
&\lesssim\|u\|^{p}_{L^{p}(\X)}+(B_{p,p,\mathcal{H}}^{\alpha,\beta}(u))^{p}\nonumber\\
&=\|u\|^{p}_{B_{p,p,\mathcal{H}}^{{\alpha},\beta}(\X)}.
\end{align}
Furthermore, by (\ref{CI1}), (\ref{CI2}), Lemma \ref{CP} (ii) and
the fact that every $F_{j}(u)$ is lower semi-continuous, we deduce
\begin{align*}
&\int_{0}^{\infty}\textrm{Cap}_{B_{p,p,\mathcal{H}}^{\alpha,\beta}}(\{g\in\X:u(g)>t\})dt^{p}\\
&\lesssim\sum_{j\in\mathbb{Z}}2^{jp}\textrm{Cap}_{B_{p,p,\mathcal{H}}^{\alpha,\beta}}(\{g\in\X:u(g)>2^{j}\})\\
&\lesssim\sum_{j\in\mathbb{Z}}2^{jp}\textrm{Cap}_{B_{p,p,\mathcal{H}}^{\alpha,\beta}}(\{g\in\X:2^{-j}F_{j}(u(g))>1\})\\
&\lesssim\sum_{j\in\mathbb{Z}}2^{jp}\|2^{-j}F_{j}(u)\|^{p}_{B_{p,p,\mathcal{H}}^{{\alpha},\beta}(\X)}\lesssim
\|u\|^{p}_{B_{p,p,\mathcal{H}}^{{\alpha},\beta}(\X)}.
\end{align*}
This completes the proof of Theorem \ref{CI}.
\end{proof}

The capacitary inequalities in Theorem \ref{CI} enable us to
obtain the following trace inequalities for the Besov type space $B_{p,p,\mathcal{H}}^{{\alpha},\beta}(\X)$.
\begin{theorem}\label{trace}
Assume that $1< p\leq q<\infty$ with $2\alpha>p$ and
$\beta\in(0,1/\alpha).$ Let $\nu$ be a nonnegative measure  on $\X$.
 Then the following two assertions are equivalent:

  \item [(i)]For all sets $E\subset\X,$ $\nu(E)\lesssim(\textrm{Cap}_{B_{p,p,\mathcal{H}}^{\alpha,\beta}}(E))^{q/p}.$
  \item [(ii)] For all lower semi-continuous functions $u\in B_{p,p,\mathcal{H}}^{{\alpha},\beta}(\X),$
  $$\Big(\int_{\X}|u(g)|^{q}d\nu(g)\Big)^{1/q}\lesssim\|u\|_{B_{p,p,\mathcal{H}}^{{\alpha},\beta}(\X)}.$$

\end{theorem}
\begin{proof}
Suppose that (i) is valid.
For all lower semi-continuous functions $u\in
B_{p,p,\mathcal{H}}^{{\alpha},\beta}(\X),$ it follows from Theorem
\ref{CI} and  inequality (\ref{4.4}) that
\begin{align*}
\int_{\X}|u(g)|^{q}d\nu(g)&=\int_{0}^{\infty}\nu(\{g\in\X:|u(g)|>t\})dt^{q}\\
&\sim\sum_{k\in\mathbb{Z}}2^{kq}\nu(\{g\in\X:|u(g)|>2^{k}\})\\
&\lesssim\bigg(\sum_{k\in\mathbb{Z}}2^{kq}\Big(\nu(\{g\in\X:|u(g)|>2^{k}\})\Big)^{p/q}\bigg)^{q/p}\\
&\lesssim \bigg(\sum_{j\in\mathbb{Z}}2^{kp}\textrm{Cap}_{B_{p,p,\mathcal{H}}^{\alpha,\beta}}(\{g\in\X:|u(g)|>2^{k}\})\bigg)^{q/p}\\
&\lesssim \bigg(\int_{0}^{\infty}\textrm{Cap}_{B_{p,p,\mathcal{H}}^{\alpha,\beta}}(\{g\in\X:|u(g)|>t\})dt^{p}\bigg)^{q/p}\\
&\lesssim \|u\|^{q}_{B_{p,p,\mathcal{H}}^{{\alpha},\beta}(\X)},
\end{align*}
whence (i) derives (ii).

Conversely, if $\mathcal{K}\subset\X$ is compact, then for any $u\in
C_{c}^{\infty}(\X)$ with $u\geq1$ on $\mathcal{K},$
$$\nu(\mathcal{K})\leq\int_{\mathcal{K}}|u(g)|^{q}d\nu(g)\lesssim\|u\|^{q}_{B_{p,p,\mathcal{H}}^{{\alpha},\beta}(\X)}.$$
By Definition \ref{defn2.9} and
 Lemma \ref{EquiC},  we have
$$\nu(\mathcal{K})\lesssim\big(\textrm{Cap}_{B_{p,p,\mathcal{H}}^{\alpha,\beta}}(\mathcal{K})\big)^{q/p}.$$
The above inequality holds for any set $E$ due to the regularity
properties of $\nu$ and the
$B_{p,p,\mathcal{H}}^{\alpha,\beta}$-capacity, which proves (i).
\end{proof}
%
%

\begin{remark}
 Let $(s,\alpha)\in(0,1)\times(0,1)$.   For $p\in(1,\min\{2\alpha,\mathcal{Q}/2\mathbf{s}\})$ with $p=q$ and $1/\alpha>\beta>2s$, via Theorem
  \ref{thm2} and Proposition \ref{boundedness}, we conclude  that
if  the measure $\nu$ satisfies
$$\sup_{E\subset\X}\frac{(\nu(E))^{1/q}}{(Cap_{\mathcal{\dot{W}}^{2\mathbf{s},p}}(E))^{1/p}}<\infty, $$
then
$$\Big(\int_{\X}\big|u(g)\big|^{q}d\nu(g)\Big)^{1/q}\lesssim \|u\|_{{B_{p,p,\mathcal{H}}^{{\alpha},\beta}(\X)}}\ \ \forall\ u\in C_{c}^{\infty}(\X).$$

Moreover, by Remark \ref{ccc}, for
$p\in(1,\min\{2\alpha,\mathcal{Q}/2\mathbf{s}\})$ with  $p=q$ and
$\min\{1/\alpha,\alpha/p\}>\beta>\alpha s$, we also can see that if
the measure $\nu$ satisfies
$$\sup_{E\subset\X}\frac{(\nu(E))^{1/q}}{(Cap_{\mathcal{\dot{W}}^{2\mathbf{s},p}}(E))^{1/p}}<\infty,$$
then
$$\Big(\int_{\X}\big|u(g)\big|^{q}d\nu(g)\Big)^{1/q}\lesssim \|u\|_{N_{p,p}^{\beta}}\ \ \forall\ u\in C_{c}^{\infty}(\X).$$

%

\end{remark}

\begin{theorem}\label{trace1}
Assume that $p\in[1,2\alpha),0<q<p$ and $\beta\in(0,1/\alpha).$ Let
$\nu$ be a nonnegative Radon measure on $\X.$ Then the following two
assertions are equivalent:

  \item [(i)] The function
  $$(0,+\infty)\ni t\mapsto h_{p,\nu}(t):=\inf\{\textrm{Cap}_{B_{p,p,\mathcal{H}}^{\alpha,\beta}}(E):E\subset\X~with~\nu(E)\geq t\}$$ satisfies
  $$\bigg(\int_{0}^{\infty}\big(\frac{t^{p/q}}{h_{p,\nu}(t)}\big)^{q/(p-q)}\frac{dt}{t}\bigg)^{(p-q)/p}=:\|\nu\|_{p,q}<\infty.$$
  \item [(ii)] For all lower semi-continuous functions $u\in B_{p,p,\mathcal{H}}^{{\alpha},\beta}(\X),$
  $$\Big(\int_{\X}|u(g)|^{q}d\nu(g)\Big)^{1/q}\lesssim\|u\|_{B_{p,p,\mathcal{H}}^{{\alpha},\beta}(\X)}.$$

\end{theorem}

\begin{proof}
Assume that (i) holds. For any lower semi-continuous function $u\in
B_{p,p,\mathcal{H}}^{{\alpha},\beta}(\X),$ let
$$E_{t}:=\{g\in\X:|u(g)|>t\}$$ for all $t>0.$ Then
\begin{align*}
\int_{\X}|u(g)|^{q}d\nu(g)&=\sum_{k\in\mathbb{Z}}\int_{2^{k}}^{2^{k+1}}\nu(E_{t})dt^{q}\lesssim\sum_{k\in\mathbb{Z}}2^{kq}\nu(E_{2^{k}})\\
&\sim\sum_{k\in\mathbb{Z}}\sum_{i=k}^{\infty}2^{kq}\nu(E_{2^{i}}\setminus
E_{2^{i+1}})
\lesssim\sum_{i\in\mathbb{Z}}2^{iq}(\nu(E_{2^{i}})-\nu(E_{2^{i+1}})).
\end{align*}
Note that $q<p$ and $q/p+(p-q)/p=1,$ it follows from H\"older's inequality and Theorem \ref{CI} that
\begin{align*}
&\sum_{i\in\mathbb{Z}}2^{iq}(\nu(E_{2^{i}})-\nu(E_{2^{i+1}}))\\
&\leq\bigg(\sum_{i\in\mathbb{Z}}2^{ip}\textrm{Cap}_{B_{p,p,\mathcal{H}}^{\alpha,\beta}}(E_{2^{i}})\bigg)^{q/p}
\bigg(\sum_{i\in\mathbb{Z}}\big(\dfrac{\nu(E_{2^{i}})-\nu(E_{2^{i+1}})}
{(\textrm{Cap}_{B_{p,p,\mathcal{H}}^{\alpha,\beta}}(E_{2^{i}}))^{q/p}}\big)^{p/(p-q)}\bigg)^{(p-q)/p}\\
&\sim\bigg(\int_{0}^{\infty}\textrm{Cap}_{B_{p,p,\mathcal{H}}^{\alpha,\beta}}(E_{t})dt^{p}\bigg)^{q/p}
\bigg(\sum_{i\in\mathbb{Z}}\big(\dfrac{\nu(E_{2^{i}})-\nu(E_{2^{i+1}})}
{(\textrm{Cap}_{B_{p,p,\mathcal{H}}^{\alpha,\beta}}(E_{2^{i}}))^{q/p}}\big)^{p/(p-q)}\bigg)^{(p-q)/p}\\
&\leq\|u\|^{q}_{B_{p,p,\mathcal{H}}^{{\alpha},\beta}(\X)}\bigg(\sum_{i\in\mathbb{Z}}
\big(\dfrac{\nu(E_{2^{i}})-\nu(E_{2^{i+1}})}{(\textrm{Cap}_{B_{p,p,\mathcal{H}}^{\alpha,\beta}}
(E_{2^{i}}))^{q/p}}\big)^{p/(p-q)}\bigg)^{(p-q)/p}.
\end{align*}
Since $u$ is lower semi-continuous, we can see hat every $E_{2^{i}}$
is open. Thus there exists a compact set $\mathcal{K}_{i}\subset
\mathcal{K}_{2^{i}}$ such that
$2^{-1}\nu(E_{2^{i}})\leq\nu(\mathcal{K}_{i})\leq\nu(E_{2^{i}}).$ By
the definition of $h_{p,\nu}(\cdot),$ we have
$$h_{p,\nu}(2^{-1}\nu(E_{2^{i}}))\leq\textrm{Cap}_{B_{p,p,\mathcal{H}}^{\alpha,\beta}}(\mathcal{K}_{i})
\leq\textrm{Cap}_{B_{p,p,\mathcal{H}}^{\alpha,\beta}}(E_{2^{i}}).$$
Combining this with the increasing property of
$h_{p,\nu}(\cdot)$, we conclude
\begin{align*}
\sum_{i\in\mathbb{Z}}\big(\dfrac{\nu(E_{2^{i}})-\nu(E_{2^{i+1}})}{(\textrm{Cap}_{B_{p,p,\mathcal{H}}^{\alpha,\beta}}
(E_{2^{i}}))^{q/p}}\big)^{p/(p-q)}
&\leq\sum_{i\in\mathbb{Z}}\dfrac{(\nu(E_{2^{i}})-\nu(E_{2^{i+1}}))^{p/(p-q)}}{(h_{p,\nu}(2^{-1}\nu(E_{2^{i}})))^{q/(p-q)}}\\
&\leq\sum_{i\in\mathbb{Z}}\dfrac{(\nu(E_{2^{i}}))^{p/(p-q)}-(\nu(E_{2^{i+1}}))^{p/(p-q)}}{(h_{p,\nu}(2^{-1}\nu(E_{2^{i}})))^{q/(p-q)}}\\
&\sim\sum_{i\in\mathbb{Z}}\int_{2^{-1}\nu(E_{2^{i+1}})}^{2^{-1}\nu(E_{2^{i}})}\dfrac{ds^{p/(p-q)}}{(h_{p,\nu}(2^{-1}\nu(E_{2^{i}})))^{q/(p-q)}}\\
&\leq\|\nu\|_{p,q}^{p/(p-q)}.
\end{align*}
Therefore,
$$\int_{\X}|u(g)|^{q}d\nu(g)\lesssim\sum_{i\in\mathbb{Z}}2^{iq}(\nu(E_{2^{i}})-\nu(E_{2^{i+1}}))
\lesssim\|\nu\|_{p,q}^{p/(p-q)}\|u\|^{q}_{B_{p,p,\mathcal{H}}^{{\alpha},\beta}(\X)}.$$
This proves (ii).

Next we prove that (ii) implies (i). For any $t>0,$ we claim that
$$h_{p,\nu}(t)>0.$$ Otherwise, if there exists some $t_{0}>0$ such
that $h_{p,\nu}(t_{0})=0,$ then for any $\epsilon>0$ there exists
some compact set $\mathcal{K}_{\epsilon}\subset\X$ such that
$\nu(\mathcal{K}_{\epsilon})\geq t_{0}$ and
$\textrm{Cap}_{B_{p,p,\mathcal{H}}^{\alpha,\beta}}(\mathcal{K}_{\epsilon})<\epsilon,$
but (ii) implies
$$t_{0}\leq \nu(\mathcal{K}_{\epsilon})\lesssim\inf\{\|u\|^{q}_{B_{p,p,
\mathcal{H}}^{{\alpha},\beta}(\X)}:u\in C_{c}^{\infty}(\X)\ \textrm{and}\ u\geq1\
\textrm{on}\ \mathcal{K}_{\epsilon}\}\sim(\textrm{Cap}_{B_{p,p,\mathcal{H}}^{\alpha,\beta}}^{*}(\mathcal{K}_{\epsilon}))^{q/p}
\lesssim\epsilon^{q/p}.$$ Letting $\epsilon\rightarrow0$ yields
$t_{0}=0$ against $t_{0}>0.$ Thus, the claim is valid.

Fix $\epsilon>0.$ By the definition of $h_{p,\nu},$ there exists a
compact set $\mathcal{K}_{j}\, (j\in\mathbb{Z})$ satisfying
$$\nu(\mathcal{K}_{j})\geq  2^{j} \ \textrm{and} \ \textrm{Cap}_{B_{p,p,\mathcal{H}}^{\alpha,\beta}}(\mathcal{K}_{j})-\epsilon h_{p,\nu}
(2^{j})\leq h_{p,\nu}(2^{j})\leq
 \textrm{Cap}_{B_{p,p,\mathcal{H}}^{\alpha,\beta}}(\mathcal{K}_{j}).$$
From Lemma \ref{EquiC}, there exists $u_{j}\in C_{c}^{\infty}(\X)$ such that $u_{j}\geq1$ on $\mathcal{K}_{j}$ and
$$\|u_{j}\|^{p}_{B_{p,p,\mathcal{H}}^{{\alpha},\beta}(\X)}-\epsilon h_{p,\nu}(2^{j})\leq
\textrm{Cap}_{B_{p,p,\mathcal{H}}^{\alpha,\beta}}(\mathcal{K}_{j})=\textrm{Cap}_{B_{p,p,\mathcal{H}}^{\alpha,\beta}}^{*}(\mathcal{K}_{j})
\leq\|u_{j}\|^{p}_{B_{p,p,\mathcal{H}}^{{\alpha},\beta}(\X)}.$$

For any $i,k\in\mathbb{Z}$ with $i\leq k,$ let
$$u_{i,k}:=\max_{i\leq j\leq k}\bigg(\dfrac{2^{j}}{h_{p,\nu}(2^{j})}\bigg)^{1/(p-q)}u_{j}.$$
We deduce from Lemma \ref{max} that $u_{i,k}\in
B_{p,p,\mathcal{H}}^{{\alpha},\beta}(\X)$ and
\begin{eqnarray}\label{Ineq1}
\|u_{i,k}\|^{p}_{B_{p,p,\mathcal{H}}^{{\alpha},\beta}(\X)}&\leq& \sum_{i\leq j\leq k}\bigg(\dfrac{2^{j}}{h_{p,\nu}(2^{j})}\bigg)^{p/(p-q)}\|u_{j}\|^{p}_{B_{p,p,\mathcal{H}}^{{\alpha},\beta}(\X)}\nonumber\\
&\leq&(1+2\epsilon)\sum_{i\leq j\leq
k}\dfrac{2^{jp/(p-q)}}{(h_{p,\nu}(2^{j}))^{q/(p-q)}}.
\end{eqnarray}

Note that for $i\leq j\leq k$ and $g\in \mathcal{K}_{j},$ we have
$$u_{i,k}(g)\geq\big(\frac{2^{j}}{h_{p,\nu}(2^{j})}\big)^{1/(p-q)}.$$
Then it follows from (ii) and the nonincreasing rearrangement of
$u_{i,k}$ that
\begin{align*}
\|u_{i,k}\|^{q}_{B_{p,p,\mathcal{H}}^{{\alpha},\beta}(\X)}&\gtrsim\int_{\X}|u_{i,k}(g)|^{q}d\nu(g)\\
&\sim\int_{0}^{\infty}\big(\inf\{\lambda>0:\nu(\{g\in\X:|u_{i,k}(g)|>\lambda\})\leq s\}\big)^{q}ds\\
&\sim\sum_{j\in\mathbb{Z}}2^{j}\big(\inf\{\lambda:\nu(\{g\in\X:|u_{i,k}(g)|>\lambda\})\leq 2^{j}\}\big)^{q}\\
&\geq\sum_{i\leq j\leq k}2^{j}\bigg(\frac{2^{j}}{h_{p,\nu}(2^{j})}\bigg)^{q/(p-q)}\\
&\sim\sum_{i\leq j\leq
k}\dfrac{2^{jp/(p-q)}}{(h_{p,\nu}(2^{j}))^{q/(p-q)}}.
\end{align*}
Inequality (\ref{Ineq1}) gives
$$\sum_{i\leq j\leq k}\dfrac{2^{jp/(p-q)}}{(h_{p,\nu}(2^{j}))^{q/(p-q)}}\lesssim\bigg((1+2\epsilon)\sum_{i\leq j\leq k}\dfrac{2^{jp/(p-q)}}
{(h_{p,\nu}(2^{j}))^{q/(p-q)}}\bigg)^{q/p}.$$
Letting $\epsilon\rightarrow0,$ we have
$$\bigg(\sum_{i\leq j\leq k}\dfrac{2^{jp/(p-q)}}{(h_{p,\nu}(2^{j}))^{q/(p-q)}}\bigg)^{(p-q)/p}\lesssim1.$$
Finally, letting $i\rightarrow-\infty$ and $k\rightarrow\infty,$ we obtain
\begin{align*}
\|\nu\|_{p,q}&=\bigg(\int_{0}^{\infty}\bigg(\dfrac{t^{p/q}}{h_{p,\nu}(t)}\bigg)^{q/(p-q)}\dfrac{dt}{t}\bigg)^{(p-q)/p}\\
&\sim\bigg(\sum_{j=-\infty}^{\infty}\dfrac{2^{jp/(p-q)}}{(h_{p,\nu}(2^{j}))^{q/(p-q)}}\bigg)^{(p-q)/p}\lesssim1.
\end{align*}
This completes the proof of Theorem \ref{trace1}.
\end{proof}

%
%
%

\begin{remark}
For Besov capacities, since we do not obtain an strong-type capacitary inequality similar to the one in Proposition \ref{strong} (ii),
 we cannot discuss here the extension theorem for Besov spaces.
But using Theorem  \ref{thm1}, Proposition \ref{boundedness} and
Remark \ref{ccc}, we obtain the following results: Let
$(s,\alpha)\in(0,1)\times(0,1)$. For
$p\in(1,\min\{2\alpha,\mathcal{Q}/2\mathbf{s}\})$ with $p=q$ and
$1/\alpha>\beta>2s$, if one of the following conditions is
satisfied:
\item{\rm (i)}
$$\sup_{\lambda>0}\lambda(\mu(\{(g,t)\in\X\times\mathbb{R}_{+}:|H_{\alpha,t^{2\alpha}}u(g)|>\lambda\}))^{1/q}\lesssim \|u\|_{\mathcal{\dot{W}}^{2\mathbf{s},p}(\X)};$$
\item{\rm (ii)} The measure $\mu$ satisfies
$$\sup_{t>0}\frac{t^{1/q}}{(c_{p}(\mu;t))^{1/p}}<\infty;$$
\item{\rm (iii)} The measure $\mu$ satisfies
$$(\mu(T(O)))^{p/q}\lesssim Cap_{\mathcal{\dot{W}}^{2\mathbf{s},p}}(O) $$
  for any bounded open set $O\subseteq \X,$ then
 we get
$$\Bigg(\int_{\X\times\mathbb{R}_{+}}\big|H_{\alpha,t^{2\alpha}}u(g)\big|^{q}d\mu(g,t)\Bigg)^{1/q}\lesssim
 \|u\|_{{B_{p,p,\mathcal{H}}^{{\alpha},\beta}(\X)}}\ \ \forall\ u\in C_{c}^{\infty}(\X).$$

Moreover, if $\beta$ satisfies
$\min\{1/\alpha,\alpha/p\}>\beta>\alpha s$, then
$N_{B_{p,p,\mathcal{H}}^{{\alpha},\beta}}(\cdot)$ above can be
replaced by $N_{p,p}^{\beta}(\cdot)$.

%

\end{remark}

\subsection*{Acknowledgements}

 {{ P.T. Li was supported by the National Natural Science Foundation of China (No. 12071272) and Shandong
Natural Science Foundation of China (No. ZR2020MA004).}  Y. Liu was
supported by the National Natural Science Foundation of China (No.
12271042, No. 12471089)  and Beijing Natural Science Foundation of
China (No. 1232023).}

\bibliographystyle{amsplain}

\hspace{ -0.5cm}{\bf Address: }

          \flushleft Zhiyong Wang\\
          School of Mathematics and Physics\\
          University of Science and Technology Beijing\\
          Beijing 100083, China\\
          E-mail: zywang199703@163.com

          \flushleft Pengtao Li\\
          School of   Mathematics and Statistics \\
           Qingdao University \\
          Qingdao 266071, China\\
          E-mail address: ptli@qdu.edu.cn

          \flushleft Yu Liu\\
          School of Mathematics and Physics\\
          University of Science and Technology Beijing\\
          Beijing 100083, China\\
          E-mail: liuyu75@pku.org.cn

\end{document}